\documentclass[12pt]{amsart}
\usepackage{amsmath,amssymb,amsthm}
\usepackage{graphicx}
\usepackage{xcolor}
\usepackage{enumitem}
\usepackage{tikz}
\usetikzlibrary{calc,arrows.meta,decorations.text} 
\allowdisplaybreaks[4]
\usepackage{aliascnt}
\usepackage[margin=1.3in]{geometry}
\usepackage[
colorlinks = true,
linkcolor  = blue,
citecolor  = blue,
urlcolor   = blue,
pagebackref = true
]{hyperref}

\newtheorem*{rmks}{Remarks}

\newtheorem{thm}{Theorem}[section]

\newaliascnt{lma}{thm}
\newtheorem{lma}[lma]{Lemma}
\aliascntresetthe{lma}

\newaliascnt{prop}{thm}
\newtheorem{prop}[prop]{Proposition}
\aliascntresetthe{prop}

\newaliascnt{cor}{thm}

\aliascntresetthe{cor}

\newaliascnt{conj}{thm}

\aliascntresetthe{conj}

\theoremstyle{remark}
\newaliascnt{rmk}{thm}
\newtheorem{rmk}[rmk]{Remark}
\aliascntresetthe{rmk}

\theoremstyle{definition}
\newaliascnt{defn}{thm}
\newtheorem{defn}[defn]{Definition}
\aliascntresetthe{defn}

\theoremstyle{plain}

\numberwithin{equation}{section}
\newcommand{\R}{\mathbb{R}}

\begin{document}
\title[Symmetry in a sub-spherical sector]{Symmetry properties for positive solutions of mixed boundary value problems in a sub-spherical sector}
\author{Ruofei Yao}
\address{School of Mathematics, South China University of Technology, Guangzhou, 510641, P.R. China}
\email{yaorf5812@126.com; ruofeiyaopde@gmail.com}
\date{\today}
\begin{abstract}
In this paper, we investigate the symmetry properties of positive solutions $u$ to a semilinear elliptic equation under mixed Dirichlet-Neumann boundary conditions in symmetric domains. First, we establish a maximum principle tailored to mixed-boundary problems in domains of either small volume or narrow width, thereby enabling the application of the moving plane method. Secondly, in contrast to the purely Dirichlet case, a key challenge is to establish the non-vanishing of the tangential derivative of $u$ along the Neumann boundary. To address this, we employ local analysis techniques of angular derivatives, as introduced by Hartman and Wintner [Amer. J. Math., 1953]. Thirdly, we identify the signs of directional derivatives of $u$ along sections of the moving line. Using a planar sub-spherical sector as an example, we illustrate how these new innovative techniques and the moving plane method can be combined to derive symmetry and monotonicity results, particularly when the amplitude is less than or equal to $2\pi/3$. 
\end{abstract}

\maketitle




\section{Introduction}

\subsection{Background and motivation}
This article is devoted to the qualitative properties of positive solutions to semilinear elliptic partial differential equations with mixed Dirichlet-Neumann boundary conditions: 
\begin{equation} \label{Yao0102}
\begin{cases}
\Delta {u} + {f}({u}) = 0 & \text{in } {\Omega}, \\
{u} = 0 & \text{on } \Gamma_{D}, \\
\frac{\partial {u}}{\partial \nu} = 0 & \text{on } \Gamma_{N}, 
\end{cases}
\end{equation}
where $\Omega$ is a bounded Lipschitz domain, $\Gamma_{D}$ is a portion of the boundary $\partial \Omega$, $\Gamma_{N} = \partial \Omega \setminus \Gamma_{D}$, and $\nu$ denotes the unit outward normal to $\partial \Omega$. 

Research on the qualitative properties of solutions is a central topic in the theory of partial differential equations. A classical and influential approach is the method of moving planes, introduced by~\cite{Ale56, Ser71} and further developed by Gidas, Ni, and Nirenberg~\cite{GNN79}. In particular, \cite{GNN79} established radial symmetry for positive solutions of semilinear elliptic equations under purely Dirichlet boundary conditions (i.e., $\Gamma_{D}=\partial\Omega$). Later, Berestycki and Nirenberg~\cite{BN91} refined these results using a maximum principle in domains of small volume and introduced the sliding method to obtain monotonicity properties. These developments sparked extensive research on symmetry and monotonicity; see, for instance, \cite{CGS89, CL91, LZ95, BCN97a, BCN97b, FMR16, GM18, EFMS22}. Motivated by these developments, it is natural to ask whether, when both the domain and the boundary decomposition are symmetric, positive solutions of \eqref{Yao0102} inherit the same symmetry, as raised in~\cite{Ni11}. In general, however, such a conclusion fails once Neumann boundary conditions are present; see~\cite{NT91, GG98, dPFW00, GWW00, Lin01} and the references therein. Establishing analogous qualitative properties becomes even more subtle for sign-changing solutions. For instance, the monotonicity of the second Neumann eigenfunction on certain planar domains has been investigated in~\cite{JN00, AB04, Siu15, CWY26, CGY26}. 

Researchers have devoted increasing attention to qualitative properties (in particular, symmetry and monotonicity) for elliptic problems with mixed boundary conditions, where Dirichlet and Neumann conditions coexist on $\partial\Omega$. In particular, for spherical sectors, Berestycki and Pacella~\cite{BP89} established radial symmetry results when the opening angle is at most $\pi$, while Zhu~\cite{Zhu01} extended these results to certain supercritical nonlinearities when the opening angle exceeds $\pi$. Building on~\cite{BP89}, further radial symmetry results for infinite sectorial cones were obtained in~\cite{dCP89, CW92, SGC97}. The author has contributed to this area by establishing symmetry and monotonicity properties of positive solutions to elliptic equations with mixed boundary conditions in various domains; see~\cite{CY18, YCL18, CLY21, YCG21, CWY23, LrY24, MY26}. In recent years, analogous questions have also been investigated for mixed Dirichlet--Neumann eigenfunctions; see, e.g., \cite{AR25, LrY24, Hat24, Hat25a, Hat25b, Hat25d}. For broader discussions of qualitative properties in nonlinear mixed boundary value problems, we refer to~\cite{Ter95, DR04, WH07, DP19, DP20, GG23} and the references therein.

This paper is part of a series of works~\cite{CY18, YCL18, CLY21, YCG21, MY26} investigating the symmetry and monotonicity of solutions to semilinear elliptic equations with mixed boundary conditions, in bounded symmetric domains. In particular, the symmetry result for super-spherical sectors was established in~\cite{YCG21}, while a partial result for sub-spherical sectors was obtained in~\cite{CLY21}. We now extend this investigation by studying the symmetry and monotonicity of solutions to \eqref{Yao0102} in a planar sub-spherical sector, under a broad range of geometric assumptions as described in \eqref{Yao0106}. 

\subsection{Problem setting and main result}
We begin by introducing the notation and the geometric setting. 
For $\alpha \in (0, \pi]$ and $\beta \in (0, 2\pi)$, let $\Sigma = \Sigma_{\alpha, \beta}$ denote a planar domain bounded by an arc $\Gamma_D$ and a portion of the boundary of a sector $\partial \mathcal{C}$, where $\Gamma_{D} = \{(\cos\theta, \sin\theta) \in \R^{2}: |\theta| \leq \alpha/2\}$ is a portion of the unit circle with opening angle $\alpha$, and $\mathcal{C} = \mathcal{C}_{\alpha, \beta}$ denotes the open sector with vertex ${V} = ({a}, 0)$ and opening $\beta$, that is, 
\begin{equation*}
\mathcal{C}_{\alpha, \beta} = \big\{ ({x}_{1}, {x}_{2}) \in \R^{2}: \; {x}_{1} - {a} > |{x}_{2}|\cot\frac{\beta}{2} \big\}. 
\end{equation*}
Here, the constant ${a}$, which depends on $\alpha$ and $\beta$, is given by
\begin{equation} \label{Yao0103b}
{a} = \cos\frac{\alpha}{2} - \sin\frac{\alpha}{2}\cot\frac{\beta}{2} = - \sin\frac{\alpha - \beta }{2} \csc\frac{\beta}{2}. 
\end{equation}
This relation guarantees that both endpoints ${P}_{\pm} = (\cos(\alpha/2), \pm\sin(\alpha/2))$ of the arc $\Gamma_{D}$ lie on the boundary $\partial \mathcal{C}$. Thus, $\Sigma = \Sigma_{\alpha, \beta}$ can be characterized as 
\begin{equation*}
\Sigma_{\alpha, \beta} = \big\{ ({x}_{1}, {x}_{2}) \in \R^{2}: \; {a} + |{x}_{2}|\cot\tfrac{\beta}{2} < {x}_{1} < \textstyle\sqrt{1 - |{x}_{2}|^{2}} \big\}. 
\end{equation*}
For ${a} \in [ - 1, 1)$, it is clear that $\Sigma_{\alpha, \beta}$ is the intersection of the sector $\mathcal{C}_{\alpha, \beta}$ with the unit ball ${B} = \{{x} \in \R^{2}: |{x}| < 1\}$. 
We refer to $\Sigma_{\alpha, \beta}$ as a (standard) spherical sector if ${a} = 0$ (i.e., $\beta = \alpha$); as a sub-spherical sector if ${a} < 0$ (i.e., $\beta < \alpha$); and as a super-spherical sector if ${a} \in (0, 1)$. See \autoref{Fig01Sector}. 

\begin{figure}[htp] \centering 
\begin{tikzpicture}[scale = 2] 
\pgfmathsetmacro\ALPHA{65.04}; \pgfmathsetmacro\radius{1.00}; 
\pgfmathsetmacro\Move{ - 1.8}; \pgfmathsetmacro\BETA{28.29}; 
\pgfmathsetmacro\aaa{\radius*sin(\BETA - \ALPHA)/sin(\BETA)}; 
\pgfmathsetmacro\xA{\radius*cos(\ALPHA)}; \pgfmathsetmacro\yA{\radius*sin(\ALPHA)}; 
\fill[fill = gray, fill opacity = 0.3, draw = black, thick] ({\Move + \xA}, {0 - \yA}) arc ( - \ALPHA : \ALPHA : \radius) -- (\Move + \aaa, 0) -- cycle; 
\draw[thin, dotted] ({\Move + \xA}, {0 + \yA}) -- (\Move + 0, 0) node[left] {\footnotesize ${O}$} -- ({\Move + \xA}, {0 - \yA}); 
\node[left] at (\Move + \aaa, 0) {\footnotesize ${V}$}; 
\node[right] at ({\Move + \radius*cos(\ALPHA/3)}, {\radius*sin(\ALPHA/3)}) {\small $\Gamma_{D}$}; 
\node[below] at ({\Move + (\aaa + \xA)/2}, { - \yA/2}) {\small $\Gamma_{N}$}; 
\node[above] at ({\Move + (\aaa + \xA)/2}, {\yA/2}) {\small $\Gamma_{N}$}; 
\node[left] at (\Move + \xA, \yA) {\footnotesize ${P}_{ + }$}; \node[left] at (\Move + \xA, - \yA) {\footnotesize ${P}_{ - }$}; 
\node[right] at (\Move + \aaa + \radius*0.1, 0) {\small $\beta$}; 
\node[right] at (\Move + \radius*0.04, 0) {\small $\alpha$}; 
\pgfmathsetmacro\Move{0}; \pgfmathsetmacro\BETA{\ALPHA}; 
\pgfmathsetmacro\aaa{\radius*sin(\BETA - \ALPHA)/sin(\BETA)}; 
\pgfmathsetmacro\xA{\radius*cos(\ALPHA)}; \pgfmathsetmacro\yA{\radius*sin(\ALPHA)}; 
\fill[fill = gray, fill opacity = 0.3, draw = black, thick] ({\Move + \xA}, {0 - \yA}) arc ( - \ALPHA : \ALPHA : \radius) -- (\Move + \aaa, 0) -- cycle; 
\node[left] at (\Move + 0, 0) {\footnotesize ${O}({V})$}; 
\pgfmathsetmacro\Move{1.4}; \pgfmathsetmacro\BETA{81.29}; 
\pgfmathsetmacro\aaa{\radius*sin(\BETA - \ALPHA)/sin(\BETA)}; 
\pgfmathsetmacro\xA{\radius*cos(\ALPHA)}; \pgfmathsetmacro\yA{\radius*sin(\ALPHA)}; 
\fill[fill = gray, fill opacity = 0.3, draw = black, thick] ({\Move + \xA}, {0 - \yA}) arc ( - \ALPHA : \ALPHA : \radius) -- (\Move + \aaa, 0) -- cycle; 
\draw[thin, dotted] ({\Move + \xA}, {0 + \yA}) -- (\Move + 0, 0) node[left] {\footnotesize ${O}$} -- ({\Move + \xA}, {0 - \yA}); 
\node[left] at (\Move + \aaa, 0) {\footnotesize ${V}$}; 
\end{tikzpicture}
\caption[Sub-spherical sector]{Sub-spherical sector (left), spherical sector (middle) and super-spherical sector (right).}
\label{Fig01Sector}
\end{figure}
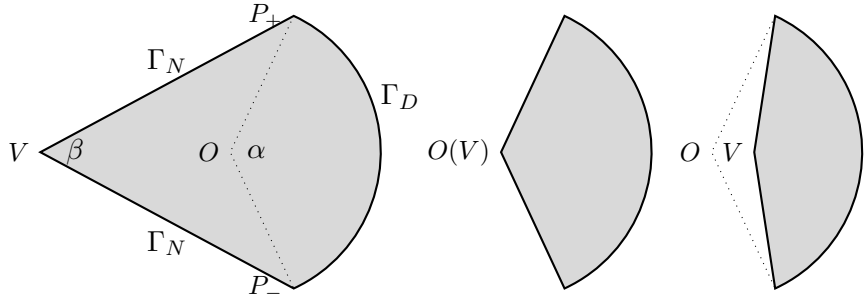

We consider the following equation with mixed boundary conditions: 
\begin{equation} \label{Yao0105}
\begin{cases}
\Delta {u} + {f}({u}) = 0 & \text{in } \Sigma, \\
{u} > 0 & \text{in } \Sigma, \\
\frac{\partial {u}}{\partial \nu} = 0 & \text{on } \Gamma_{N}, \\
{u} = 0 & \text{on } \Gamma_{D}
\end{cases}
\end{equation}
where $\nu$ denotes the unit outward normal to $\partial \Sigma$. Here, we set $\Gamma_{D} = \{(\cos\theta, \sin\theta) \in \R^{2}: |\theta| \leq \alpha/2\}$ as the spherical part, and $\Gamma_{N} = \partial \Sigma \setminus \Gamma_{D}$ as the union of two straight segments. The nonlinearity ${f}$ is assumed to be locally Lipschitz continuous (${f} \in \mathrm{Lip}_{\mathrm{loc}}(\R)$), so that the method of moving planes applies. A solution ${u}$ to \eqref{Yao0105} is always understood in the classical sense, namely, ${u} \in {C}^{2}(\Sigma) \cap {C}(\overline{\Sigma}) \cap {C}^{1}(\Sigma \cup \Gamma_{N})$. By standard elliptic theory, it follows that ${u}$ is also ${C}^{2}$ at the smooth boundary points. Therefore, throughout this paper, we always have ${u} \in {C}(\overline{\Sigma}) \cap {C}^{2}(\overline{\Sigma} \setminus \{{V}, {P}_{\pm}\})$. 

The main result of this paper is as follows: 

\begin{thm} \label{thm101main}
Let ${f} \in \mathrm{Lip}_{\mathrm{loc}}(\R)$, and let $\Sigma$, $\Gamma_{D}$, and $\Gamma_{N}$ be as defined above. Suppose that $0 < \beta < \alpha \leq \pi$ and
\begin{equation} \label{Yao0106}
\beta \leq \frac{2\pi}{3}. 
\end{equation}
Then any solution ${u}$ of \eqref{Yao0105} satisfies the following properties: 
\begin{enumerate}[label = {\rm(\roman*)}, start = 1]
\item \label{Yaoit11a}
${u}$ is symmetric with respect to the line ${x}_{2} = 0$; 
\item \label{Yaoit11b}
${u}$ is monotone in ${x}_{2}$ in a half-domain, that is, ${x}_{2} {u}_{{x}_{2}} < 0$ for ${x} \in \Sigma$ with ${x}_{2} \neq 0$; 
\item \label{Yaoit11c}
${u}$ is monotone in ${x}_{1}$, that is, ${u}_{{x}_{1}} < 0$ in $\Sigma$. 
\end{enumerate}
\end{thm}

\begin{rmk}
In the previous paper~\cite{CLY21}, two results were established: 
\begin{enumerate}[label = \rm(\arabic*)]
\item
\autoref{thm101main} holds under the assumptions $0 < \beta < \alpha \leq \pi$ and
\begin{equation} \label{Yao0107}
\alpha + \beta \leq \pi \quad \quad \text{and }\quad \quad \alpha/3 \leq \beta \leq \pi/3. 
\end{equation}
\item
The properties \ref{Yaoit11a}, \ref{Yaoit11b}, and \ref{Yaoit11c} are equivalent without the condition \eqref{Yao0107}. 
\end{enumerate}
\end{rmk}

\subsection{Ideas of the proof}
We briefly explain the main ideas of the proof. 
The mixed Dirichlet--Neumann boundary conditions force the expected symmetry to be an even symmetry (rather than a spherical symmetry), since the center of the ball does not coincide with the vertex of the sector $\mathcal{C}$. 
The proof of \autoref{thm101main} relies heavily on the method of moving planes~\cite{GNN79, BP89, BN91}. As the moving plane method crucially depends on the monotonicity (such as property \ref{Yaoit11c}) of the positive solution, the condition $\alpha \leq \pi$ in \autoref{thm101main} is required. Condition \eqref{Yao0106} is a technical assumption; for general $\beta$ (without \eqref{Yao0106}), it remains unknown whether the symmetry result holds. 

On the Dirichlet boundary, the direction of the gradient of a positive solution ${u}$ is known, as it points along the interior normal to the Dirichlet boundary. However, on the Neumann boundary, the direction of $\nabla {u}$ is unknown. Thus, a primary challenge is to establish the desired monotonicity property on $\Gamma_{N}$, namely, 
\begin{equation} \label{Yao0109a}
(\cos\tfrac{\beta}{2}, - \sin\tfrac{\beta}{2}) \cdot \nabla {u} < 0 \quad \text{on } \Gamma_{N}^{ - }, \quad
(\cos\tfrac{\beta}{2}, \sin\tfrac{\beta}{2}) \cdot \nabla {u} < 0 \quad \text{on } \Gamma_{N}^{ + }, 
\end{equation}
where $\Gamma_{N}^{ - } = \{{x} \in \Gamma_{N}: {x}_{2} < 0\}$ and $\Gamma_{N}^{ + } = \{{x} \in \Gamma_{N}: {x}_{2} > 0\}$ are the open line segments. 
In the previous paper~\cite{CLY21}, the key steps involved demonstrating that
\begin{equation} \label{Yao0109b}
{w}^{\lambda, \pi/2} < 0 \quad \text{in } {D}_{\lambda, \pi/2}
\end{equation}
and
\begin{equation} \label{Yao0109c}
{u}_{{x}_{1}}\cos\frac{\beta}{2} - {u}_{{x}_{2}}\sin\frac{\beta}{2} < 0 \quad \text{in } {T}_{\lambda, \pi/2} \cap \Sigma
\end{equation}
hold for all $\lambda > 0$, where ${T}_{\lambda, \pi/2}$ is the moving line perpendicular to $\Gamma_{N}^{ - }$, and ${w}^{\lambda, \pi/2}$ is the difference function between ${u}$ and its reflection; see \eqref{Yao0304a} and \eqref{Yao0305b} for details. However, in the case $\alpha + \beta > \pi$, it is known from the monotonicity near the Dirichlet boundary $\Gamma_{D}$ (see, e.g., \autoref{lma206GNN}) that \eqref{Yao0109c}, and hence \eqref{Yao0109b}, fails for some positive $\lambda$. On the other hand, to apply the moving plane method, it is also necessary to obtain a priori information about the boundary condition for the difference function ${w}^{\lambda, \vartheta}$ on the portion of the boundary corresponding to $\Gamma_{N}$. Therefore, to establish the required symmetry and monotonicity properties, new contributions are needed to verify both \eqref{Yao0109a} and the appropriate boundary conditions for the difference function ${w}^{\lambda, \vartheta}$. 

In~\cite{BP89, CLY21, YCG21}, the sign of the tangential derivative of ${u}$ along the Neumann boundary is obtained by proving \eqref{Yao0109b} and applying Serrin's boundary lemma. In this paper, to derive \eqref{Yao0109a}, we employ local analytical techniques for certain angular derivatives, developed by Hartman and Wintner~\cite{HW53}. The details of this approach are presented in \autoref{lma307} below. By leveraging the strict monotonicity along the Neumann boundary, we are able to establish the symmetry result when $\pi/2 \leq \beta \leq 2\pi/3$, as detailed in \autoref{Sec04Large}. 

The condition $\beta \leq \pi/3$ plays a crucial role in~\cite{CLY21}, and proving \eqref{Yao0109b} for both large and small values of $\lambda$ is relatively straightforward. However, bridging the gap and addressing the intermediate range of $\lambda$ requires a careful use of a continuous family of moving planes with varying directions; see \autoref{Sec05Small}. 

For the case $\pi/3 < \beta < \pi/2$, the main difficulty lies in establishing monotonicity on the moving line passing through the upper mixed boundary point. To overcome this, we extend the solution to a double domain via mirror reflection along the lower Neumann boundary, and utilize a monotonicity property on a portion of the moving line (see \autoref{lma606}). 

\subsection{Organization of the paper}
We conclude this introduction by describing the organization of the paper.
In \autoref{Sec02Pre}, we present some preliminaries, including the Sobolev inequality in infinite sectors, the maximum principle for mixed boundary value problems in domains of small volume, the structure of nodal lines, and certain geometric angle inequalities. In \autoref{Sec03MMP}, we define the notions of moving lines and moving domains, and show that the moving plane method can be carried out up to a maximal position, namely for all $\lambda \geq \lambda_{\sharp}$. The detailed proofs of \autoref{thm101main} are then divided according to the range of $\beta$ and presented in the following sections: \autoref{Sec04Large} deals with the case $\beta \in [\pi/2, 2\pi/3]$; \autoref{Sec05Small} covers the case $\beta \in (0, \pi/3]$; and finally, the proof for $\beta \in (\pi/3, \pi/2)$ is given in \autoref{Sec05Mid}. 


\section{Some preliminaries} \label{Sec02Pre}

\subsection{The maximum principle with mixed boundary conditions}

The method of moving planes often relies on a clever use of the maximum principle, a fundamental tool in the study of elliptic and parabolic partial differential equations~\cite{PW67, GT83}. In the present context, we establish a new version of the maximum principle specifically tailored for mixed boundary problems in sectorial-like domains. These novel maximum principles significantly enhance the efficacy of the moving plane method for this problem. 

Consider the linear equation with mixed boundary conditions: 
\begin{equation}
\begin{cases}
\mathcal{L}[{u}] = \Delta {u} + {c}({x}){u} = {f} & \text{in } \Omega, \\
\mathcal{B}[{u}] = {u} = {g} & \text{on } \Gamma_{0}, \\
\mathcal{B}[{u}] = \nabla {u} \cdot \gamma + \beta {u} = {h} & \text{on } \Gamma_{1}, 
\end{cases}
\end{equation}
where $\Omega \subset \R^{n}$ is a bounded Lipschitz domain whose boundary is partitioned into two disjoint subsets $\Gamma_{0}$ and $\Gamma_{1}$. We assume that the vector-valued function $\gamma$ satisfies $|\gamma| = 1$ and $\gamma \cdot \nu > 0$ on $\Gamma_{1}$, where $\nu$ is the outward unit normal vector. Furthermore, $\beta \geq 0$ on $\Gamma_{1}$ and ${c}$ is bounded in $\Omega$: 
\begin{equation*}
|{c}({x})| < {c}_{0}, \quad {x} \in \Omega \quad \text{for some } {c}_{0} \in \R^{ + }. 
\end{equation*}
As usual, we say that the maximum principle holds for $(\mathcal{L}, \mathcal{B})$ in $\Omega$ if
\begin{equation*}
\text{${f} \leq 0$ in $\Omega$, ${g} \geq 0$ on $\Gamma_{0}$, ${h} \geq 0$ on $\Gamma_{1}$}
\end{equation*}
implies that ${u} \geq 0$ in $\Omega$. 
When $\Gamma_{1}$ is empty, this reduces to the usual maximum principle~\cite{BN91}. Several well-known sufficient conditions for the maximum principle include~\cite{PW67, GT83, BN91, BNV94, HL11}: 
\begin{enumerate}[label = {\rm(\arabic*)}]
\item
\textbf{Nonpositive Coefficient: }
${c} \leq 0$ in $\Omega$. 
\item
\textbf{Existence of a Barrier Function: }
There exists a positive continuous function ${g}$ on $\overline{\Omega}$ such that ${g} \in W_{loc}^{2, \infty}(\Omega)$ and $\mathcal{L}[{g}] \leq 0$ in $\Omega$. 
\item
\textbf{Narrow Domains: }
$\Omega$ lies in a narrow band, $\Omega \subset \{0 < ({x} - {x}_{0}) \cdot \mathbf{e} < \eta\}$ for some ${x}_{0} \in \R^{n}$, $|\mathbf{e}| = 1$, where $\eta > 0$ is a constant depending only on ${c}_{0}$. 
\item
\textbf{Small Volume: }
The measure $|\Omega|$ is less than $\eta$, where $\eta > 0$ depends only on ${c}_{0}$. 
\item
\textbf{Positive Principal Eigenvalue: } $\lambda_{1}(\mathcal{L}, \Omega) > 0$; this is a necessary and sufficient condition. Here, the first (principal) eigenvalue $\lambda_{1}(\mathcal{L}, \Omega)$ is defined by
\begin{equation}
\lambda_{1}(\mathcal{L}, \Omega) = \sup \left\{ \lambda \in \R: \; \exists \; \phi > 0 \text{ in } \Omega \text{ satisfying } (\mathcal{L} + \lambda)\phi \leq 0 \text{ in } \Omega \right\}
\end{equation}
where $\phi \in W_{loc}^{2, n}(\Omega) \cap {C}(\overline{\Omega})$. 
\end{enumerate}
The maximum principle can be generalized to unbounded domains and more general operators; see~\cite{BR15, CLL17, DSV17, CL18, WC20, Nor21, DQ23} and the references therein. 

When $\Gamma_{1}$ is non-empty, Zhu~\cite{Zhu01} established the maximum principle for a narrow annulus, and Damascelli and Pacella~\cite{DP19} employed a different type of maximum principle in small volume domains to address nonlinear boundary problems. The second and fourth sufficient conditions were extended in~\cite{YCL18} to mixed boundary problems, while further sufficient conditions for narrow domains were given in~\cite{CLY21, YCG21, LrY24, MY26}. Since the mixed boundary problem depends heavily on the geometric shape of the oblique derivative boundary $\Gamma_{1}$, finding a good sufficient condition analogous to the narrow or small volume domain cases remains an open problem. 

\begin{lma} \label{lma201MP}
Let $\Omega \subset \R^{n}$ be a bounded Lipschitz domain with boundary partitioned as $\partial \Omega = \Gamma_{0} \cup \Gamma_{1}$, where $\Gamma_{0}$ is a relatively closed subset. Suppose that ${w}$ satisfies 
\begin{equation} \label{Yao0203w}
\begin{cases}
\Delta {w} + {c}({x}) {w} \geq 0 & \text{in } \Omega, \\
{w} \leq 0 & \text{on } \Gamma_{0}, \\
\frac{\partial {w}}{\partial \nu} \leq 0 & \text{on } \Gamma_{1}, 
\end{cases}
\end{equation}
where $\nu$ is the outward unit normal vector on $\Gamma_{1}$, and 
\begin{equation} \label{Yao0203c}
{c} \in {L}^{\infty}(\Omega) \quad \text{and} \quad \|{c}^{ + }\|_{{L}^{\infty}(\Omega)} < {c}_{0} \quad \text{for some } {c}_{0} \in \R^{ + }. 
\end{equation}
If $\Omega$ and $\Gamma_{1}$ satisfy the following conditions: 
\begin{enumerate}[label = \rm(\arabic*)]
\item
$\Omega \cup \Gamma_{1} \subset \{{x} \in \R^{n}: \; 0 < {x}_{1}\}$; 
\item
$\nu \cdot {e}_{1} \geq 0$ on $\Gamma_{1}$, where ${e}_{1} = (1, 0, \ldots, 0)$ is the unit vector in $\R^{n}$; 
\item
$\Omega \subset \{{x}: \; {x} \cdot {e}_{1} < \eta\}$ where $\eta = \pi/(2\sqrt{{c}_{0}})$. 
\end{enumerate}
Then ${w} \leq 0$ in $\Omega$. 
\end{lma}

\begin{proof}
We construct a positive upper solution associated to \eqref{Yao0203w}: 
\begin{equation*}
{g}({x}) = \sin \frac{\pi {x}_{1}}{2\eta}. 
\end{equation*}
Then ${g}$ is positive in $\overline{\Omega} \setminus \{{x}_{1} = 0\}$ and satisfies
\begin{equation*}
\begin{cases}
\Delta {g} + {c}({x}) {g} < 0 & \text{in } \Omega, \\
\nabla {g} \cdot \nu \geq 0 & \text{on } \Gamma_{1}. 
\end{cases}
\end{equation*}
Therefore, by Lemma 6 and Lemma 7 in~\cite{YCL18}, we conclude that ${w}$ is nonpositive in $\Omega$. 
\end{proof}

\begin{rmk}
If the double domain $\tilde{\Omega}$, obtained by mirroring $\Omega$ along the hyperplane $\{{x}_{1} = 0\}$, is convex, then $\Omega$ satisfies $\nu \cdot {e}_{1} \geq 0$ on $\Gamma_{1}$. This condition is both natural and valid when applying the moving plane method, with $\Omega$ as the moving domain. 
\end{rmk}

Now, we turn to the maximum principle in sectorial-like domains. 
Let $\mathcal{A}_{\beta} \subset \R^{n}$ be an (infinite) sector with amplitude $\beta \in (0, 2\pi)$, 
\begin{equation*}
\mathcal{A}_{\beta} = \{{x} \in \R^{n}: \; \theta_{1}, \ldots, \theta_{n - 2} \in (0, \pi), \; \theta_{n - 1} \in (0, \beta)\}. 
\end{equation*}
Here $({r}, \theta_{1}, \ldots, \theta_{n - 1})$ denotes the usual polar coordinates in $\R^{n}$. 

\begin{lma} \label{lma203MP}
Suppose that the domain $\Omega$ is contained in an infinite sector $\mathcal{A}_{\beta}$ for some $\beta \in (0, 2\pi)$, with $\Gamma_{1} \subset \partial \Omega \cap \partial \mathcal{A}_{\beta}$ and $\Gamma_{0} = \partial \Omega \setminus \Gamma_{1}$. Suppose that ${c}$ satisfies \eqref{Yao0203c} and ${w}$ satisfies \eqref{Yao0203w}. If $\Omega$ is bounded and satisfies
\begin{equation}
\Omega \subset \{{x} \in \R^{n}: \; {x}_{n - 1}^{2} + {x}_{n}^{2} < \eta^{2}\}
\end{equation}
where $\eta = {j}_{0}/\sqrt{{c}_{0}}$ and ${j}_{0}$ is the first zero of the Bessel function of the first kind ${J}_{0}$, then
${w} \leq 0$ in $\Omega$. 
\end{lma}

\begin{proof}
Let
\begin{equation*}
{g}({x}) = {J}_{0}\big(\frac{{j}_{0}}{\eta} {\varrho}\big), \quad \varrho = \sqrt{{x}_{n - 1}^{2} + {x}_{n}^{2}}. 
\end{equation*}
Then
\begin{equation*}
\begin{cases}
\Delta {g} + {c}({x}) {g} = \big({c} - \big(\frac{{j}_{0}}{\eta}\big)^{2}\big) {g} < 0 & \text{in } \overline{\Omega}, \\
{g} > 0 & \text{in } \overline{\Omega}, \\
\partial_{\nu} {g} = 0 & \text{on } \Gamma_{1}. 
\end{cases}
\end{equation*}
Thus, ${g}$ is a strict upper solution corresponding to \eqref{Yao0203w}. We conclude the proof by the generalized maximum principle~\cite{PW67}; see also \cite[Lemma 2]{YCL18}. 
\end{proof}

In 1991, Berestycki and Nirenberg~\cite{BN91} demonstrated that the maximum principle with Dirichlet boundary conditions holds for bounded domains with small volume, specifically when $\operatorname{meas}(\Omega) < \delta$, where $\delta$ depends on the dimension ${n}$, the diameter of $\Omega$, and the coefficients of the elliptic operator. This result is proved using the fundamental inequality of Alexandroff, Bakelman, and Pucci; see \cite[Proposition 1.1]{BN91}. Later, Berestycki, Nirenberg, and Varadhan improved this result by showing that the refined maximum principle holds if $\operatorname{meas}(\Omega) < \delta$, with $\delta$ independent of the diameter of $\Omega$; see \cite[Theorem 2.6]{BNV94}. 

Before extending these results to mixed boundary problems in sector-like domains, we recall some well-known inequalities in sector domains (see any book dealing with Sobolev inequalities, e.g.,~\cite{GT83, AF03}). 

\begin{lma} \label{lma204Sob}
Let ${n} \geq 2$ and $1 \leq {p} < {n}$. Let $\mathcal{A}_{\beta}$ be a sector with amplitude $\beta \in (0, 2\pi)$. There exists a constant ${C}({n}, {p}) > 0$ such that
\begin{equation}
\bigg( \int_{ \mathcal{A}_{\beta} } |{v}|^{\frac{{n}{p}}{{n} - {p}}} d{x} \bigg)^{\frac{{n} - {p}}{{n}{p}}}
\leq {C}({n}, {p}) \beta^{ - \frac{1}{{n}} } 
\bigg( \int_{ \mathcal{A}_{\beta} } | \nabla {v} |^{ {p} } d{x} \bigg)^{ \frac{1}{ {p} } }
\end{equation}
for ${v} \in W^{1, 2}( \mathcal{A}_{\beta} )$. 
\end{lma}

\begin{proof}
Let ${q} = {n}{p}/({n} - {p})$. It is well-known that the constant
\begin{equation*}
{C}_{2\pi} = \sup\left\{ \frac{ \| {v} \|_{{L}^{ {q} }( \R^{n} )} }
{ \| \nabla {v} \|_{{L}^{ {p} }( \R^{n} )} }: \; {v} \in W^{1, {p}}( \R^{n} ), \; {v} \neq 0 \right\}
\end{equation*}
is positive and finite, and ${C}_{2\pi}$ depends only on ${n}$ and ${p}$. For more details, see \cite[Theorem 4.31]{AF03} or \cite[Theorem 7.10]{GT83}. For $\beta \in (0, 2\pi)$, define
\begin{equation} \label{Yao0205a}
{C}_{\beta} = \sup\bigg\{ \frac{ \| {v} \|_{{L}^{ {q} }( \mathcal{A}_{\beta} )} }
{ \| \nabla {v} \|_{{L}^{ {p} }( \mathcal{A}_{\beta} )} }: \; {v} \in W^{1, {p}}( \mathcal{A}_{\beta} ), \; {v} \neq 0 \bigg\}. 
\end{equation}
From Sobolev inequality in~\cite{AF03, GT83}, ${C}_{\beta}$ is also positive and finite for $\beta \in (0, 2\pi)$. Now, we illustrate the relationship between the constant ${C}_{\beta}$ and $\beta$. 

Let $\beta \in (0, \pi)$. Set ${A}_{j} = \{ {x} \in \mathcal{A}_{2\beta}: \; ({j} - 1)\beta < \theta_{n - 1} < {j}\beta \}$ for ${j} = 1, 2$. For ${v} \in W^{1, 2}( \mathcal{A}_{2\beta} )$, we have
\begin{equation*}
\begin{aligned}
\| {v} \|_{ {L}^{ {q} }( \mathcal{A}_{2\beta} ) }
& = \bigg( \sum_{ {j} = 1 }^{ 2 } \| {v} \|_{ {L}^{ {q} }( {A}_{j} ) }^{ {q} } \bigg)^{ 1/{q} }
\leq \bigg( \sum_{ {j} = 1 }^{ 2 } \big( {C}_{\beta} \| \nabla {v} \|_{ {L}^{ {p} }( {A}_{j} ) } \big)^{ {q} } \bigg)^{ 1/{q} }, \\
\| \nabla {v} \|_{ {L}^{ {p} }( \mathcal{A}_{2\beta} ) }
& = \bigg( \sum_{ {j} = 1 }^{ 2 } \| \nabla {v} \|_{ {L}^{ {p} }( {A}_{j} ) }^{ {p} } \bigg)^{ 1/{p} }
\geq \bigg( \sum_{ {j} = 1 }^{ 2 } \| \nabla {v} \|_{ {L}^{ {p} }( {A}_{j} ) }^{ {q} } \bigg)^{ 1/{q} }, 
\end{aligned}
\end{equation*}
which implies that
\begin{equation} \label{Yao0206a}
{C}_{2\beta} \leq {C}_{\beta} \quad \text{for } \beta \in (0, \pi). 
\end{equation}
Let $0 < \beta_{1} < \beta_{2} \leq 2\beta_{1} < 2\pi$ and ${v} \in W^{1, {p}}( \mathcal{A}_{\beta_{1}} )$. By reflecting with respect to the flat boundary $\partial \mathcal{A}_{\beta_{1}}$, we obtain a new function $\tilde{v}$ in $\mathcal{A}_{2\beta_{1}}$, with $\tilde{v} \in W^{1, {p}}( \mathcal{A}_{2\beta_{1}} )$. By direct calculation, one has 
\begin{gather*}
\| {v} \|_{ {L}^{ {q} }( \mathcal{A}_{\beta_{1}} ) }
\leq
\| \tilde{v} \|_{ {L}^{ {q} }( \mathcal{A}_{\beta_{2}} ) }
\leq
{C}_{\beta_{2}} \| \nabla \tilde{v} \|_{ {L}^{ {p} }( \mathcal{A}_{\beta_{2}} ) }, \\
\| \nabla {v} \|_{ {L}^{ {p} }( \mathcal{A}_{\beta_{1}} ) }
\geq
2^{ - \frac{1}{ {p} } } \| \nabla \tilde{v} \|_{ {L}^{ {p} }( \mathcal{A}_{2\beta_{1}} ) }
\geq
2^{ - \frac{1}{ {p} } } \| \nabla \tilde{v} \|_{ {L}^{ {p} }( \mathcal{A}_{\beta_{2}} ) }, 
\end{gather*}
and
\begin{equation*}
\| {v} \|_{ {L}^{ {q} }( \mathcal{A}_{\beta_{1}} ) }
 = 2^{ - \frac{1}{ {q} } } \| \tilde{v} \|_{ {L}^{ {q} }( \mathcal{A}_{2\beta_{1}} ) }
\leq 2^{ - \frac{1}{ {q} } } {C}_{2\beta_{1}} \| \nabla \tilde{v} \|_{ {L}^{ {p} }( \mathcal{A}_{2\beta_{1}} ) }
 = 2^{ \frac{1}{ {p} } - \frac{1}{ {q} } } {C}_{2\beta_{1}} \| \nabla {v} \|_{ {L}^{ {p} }( \mathcal{A}_{\beta_{1}} ) }. 
\end{equation*}
It follows that
\begin{equation*}
\begin{aligned}
\| {v} \|_{ {L}^{ {q} }( \mathcal{A}_{\beta_{1}} ) }
&\leq 2^{ \frac{1}{ {p} } } {C}_{\beta_{2}} \| \nabla {v} \|_{ {L}^{ {p} }( \mathcal{A}_{\beta_{1}} ) }, \\
\| {v} \|_{ {L}^{ {q} }( \mathcal{A}_{\beta_{1}} ) }
&\leq 2^{ \frac{1}{ {n} } } {C}_{2\beta_{1}} \| \nabla {v} \|_{ {L}^{ {p} }( \mathcal{A}_{\beta_{1}} ) }, 
\end{aligned}
\end{equation*}
and hence
\begin{equation} \label{Yao0206b}
{C}_{\beta_{1}} \leq 2^{ \frac{1}{ {p} } } {C}_{\beta_{2}}, \quad
{C}_{\beta_{1}} \leq 2^{ \frac{1}{ {n} } } {C}_{2\beta_{1}}
\quad \text{for } 0 < \beta_{1} < \beta_{2} \leq 2\beta_{1} < 2\pi. 
\end{equation}
In the case $\beta = \pi$, any function ${u}$ on the half-space $\mathcal{A}_{\pi}$ can be reflected to the whole space $\R^{n}$, yielding
\begin{equation*}
{C}_{\pi} \leq 2^{ \frac{1}{ {n} } } {C}_{2\pi}. 
\end{equation*}

Now for $\beta \leq \pi$, there exists a nonnegative integer ${k}$ such that $2^{ - {k} - 1 }\pi < \beta \leq 2^{ - {k} }\pi$. From \eqref{Yao0206b}, we deduce that
\begin{equation*}
{C}_{\beta} \leq 2^{ \frac{1}{ {p} } } {C}_{ 2^{ - {k} }\pi }
\leq 2^{ \frac{1}{ {p} } } \big( 2^{ \frac{1}{ {n} } } \big)^{ {k} } {C}_{\pi }
\leq 2^{ \frac{1}{ {p} } } \big( \frac{ \pi }{ \beta } \big)^{ \frac{1}{ {n} } } {C}_{\pi }
 = 2^{ \frac{1}{ {p} } } \pi^{ \frac{1}{ {n} } } \beta^{ - \frac{1}{ {n} } } {C}_{\pi }. 
\end{equation*}
Combining this with \eqref{Yao0206a}, we get for $\beta \in (\pi, 2\pi)$, 
\begin{equation*}
{C}_{\beta} \leq {C}_{\beta/2} \leq 2^{ \frac{1}{ {p} } + \frac{1}{ {n} } } \pi^{ \frac{1}{ {n} } } \beta^{ - \frac{1}{ {n} } } {C}_{\pi }. 
\end{equation*}
Thus, 
${C}_{\beta} \leq {C}({n}, {p}) \beta^{ - \frac{1}{ {n} } }$ with ${C}({n}, {p}) = 2^{ \frac{1}{ {p} } + \frac{2}{ {n} } } \pi^{ \frac{1}{ {n} } } {C}_{2\pi }$. 
\end{proof}

By considering a radially symmetric test function ${v} \in {C}_{c}^{1}(\R^{n})$, the quantity ${C}_{\beta} \beta^{1/{n}}$ is bounded from below by a constant independent of $\beta$. As a result, the quantity ${C}_{\beta} \beta^{1/{n}}$ is uniformly bounded above and below by two positive constants. 

\begin{lma}[Maximum principle with small volume] \label{lma205MP}
Suppose that the domain $\Omega$ is contained in a sector $\mathcal{A}_{\beta}$ for some $\beta \in (0, 2\pi)$, 
$\Gamma_{1} \subset \partial \Omega \cap \partial \mathcal{A}_{\beta}$ and $\Gamma_{0} = \partial \Omega \setminus \Gamma_{1}$. Suppose that ${c}$ satisfies \eqref{Yao0203c}. Assume that ${w} \in W^{1, 2}(\Omega)$ satisfies \eqref{Yao0203w}, i.e., ${w}^{ + } \in H_{0}^{1}(\Omega \cup \Gamma_{1})$ and for all ${\phi} \in H_{0}^{1}(\Omega \cup \Gamma_{1})$, ${\phi} \geq 0$, there holds: 
\begin{equation} \label{Yao0208}
 \int_{\Omega} \nabla {w} \nabla {\phi}\, d{x} - \int_{\Omega} {c} {w} {\phi}\, d{x} \leq 0. 
\end{equation}
Then there exists a small constant $\eta > 0$, depending only on the ${L}^{\infty}$ bound of ${c}^{ + }$ and the dimension ${n}$, such that if
\begin{equation} \label{Yao0209}
\operatorname{meas}(\Omega) < \eta \beta, 
\end{equation}
then ${w} \leq 0$ in $\Omega$. 
\end{lma}

\begin{proof}
Let us fixed any ${r}$ such that ${r} \in (1, 2)$ if ${n} = 2$ and ${r} \in [2{n}/({n} + 2), 2]$ if ${n} \geq 3$. 
Set ${\phi} = {w}^{ + }$ and $\Omega_{ + } = \{ {x}: {w}({x}) > 0 \} \subset \Omega$. Then ${\phi} \in W^{1, 2}(\mathcal{A}_{\beta})$ with $\operatorname{supp}({\phi}) \subset \Omega \cup \partial \mathcal{A}_{\beta}$. Taking ${\phi}$ as a test function in \eqref{Yao0208}, we get
\begin{equation*}
\| \nabla {w}^{ + } \|_{ {L}^{2}(\mathcal{A}_{\beta}) } \leq {c}_{0}^{1/2} \| {w}^{ + } \|_{ {L}^{2}(\mathcal{A}_{\beta}) }. 
\end{equation*}
From the Sobolev inequality in \autoref{lma204Sob}, we have
\begin{equation*}
\| {w}^{ + } \|_{ {L}^{ \frac{ {n}{r} }{ {n} - {r} } }( \mathcal{A}_{\beta} ) } \leq {C}_{ {n}, {r} } \beta^{ - 1/{n} } \| \nabla {w}^{ + } \|_{ {L}^{ {r} }( \mathcal{A}_{\beta} ) }. 
\end{equation*}
Combining these inequalities with H\"older's inequality, we deduce that
\begin{equation*}
\| \nabla {w}^{ + } \|_{ {L}^{ {r} }( \mathcal{A}_{\beta} ) }
\leq |\Omega_{ + }|^{ \frac{1}{{r}} - \frac{1}{2} } \| \nabla {w}^{ + } \|_{ {L}^{2}( \mathcal{A}_{\beta} ) }
\leq |\Omega_{ + }|^{ \frac{1}{{r}} - \frac{1}{2} } {c}_{0}^{\frac{1}{2}} \| {w}^{ + } \|_{ {L}^{2}( \mathcal{A}_{\beta} ) }, 
\end{equation*}
and
\begin{equation*}
\| {w}^{ + } \|_{ {L}^{2}( \mathcal{A}_{\beta} ) }
\leq |\Omega_{ + }|^{ \frac{1}{2} - \frac{1}{{r}} + \frac{1}{{n}} } \| {w}^{ + } \|_{ {L}^{ \frac{ {n}{r} }{ {n} - {r} } }( \mathcal{A}_{\beta} ) }
\leq |\Omega_{ + }|^{ \frac{1}{2} - \frac{1}{{r}} + \frac{1}{{n}} } {C}_{ {n}, {r} } \beta^{ - 1/{n} } \| \nabla {w}^{ + } \|_{ {L}^{ {r} }( \mathcal{A}_{\beta} ) } 
\end{equation*}
where $|\Omega_{ + }| = \operatorname{meas}(\Omega^{ + })$ denotes the measure of $\Omega_{ + }$. It follows that
\begin{equation*}
\| \nabla {w}^{ + } \|_{ {L}^{ {r} }( \mathcal{A}_{\beta} ) }
\leq {C}_{ {n}, {r} } {c}_{0}^{\frac{1}{2}} |\Omega_{ + }|^{\frac{1}{{n}}} \beta^{ - \frac{1}{{n}} } \| \nabla {w}^{ + } \|_{ {L}^{ {r} }( \mathcal{A}_{\beta} ) }. 
\end{equation*}
Let $\eta = (2 {c}_{0}^{1/2} {C}_{ {n}, {r} })^{ - {n} }$. Since ${r}$ depends only on ${n}$, we know $\eta$ depends only on ${n}$ and ${c}_{0}$, and is independent of $\beta \in (0, 2\pi)$. Thus, if $|\Omega_{ + }| < \eta \beta$, i.e., ${C}_{ {n}, {r} } {c}_{0}^{1/2} (|\Omega_{ + }|/\beta)^{ 1/{n} } < 1/2$, then $\| \nabla {w}^{ + } \|_{ {L}^{ {r} }( \mathcal{A}_{\beta} ) } = 0$ and hence ${w}^{ + } \equiv 0$. Therefore, ${w} \leq 0$ in $\Omega$. In particular, ${w}^{ + } \equiv 0$ if $|\Omega| < \eta \beta$. 
\end{proof}

\begin{rmks}
It is worth noting that the domains $\Omega$ considered in \autoref{lma205MP} can be either bounded or unbounded, as long as they have finite and small volume $|\Omega|$. Importantly, the result established in \autoref{lma205MP} can be applied to general operators of divergence form. 
\end{rmks}

\begin{rmks}
If the amplitude $\beta$ corresponding to the domain $\Omega$ is $\pi/{k}$ for some integer ${k}$, it is possible to extend the functions ${w}$ and ${c}$ to a sectorial domain with a flat amplitude by a finite number of reflections with respect to the flat boundary $\partial \mathcal{A}_{\beta}$. By reflecting once again, we obtain new functions $\tilde{w}$ and $\tilde{c}$ defined in a new domain $\tilde{\Omega}$. It follows that $|\tilde{\Omega}| = 2{k}|\Omega| = 2\pi|\Omega|/\beta$. Moreover, $\tilde{w}$ satisfies a similar elliptic inequality as in \eqref{Yao0203w}, but with Dirichlet boundary condition: $\tilde{w} \leq 0$ on $\partial \tilde{\Omega}$. To establish the non-positivity of $\tilde{w}$, a sufficient condition is given in~\cite{BN91, BNV94}, which states that $|\tilde{\Omega}| < \delta$, where $\delta$ is a positive constant. In terms of the original domain $\Omega$, this condition translates to $|\Omega| < \delta \beta / (2\pi)$. Remarkably, this coincides with the condition \eqref{Yao0209} in \autoref{lma205MP}. Hence, it appears that the condition \eqref{Yao0209} is optimal in ensuring the desired non-positivity of ${w}$. 
\end{rmks}


\subsection{The monotonicity near the Dirichlet boundary}

Consider a solution ${u}({x})$ to the equation 
\begin{equation} \label{Yao0211}
\Delta {u} + {f}( {u} ) = 0 \mbox{ in } \Omega, 
\end{equation}
where ${f}$ is a locally Lipschitz continuous function, and $\Omega$ is a bounded domain. 

\begin{lma} \label{lma206GNN}
Let $\bar{x} \in \partial\Omega$ and let $\nu(\bar{x})$ denote the outward unit normal vector at $\bar{x} \in \partial\Omega$. Let $\gamma$ be a unit vector in $\R^{n}$ satisfying $\nu(\bar{x}) \cdot \gamma > 0$. For some $\epsilon > 0$, define $\Omega_{\epsilon, \bar{x}} = \Omega \cap \{ |{x} - \bar{x}| < \epsilon \}$ and $\Gamma_{\epsilon, \bar{x}} = \partial\Omega \cap \{ |{x} - \bar{x}| < \epsilon \}$. Assume that ${u}$ is a $C^{2}$ solution of \eqref{Yao0211} in $\overline{\Omega_{\epsilon, \bar{x}}}$ satisfying 
\begin{equation*}
{u} \geq, \not\equiv 0 \mbox{ in } \Omega_{\epsilon} \mbox{ and } {u} = 0 \mbox{ on } \Gamma_{\epsilon, \bar{x}}.
\end{equation*}
Moreover, suppose that the boundary $\Gamma_{\epsilon, \bar{x}}$ is of class $C^{2}$. Then there exists a $\delta > 0$ such that
\begin{equation*}
\frac{\partial {u}}{\partial \gamma} < 0 \mbox{ in } \Omega_{\delta, \bar{x}} = \Omega \cap \{ |{x} - \bar{x}| < \delta \}. 
\end{equation*}
\end{lma}

\begin{proof}
See the proof in \cite[Lemma 2.1]{GNN79}.
\end{proof}


\subsection{The zero set}

Now we state the properties of the zero set of elliptic solutions, which is a key tool for obtaining the signs of tangential derivatives of the solution ${u}$ of \eqref{Yao0105} along the Neumann boundary $\Gamma_{N}$. 

\begin{prop} \label{prop207}
Let $\Omega$ be a planar domain with smooth boundary. Let ${V} \in {L}^{\infty}(\Omega)$, and let ${u}$ be a weak solution of $\Delta {u} + {V}({x}) {u} = 0$ in $\Omega$. Then ${u} \in {C}^{1}(\Omega)$. Furthermore, ${u}$ has the following properties: 

{\rm(i)}
If ${u}$ has a zero of any order at ${x}^{0}$ in $\Omega$, then ${u} \equiv 0$ in $\Omega$. 

{\rm(ii)}
If ${u}$ has a zero of order (exactly) ${l}$ at ${x}^{0}$ in $\Omega$, then the Taylor expansion of ${u}$ is
\begin{equation*}
{u}({x}) = {H}_{l}({x} - {x}^{0}) + O(|{x} - {x}^{0}|^{{l} + 1}), 
\end{equation*}
where ${H}_{l}$ is a real-valued, non-zero, harmonic, homogeneous polynomial of degree ${l}$. Therefore, 
$\{{u} = 0\}$ has exactly $2{l}$ branches at ${x}^{0}$. 

{\rm(iii)}
If ${u}$ has a zero of order (exactly) ${l}$ at ${x}^{0}$ on $\partial \Omega$ and if ${u}$ satisfies the Neumann (respectively Dirichlet) boundary condition, then
\begin{equation*}
{u}({x}) = {C}_{0} {r}^{l} \cos( {l} \theta ) + O({r}^{{l} + 1}), 
\end{equation*}
for some ${C}_{0} \in \R \setminus \{ 0 \}$, where $({r}, \theta)$ are polar coordinates of ${x}$ around ${x}^{0}$. The angle $\theta$ is chosen so that the tangent to the boundary at ${x}^{0}$ is given by the equation $\sin \theta = 0$ (respectively $\cos \theta = 0$). 
\end{prop}

\begin{proof}
(i) and (ii) are well-known~\cite{HW53}. (iii) is taken from \cite[Proposition 4.1]{HHHO99}. See also~\cite{HHT09}. 
\end{proof}


\subsection{The relation of angles}

We usually denote $|{P}{Q}|$ as the length of the line segment ${P}{Q}$ with endpoints ${P}$ and ${Q}$. In~\cite{CLY21}, the angles $\vartheta_{A}$ and $\vartheta_{B}$, given in \eqref{Yao0310a} and \eqref{Yao0310b}, have a key relationship. 
We now state further properties which will be used later. 

\begin{lma} \label{lma208}
Let ${H}$ be the midpoint of side ${A}_{2}{A}_{1}$ of the isosceles triangle ${V}{A}_{2}{A}_{1}$, where $|{V}{A}_{2}| = |{V}{A}_{1}|$. 
Let ${O}$ be a point on the line segment ${V}{H}$ such that
\begin{equation*}
\beta < \alpha \leq \pi, 
\end{equation*}
where $\beta = \angle {A}_{1}{V}{A}_{2}$ and $\alpha = \angle {A}_{1}{O}{A}_{2}$. 
Let ${P}$ be any point located on the (open) line segment ${V}\bar{{P}}$, where $\bar{{P}}$ denotes the intersection of the lines ${A}_{2}{O}$ and ${V}{A}_{1}$. 
We denote the angle ${A}_{1}{P}{O}$ as $\vartheta_{A}$ and the angle ${A}_{1}{P}{A}_{2}$ as $\vartheta_{B}$; see \autoref{Fig02Angles}. 
Then we have
\begin{enumerate}[label = {\rm(\roman*)}, start = 1]
\item \label{Yaoit23a}
$\vartheta_{B} < 2\vartheta_{A}$; 
\item \label{Yaoit23b}
$2\vartheta_{B} - \vartheta_{A} < (\pi + \beta)/2$ provided $\vartheta_{B} < \pi/2$; 
\item \label{Yaoit23c}
$2\vartheta_{B} - \vartheta_{A} < \pi$ provided $\beta \leq 2\pi/3$; 
\item \label{Yaoit23d}
$2\vartheta_{B} - \vartheta_{A} < \pi$ provided $\vartheta_{A} \geq \beta$. 
\end{enumerate}
\end{lma}

\begin{figure}[htp] \centering
\begin{tikzpicture}[scale = 3] 
\pgfmathsetmacro\ALPHA{72.04}; \pgfmathsetmacro\BETA{46.53}; \pgfmathsetmacro\radius{1.00}; 
\pgfmathsetmacro\aaa{\radius*sin(abs(\BETA - \ALPHA))/sin(\BETA)}; 
\pgfmathsetmacro\ka{tan(\BETA)}; 
\pgfmathsetmacro\xA{\radius*cos(\ALPHA)}; \pgfmathsetmacro\yA{\radius*sin(\ALPHA)}; 
\draw[blue, thick] (\xA, - \yA) node[left = 2pt] {\footnotesize ${A}_{1}$} arc ( - \ALPHA : \ALPHA : \radius) node[left = 2pt] {\footnotesize ${A}_{2}$} -- ( - \aaa, 0) node[left] {\footnotesize ${V}$} -- cycle; 
\draw[dotted] (\xA, - \yA) -- (0, 0) node[left = 2pt, above] {\footnotesize ${O}$} -- (\xA, \yA); 
\pgfmathsetmacro\LenA{abs(\aaa)*sin(\ALPHA)/sin(\ALPHA + \BETA)}; 
\pgfmathsetmacro\LenB{\LenA*0.63}; 
\pgfmathsetmacro\xPP{\LenA*cos(\BETA) - \aaa}; \pgfmathsetmacro\yPP{ - \LenA*sin(\BETA)}; 
\pgfmathsetmacro\xP{\LenB*cos(\BETA) - \aaa}; \pgfmathsetmacro\yP{ - \LenB*sin(\BETA)}; 
\pgfmathsetmacro\xR{\xA}; \pgfmathsetmacro\yR{\xR*\yP/\xP}; 
\draw[dotted] (0, 0) -- (\xR, \yR) node [right] {\footnotesize ${R}$}; 
\draw[] (\xA, \yA) -- (\xP, \yP) node [left = 3pt, below] {\footnotesize ${P}$} -- (0, 0); 
\draw[dotted] (\xPP, \yPP) node [left = 3pt, below] {\footnotesize $\bar{P}$} -- (0, 0); 
\draw[dashed] (\xA, \yA) -- (\xA, - \yA); 
\draw[dashed] ( - \aaa, 0) -- (\xA, 0) node [right] {\footnotesize ${H}$} -- (\xP, \yP); 
\draw[] (\xP, \yP) -- (\xP, - \yP) node[above] {\footnotesize ${Q}$}; 
\draw[dashed] (\xP, \yP) -- (\xA, \yP) node [right] {\footnotesize ${D}$}; 
\end{tikzpicture} 
\caption{The relation of the angles. }
\label{Fig02Angles}
\end{figure}
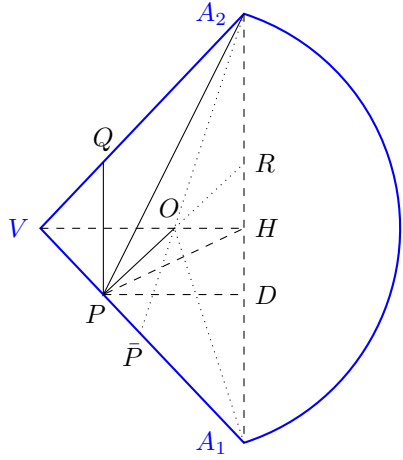

\begin{proof}
Part \ref{Yaoit23a}. 
Let ${R}$ be the intersection of the lines ${P}{O}$ and ${A}_{2}{A}_{1}$; see \autoref{Fig02Angles}. 
It is clear that $|{A}_{2}{P}|/|{A}_{1}{P}| > 1 > |{A}_{2}{R}|/|{A}_{1}{R}|$, and hence $\angle {A}_{2}{P}{R} < \angle {A}_{1}{P}{R}$. This gives the proof of part (i). 

Part \ref{Yaoit23b}. 
Let ${Q}$ be the point on the line segment ${V}{A}_{2}$ such that the straight line ${Q}{P}$ is perpendicular to the line ${V}{H}$, i.e., ${Q}{P} \perp {V}{H}$. 
Suppose that $\vartheta_{B} < \pi/2$, i.e., $\angle {A}_{2}{P}{A}_{1} < \pi/2$. Then $|{H}{P}| > |{H}{A}_{2}|$, and thus $|{R}{P}| > |{R}{A}_{2}|$. It follows that $\angle {Q}{P}{A}_{2} = \angle {R}{A}_{2}{P} > \angle {R}{P}{A}_{2}$, which implies $2\vartheta_{B} - \vartheta_{A} < (\pi + \beta)/2$. 

Part \ref{Yaoit23c}. 
Let ${D}$ be the point on the line segment ${A}_{2}{A}_{1}$ such that the straight line ${D}{P}$ is perpendicular to the line ${A}_{2}{A}_{1}$, i.e., ${D}{P} \perp {A}_{2}{A}_{1}$. 
Let $\epsilon = |{O}{H}| / |{O}{V}| \in [0, \infty)$, set $\kappa = \tan(\beta/2) > 0$ and ${t} = \tan(\vartheta_{A} - \beta/2) > 0$. Then, 
\begin{equation*}
(|{P}{D}| - |{O}{H}|)(\kappa + {t}) = |{V}{O}|\kappa, \quad
|{A}_{2}{D}| = 2(|{O}{H}| + |{O}{V}|)\kappa - |{P}{D}|\kappa, 
\end{equation*}
which gives 
\begin{gather*}
|{V}{O}| = \frac{ \kappa + {t} }{ \kappa + (\kappa + {t})\epsilon } |{P}{D}|, 
\\ 
\frac{ |{A}_{2}{D}| }{ |{P}{D}| }
 = 2(1 + \epsilon )\kappa \frac{ \kappa + {t} }{ \kappa + \epsilon \kappa + \epsilon {t} } - \kappa
 = \frac{ ( \kappa + \epsilon \kappa + 2{t} + \epsilon {t} )\kappa }{ \kappa + \epsilon \kappa + \epsilon {t} }. 
\end{gather*}
Now, by setting $2\angle {A}_{2}{P}{A}_{1} - \angle {O}{P}{A}_{1} - \pi = {f}({t}, \kappa)$, we have
\begin{equation*}
{f}({t}, \kappa) = 2\arctan \frac{ ( \kappa + \epsilon \kappa + 2{t} + \epsilon {t} )\kappa }{ \kappa + \epsilon \kappa + \epsilon {t} } - \arctan {t} + \arctan \kappa - \pi, 
\end{equation*}
and
\begin{equation*}
{f}(0, \kappa) = 3\arctan \kappa - \pi, \quad {f}( \tan \frac{\alpha}{2}, \kappa ) < 0, \quad {f}(\infty, \kappa) < 0. 
\end{equation*}
Note that the function $\kappa \in (0, \infty) \mapsto {f}({t}, \kappa)$ is strictly increasing. Moreover, 
\begin{equation*}
\partial_{ {t} } {f}({t}, \kappa) = \frac{ - {a}_{0} - {a}_{1}{t} - {a}_{2}{t}^{2} }{ ( {t}^{2} + 1 ) \left[ ( \kappa + \epsilon \kappa + 2{t} + \epsilon {t} )^{2} \kappa^{2} + ( \kappa + \epsilon \kappa + \epsilon {t} )^{2} \right] }
\end{equation*}
with ${a}_{2} = (1 + \kappa^{2}) \epsilon^{2} \geq 0$, ${a}_{1} = 2(1 + \epsilon)(\epsilon + (2 + \epsilon)\kappa^{2})\kappa > 0$, and 
\begin{equation*}
{a}_{0} = (1 + \epsilon)(\epsilon + (1 + \epsilon)\kappa^{2} - 3)\kappa^{2}. 
\end{equation*}
Observe that the function $\partial_{ {t} } {f}({t}, \kappa) < 0$ when ${a}_{0} \geq 0$ (e.g., $\kappa \geq \sqrt{3}$). In this case, the function ${t} \in [0, \infty) \mapsto {f}({t}, \kappa)$ is strictly decreasing and thus has exactly one zero. Therefore, when $\kappa \leq \sqrt{3}$ and ${t} > 0$, we have 
\begin{equation*}
{f}({t}, \kappa) \leq {f}({t}, \sqrt{3}) < {f}(0, \sqrt{3}) = 0. 
\end{equation*}
This completes the proof of part \ref{Yaoit23c}. 

Part \ref{Yaoit23d}. 
This is immediate, since $\vartheta_{B} < (\pi + \beta)/2$ and $\vartheta_{A} \geq \beta$. 
\end{proof}


\section{The start of the method of moving planes} \label{Sec03MMP}

In this section, we introduce the moving lines ${T}_{\lambda, \vartheta}$ and the moving domains ${D}_{\lambda, \vartheta, \vartheta_{1}}$. With these definitions, and for general $\beta$, the method of moving planes (\textbf{MMP}) can be carried out up to a maximal position $\lambda \geq \lambda_{\sharp}$, where $\lambda_{\sharp}$ is defined in \eqref{Yao0341}. 

\subsection{The moving lines and moving domains}

We now define the families ${T}_{\lambda, \vartheta}$ and ${D}_{\lambda, \vartheta, \vartheta_{1}}$, following the approach of~\cite{CLY21, YCG21}. For notational simplicity, we translate the vertex of the sector to the origin, so ${V} = (0, 0)$. The center of the reference ball is then ${O} = ( - {a}, 0 )$, and the mixed boundary points are given by ${P}_{\pm} = (\cos(\alpha/2) - {a}, \pm \sin(\alpha/2))$, which we also denote by ${A}_{1} = {P}_{ - }$ and ${A}_{2} = {P}_{ + }$. Here the constant ${a}$ (depending on the angles $\alpha$ and $\beta$) is given in \eqref{Yao0103b}. The domain and corresponding boundary are given by
\begin{equation}
\begin{aligned}
\mathcal{C} & = \mathcal{C}_{\alpha, \beta} = \big\{ ({x}_{1}, {x}_{2}) \in \R^{2}: \; {x}_{1} > |{x}_{2}| \cot \tfrac{\beta}{2} \big\}, \\
{B} & = \big\{ ({x}_{1}, {x}_{2}) \in \R^{2}: \; ({x}_{1} + {a})^{2} + {x}_{2}^{2} < 1 \big\}, 
\\
\Sigma & = \Sigma_{\alpha, \beta} = \big\{ ({x}_{1}, {x}_{2}) \in \R^{2}: \; |{x}_{2}| \cot \tfrac{\beta}{2} < {x}_{1} < \textstyle\sqrt{1 - {x}_{2}^{2}} - {a} \big\}. 
\end{aligned}
\end{equation}
Moreover, $\Gamma_{D} = \{({a} + \cos\theta, \sin\theta) \in \R^{2}: |\theta| \leq \alpha/2\}$ and $\Gamma_{N} = \partial\Omega \setminus \Gamma_{D}$. We define the lower and upper Neumann boundaries as follows:
\begin{equation*}
\Gamma_{N}^{ - } = \{({x}_{1}, {x}_{2}) \in \Gamma_{N}: \; {x}_{2} < 0 \} \quad \text{and} \quad \Gamma_{N}^{ + } = \{({x}_{1}, {x}_{2}) \in \Gamma_{N}: \; {x}_{2} > 0 \}. 
\end{equation*}

Set
\begin{equation} \label{Yao0303}
{P}_{\lambda} = ({x}_{1\lambda}, {x}_{2\lambda}) \quad \text{with} \quad {x}_{1\lambda} = \lambda\cos\frac{\beta}{2}, \quad {x}_{2\lambda} = - \lambda\sin\frac{\beta}{2}. 
\end{equation}
We consider the moving line ${T}_{\lambda, \vartheta}$ passing through ${P}_{\lambda}$ and forming with the lower Neumann side $\Gamma_{N}^{ - }$ an angle $\vartheta \in (0, (\pi + \beta)/2]$, that is, 
\begin{equation*}
{T}_{\lambda, \vartheta} = \{ {x} \in \R^{2}: \; ( {x} - {P}_{\lambda}) \cdot \mathbf{e}_{\vartheta - \beta/2 - \pi/2} = 0\}
\end{equation*}
where $\mathbf{e}_{\theta} = (\cos\theta, \sin\theta)$ is the unit vector. Equivalently, 
\begin{equation} \label{Yao0304a}
{T}_{\lambda, \vartheta} = \{{x} \in \R^{2}: \; 
({x}_{1} - {x}_{1\lambda})\sin(\vartheta - \tfrac{\beta}{2}) - ({x}_{2} - {x}_{2\lambda})\cos(\vartheta - \tfrac{\beta}{2}) = 0\}. 
\end{equation}
See \autoref{Fig03MMP90} for $\vartheta = \pi/2$, and \autoref{Fig03Lines} for general $\vartheta$. As usual, we denote by $\Sigma_{\lambda, \vartheta}$ the right open cap cut out from $\Sigma$ by ${T}_{\lambda, \vartheta}$. For a point ${x} \in \R^{2}$, we write ${x}^{\lambda, \vartheta}$ for its reflection with respect to ${T}_{\lambda, \vartheta}$. For a set $\Omega \subset \R^{2}$, we denote by $\Omega^{\lambda, \vartheta}$ its reflection with respect to ${T}_{\lambda, \vartheta}$. Since the reflection of the cap $\Sigma_{\lambda, \vartheta}$ may not be contained in $\Sigma$, it is natural to define the family of moving domains ${D}_{\lambda, \vartheta}$ by
\begin{equation}
\begin{aligned}
{D}_{\lambda, \vartheta} & = 
\{ {x} \in \Sigma: \; {x}^{\lambda, \vartheta} \in \Sigma \; \; \text{and} \; \; {x} \in \Sigma^{\lambda, \vartheta} \}
\\ & = 
\{ {x} \in \Sigma: \; {x}^{\lambda, \vartheta} \in \Sigma \; \; \text{and} \; \; ( {x} - {P}_{\lambda}) \cdot \mathbf{e}_{\vartheta - \beta/2 - \pi/2} > 0 \}
\end{aligned}
\end{equation}
for $\lambda \geq 0$ and $0 < \vartheta < \pi$; see \autoref{Fig03MMP90} for the case $\vartheta = \pi/2$. 
Our aim is to show that
\begin{equation} \label{Yao0305b}
{w}^{\lambda, \vartheta}({x}) = {u}({x}) - {u}^{\lambda, \vartheta}({x}) < 0 
\quad \text{for} \quad {x} \in {D}_{\lambda, \vartheta}, 
\end{equation}
where ${u}^{\lambda, \vartheta}({x}) = {u}({x}^{\lambda, \vartheta})$ denotes the reflection of ${u}$ with respect to ${T}_{\lambda, \vartheta}$. 

\begin{figure}[htp] \centering
\begin{tikzpicture}[scale = 1.25] 
\pgfmathsetmacro\ALPHA{48.04}; \pgfmathsetmacro\BETA{29.29}; \pgfmathsetmacro\radius{2.00}; 
\pgfmathsetmacro\aaa{\radius*sin(abs(\BETA - \ALPHA))/sin(\BETA)}; 
\pgfmathsetmacro\ka{tan(\BETA)}; 
\pgfmathsetmacro\xA{\radius*cos(\ALPHA)}; \pgfmathsetmacro\yA{\radius*sin(\ALPHA)}; 
\pgfmathsetmacro\LenNeu{\radius*sin(\ALPHA)/sin(\BETA)}; 
\pgfmathsetmacro\THETA{90}; \pgfmathsetmacro\LAM{0.6831*\LenNeu}; 
\pgfmathsetmacro\THEbb{\THETA - \BETA}; 
\pgfmathsetmacro\xP{ - \aaa + \LAM*cos(\BETA)}; \pgfmathsetmacro\yP{ - \LAM*sin(\BETA)}; 
\pgfmathsetmacro\kb{tan(\THETA - \BETA)}; \pgfmathsetmacro\bb{\yP - \xP*\kb}; 
\pgfmathsetmacro\xVv{ - \aaa*cos(2*\THEbb) - \bb*sin(2*\THEbb)}; 
\pgfmathsetmacro\yVv{ - \aaa*sin(2*\THEbb) + 2*\bb*(cos(\THEbb))^2}; 
\pgfmathsetmacro\Coefa{1}; 
\pgfmathsetmacro\Coefb{(\yP*cos(\THEbb) - \xP*sin(\THEbb))*sin(\THEbb)}; 
\pgfmathsetmacro\Coefc{(\yP*cos(\THEbb) - \xP*sin(\THEbb))^2 - \radius^2*(cos(\THEbb))^2}; 
\pgfmathsetmacro\xQ{( - \Coefb + sqrt((\Coefb)^2 - \Coefa*\Coefc))/\Coefa}; 
\pgfmathsetmacro\yQ{\yP + (\xQ - \xP)*tan(\THEbb)}; 
\pgfmathsetmacro\yS{\xA}; 
\pgfmathsetmacro\xS{ - \yA}; 
\pgfmathsetmacro\xSS{2*\xP - \xS}; 
\pgfmathsetmacro\ySS{2*\yP - \yS}; 
\fill[fill = yellow, draw = black, very thin] 
(\xQ, \yQ) arc ({2*\THEbb - asin(\yQ/\radius)} : {2*\THEbb + \ALPHA} : \radius) -- (\xP, \yP) -- cycle; 
\fill[fill = green, fill opacity = 1, draw = black, very thin] 
(\xQ, \yQ) arc ({asin(\yQ/\radius)} : { - \ALPHA} : \radius) -- (\xP, \yP) -- cycle; 
\draw[] ({\xA}, { - \yA}) node[left] {\footnotesize ${P}_{ - }$} arc ( - \ALPHA : \ALPHA : \radius) node[left] {\footnotesize ${P}_{ + }$} -- ( - \aaa, 0) node[left] {\footnotesize ${V}$} -- cycle; 
\draw[thin, dotted] ({\xA}, { - \yA}) -- (0, 0) node[left] {\footnotesize ${O}$} -- ({\xA}, {\yA}); 
\fill[red] ({ - \aaa + \LAM*cos(\BETA)}, { - \LAM*sin(\BETA)}) circle (0.008) node[below] {\footnotesize ${P}_{\lambda}$}; 
\draw[red, thick] (\xP, \yP) -- ({\xP + 1.2*(\xQ - \xP)}, {\yP + 1.2*(\yQ - \yP)}) node[right] {\footnotesize ${T}_{\lambda, \vartheta}$}; 
\pgfmathsetmacro\radB{0.11*\radius}; 
\draw[very thin] ({\xP + \radB*cos(\THETA - \BETA)}, {\yP + \radB*sin(\THETA - \BETA)}) arc ({\THETA - \BETA} : { - \BETA} : \radB); 
\draw ({\xP + \radB*cos(\THETA/2 - \BETA)}, {\yP + \radB*sin(\THETA/2 - \BETA)}) node [right] {\footnotesize $\vartheta = \frac{\pi}{2}$}; 
\end{tikzpicture}
\quad
\begin{tikzpicture}[scale = 1.25] 
\pgfmathsetmacro\ALPHA{48.04}; \pgfmathsetmacro\BETA{29.29}; \pgfmathsetmacro\radius{2.00}; 
\pgfmathsetmacro\aaa{\radius*sin(abs(\BETA - \ALPHA))/sin(\BETA)}; 
\pgfmathsetmacro\ka{tan(\BETA)}; 
\pgfmathsetmacro\xA{\radius*cos(\ALPHA)}; \pgfmathsetmacro\yA{\radius*sin(\ALPHA)}; 
\pgfmathsetmacro\LenNeu{\radius*sin(\ALPHA)/sin(\BETA)}; 
\pgfmathsetmacro\THETA{90}; \pgfmathsetmacro\LAM{0.6031*\LenNeu}; 
\pgfmathsetmacro\THEbb{\THETA - \BETA}; 
\pgfmathsetmacro\xP{ - \aaa + \LAM*cos(\BETA)}; \pgfmathsetmacro\yP{ - \LAM*sin(\BETA)}; 
\pgfmathsetmacro\kb{tan(\THETA - \BETA)}; \pgfmathsetmacro\bb{\yP - \xP*\kb}; 
\pgfmathsetmacro\xVv{ - \aaa*cos(2*\THEbb) - \bb*sin(2*\THEbb)}; 
\pgfmathsetmacro\yVv{ - \aaa*sin(2*\THEbb) + 2*\bb*(cos(\THEbb))^2}; 
\pgfmathsetmacro\Coefa{1}; 
\pgfmathsetmacro\Coefb{(\yP*cos(\THEbb) - \xP*sin(\THEbb))*sin(\THEbb)}; 
\pgfmathsetmacro\Coefc{(\yP*cos(\THEbb) - \xP*sin(\THEbb))^2 - \radius^2*(cos(\THEbb))^2}; 
\pgfmathsetmacro\xQ{( - \Coefb + sqrt((\Coefb)^2 - \Coefa*\Coefc))/\Coefa}; 
\pgfmathsetmacro\yQ{\yP + (\xQ - \xP)*tan(\THEbb)}; 
\pgfmathsetmacro\THETAd{2*\THEbb - \BETA}; 
\pgfmathsetmacro\Coefa{1}; 
\pgfmathsetmacro\Coefb{(\yVv*cos(\THETAd) - \xVv*sin(\THETAd))*cos(\THETAd)}; 
\pgfmathsetmacro\Coefc{(\yVv*cos(\THETAd) - \xVv*sin(\THETAd))^2 - \radius^2*(sin(\THETAd))^2}; 
\pgfmathsetmacro\yRa{(\Coefb + sqrt((\Coefb)^2 - \Coefa*\Coefc))/\Coefa}; 
\pgfmathsetmacro\xRa{\xVv + (\yRa - \yVv)*cot(\THETAd)}; 
\pgfmathsetmacro\yRb{(\Coefb - sqrt((\Coefb)^2 - \Coefa*\Coefc))/\Coefa}; 
\pgfmathsetmacro\xRb{\xVv + (\yRb - \yVv)*cot(\THETAd)}; 
\pgfmathsetmacro\xRRa{ ((1 - \kb^2)*\xRa + 2*\kb*\yRa - 2*\bb*\kb)/(1 + \kb^2) }; 
\pgfmathsetmacro\yRRa{ (2*\kb*\xRa - (1 - \kb^2)*\yRa + 2*\bb )/(1 + \kb^2) }; 
\pgfmathsetmacro\xRRb{ ((1 - \kb^2)*\xRb + 2*\kb*\yRb - 2*\bb*\kb)/(1 + \kb^2) }; 
\pgfmathsetmacro\yRRb{ (2*\kb*\xRb - (1 - \kb^2)*\yRb + 2*\bb )/(1 + \kb^2) }; 
\pgfmathsetmacro\yS{\xA}; 
\pgfmathsetmacro\xS{ - \yA}; 
\pgfmathsetmacro\xSS{2*\xP - \xS}; 
\pgfmathsetmacro\ySS{2*\yP - \yS}; 
\fill[fill = yellow, draw = black, very thin] 
(\xQ, \yQ) arc ({2*\THEbb - asin(\yQ/\radius)} : {2*\THEbb - asin(\yRa/\radius)} : \radius) -- (\xRRb, \yRRb) arc ({2*\THEbb - asin(\yRb/\radius)} : {2*\THEbb + \ALPHA} : \radius) -- (\xP, \yP) -- cycle; 
\fill[fill = green, fill opacity = 1, draw = black, very thin] 
(\xQ, \yQ) arc ({asin(\yQ/\radius)} : {asin(\yRa/\radius)} : \radius) -- (\xRb, \yRb) arc ({asin(\yRb/\radius)} : { - \ALPHA} : \radius) -- (\xP, \yP) -- cycle; 
\draw[] ({\xA}, { - \yA}) node[left] {\footnotesize ${P}_{ - }$} arc ( - \ALPHA : \ALPHA : \radius) node[left] {\footnotesize ${P}_{ + }$} -- ( - \aaa, 0) node[left] {\footnotesize ${V}$} -- cycle; 
\draw[thin, dotted] ({\xA}, {\yA}) -- (0, 0) node[left] {\footnotesize ${O}$} -- ({\xA}, { - \yA}); 
\fill[red] ({ - \aaa + \LAM*cos(\BETA)}, { - \LAM*sin(\BETA)}) circle (0.008) node[below] {\footnotesize ${P}_{\lambda}$}; 
\draw[red, thick] (\xP, \yP) -- ({\xP + 1.2*(\xQ - \xP)}, {\yP + 1.2*(\yQ - \yP)}) node[right] {\footnotesize ${T}_{\lambda, \vartheta}$}; 
\pgfmathsetmacro\radB{0.11*\radius}; 
\draw[very thin] ({\xP + \radB*cos(\THETA - \BETA)}, {\yP + \radB*sin(\THETA - \BETA)}) arc ({\THETA - \BETA} : { - \BETA} : \radB); 
\draw ({\xP + \radB*cos(\THETA/2 - \BETA)}, {\yP + \radB*sin(\THETA/2 - \BETA)}) node [right] {\footnotesize $\vartheta = \frac{\pi}{2}$}; 
\pgfmathsetmacro\lenC{1.5*\radius}; 
\draw[very thin] ({\xRa - 0.2*(\xRb - \xRa)}, {\yRa - 0.2*(\yRb - \yRa)}) -- ({\xRa + 1.3*(\xRb - \xRa)}, {\yRa + 1.3*(\yRb - \yRa)}) node[right] {\footnotesize ${T}_{\check{\lambda}, \check{\vartheta}} = \hat{T}_{\hat{\lambda}, \hat{\vartheta}}$}; 
\end{tikzpicture} 
\caption{The moving lines {\color{red}${T}_{\lambda, \pi/2}$} and the moving domains {\color{green}${D}_{\lambda, \pi/2}$}. }
\label{Fig03MMP90}
\end{figure}
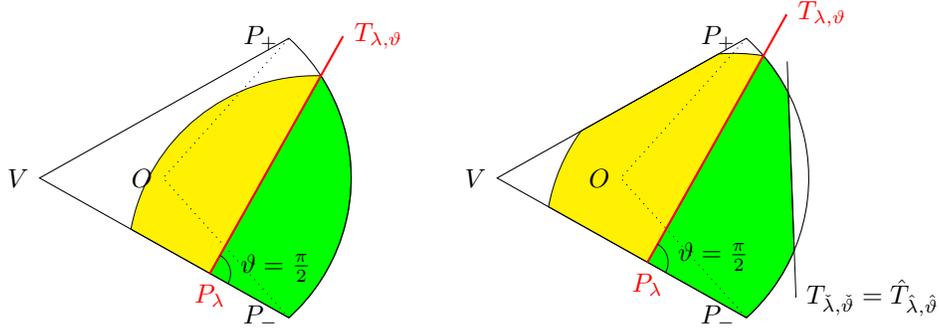

Due to the fact that ${w}^{\lambda, \vartheta}$ may not satisfy certain a priori boundary conditions on $\partial {D}_{\lambda, \vartheta}$, we will instead work with a smaller domain than ${D}_{\lambda, \vartheta}$. To achieve this, we introduce the polar coordinate $\sigma_{\lambda}({x}) \in [0, 2\pi)$ defined by 
\begin{equation} \label{Yao0306}
{x}_{1} - {x}_{1\lambda} = {r}\cos\big(\sigma_{\lambda}({x}) - \frac{\beta}{2}\big), \quad 
{x}_{2} - {x}_{2\lambda} = {r}\sin\big(\sigma_{\lambda}({x}) - \frac{\beta}{2}\big), 
\end{equation}
where ${r} = \sqrt{({x}_{1} - {x}_{1\lambda})^{2} + ({x}_{2} - {x}_{2\lambda})^{2}}$. For $\lambda \in \R^{ + }$, $0 \leq \vartheta_{1} < \vartheta = (\vartheta_{1} + \vartheta_{3})/2 < \vartheta_{3} \leq \pi$, we define the moving domains 
\begin{equation*}
{D}_{\lambda, \vartheta, \vartheta_1} = \left\{ {x} \in \Sigma \cap \Sigma^{\lambda, \vartheta}: \; \vartheta_{1} < \sigma_{\lambda}({x}) < \vartheta \right\}
\end{equation*}
and attempt to prove
\begin{equation} \label{Yao0305c}
{w}^{\lambda, \vartheta}({x}) = {u}({x}) - {u}^{\lambda, \vartheta}({x}) < 0 \quad \text{for} \quad {x} \in {D}_{\lambda, \vartheta, \vartheta_1}
\end{equation}
for suitable choices of $\vartheta$ and $\vartheta_{1}$; see \autoref{Fig03MMP90} and \autoref{Fig03Lines}. 

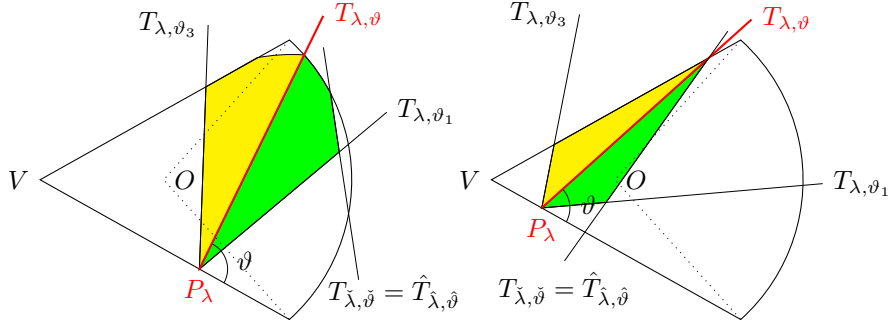
\begin{figure}[htp] \centering
\begin{tikzpicture}[scale = 1.25] 
\pgfmathsetmacro\ALPHA{48.04}; \pgfmathsetmacro\BETA{29.29}; \pgfmathsetmacro\radius{2.00}; 
\pgfmathsetmacro\aaa{\radius*sin(abs(\BETA - \ALPHA))/sin(\BETA)}; 
\pgfmathsetmacro\ka{tan(\BETA)}; 
\pgfmathsetmacro\xA{\radius*cos(\ALPHA)}; \pgfmathsetmacro\yA{\radius*sin(\ALPHA)}; 
\pgfmathsetmacro\THETA{93.24}; \pgfmathsetmacro\THEbb{\THETA - \BETA}; 
\pgfmathsetmacro\THETAa{69.24}; \pgfmathsetmacro\THETAc{2*\THETA - \THETAa}; 
\pgfmathsetmacro\THEaa{\THETAa - \BETA}; \pgfmathsetmacro\THEcc{\THETAc - \BETA}; 
\pgfmathsetmacro\LAM{0.9712*\radius}; \pgfmathsetmacro\Move{0}; 
\pgfmathsetmacro\xP{ - \aaa + \LAM*cos(\BETA)}; \pgfmathsetmacro\yP{0 - \LAM*sin(\BETA)}; 
\pgfmathsetmacro\kb{tan(\THETA - \BETA)}; \pgfmathsetmacro\bb{\yP - \xP*\kb}; 
\pgfmathsetmacro\xVv{ - \aaa*cos(2*\THEbb) - \bb*sin(2*\THEbb)}; 
\pgfmathsetmacro\yVv{ - \aaa*sin(2*\THEbb) + 2*\bb*(cos(\THEbb))^2}; 
\pgfmathsetmacro\Coefa{1}; 
\pgfmathsetmacro\Coefb{(\yP*cos(\THEbb) - \xP*sin(\THEbb))*sin(\THEbb)}; 
\pgfmathsetmacro\Coefc{(\yP*cos(\THEbb) - \xP*sin(\THEbb))^2 - \radius^2*(cos(\THEbb))^2}; 
\pgfmathsetmacro\xQ{( - \Coefb + sqrt((\Coefb)^2 - \Coefa*\Coefc))/\Coefa}; 
\pgfmathsetmacro\yQ{\yP + (\xQ - \xP)*tan(\THEbb)}; 
\pgfmathsetmacro\THETAd{2*\THEbb - \BETA}; 
\pgfmathsetmacro\Coefa{1}; 
\pgfmathsetmacro\Coefb{(\yVv*cos(\THETAd) - \xVv*sin(\THETAd))*sin(\THETAd)}; 
\pgfmathsetmacro\Coefc{(\yVv*cos(\THETAd) - \xVv*sin(\THETAd))^2 - \radius^2*(cos(\THETAd))^2}; 
\pgfmathsetmacro\xR{( - \Coefb - sqrt((\Coefb)^2 - \Coefa*\Coefc))/\Coefa}; 
\pgfmathsetmacro\yR{\yVv + (\xR - \xVv)*tan(\THETAd)}; 
\pgfmathsetmacro\xSS{( - tan(\BETA)*\aaa + (\yP - \xP*tan(\THETAc - \BETA)))/(tan(\BETA) - tan(\THETAc - \BETA))}; 
\pgfmathsetmacro\ySS{(tan(\BETA)*(\xSS + \aaa)}; 
\pgfmathsetmacro\xS{\xSS*cos(2*\THEbb) + \ySS*sin(2*\THEbb) - \bb*sin(2*\THEbb)}; 
\pgfmathsetmacro\yS{\xSS*sin(2*\THEbb) - \ySS*cos(2*\THEbb) + 2*\bb*(cos(\THEbb))^2}; 
\fill[fill = green, draw = black, very thin] 
(\Move + \xQ, \yQ) arc (asin(\yQ/\radius) : asin(\yR/\radius) : \radius) -- (\Move + \xS, \yS) -- (\Move + \xP, \yP) -- cycle; 
\fill[fill = yellow, draw = black, very thin] (\Move + \xQ, \yQ) arc ({2*\THEbb - asin(\yQ/\radius)} : {2*\THEbb - asin(\yR/\radius)} : \radius)
 -- (\Move + \xSS, \ySS) -- (\Move + \xP, \yP) -- cycle; 
\draw[] ({\Move + \xA}, {0 - \yA}) arc ( - \ALPHA : \ALPHA : \radius) -- (\Move - \aaa, 0) node[left] {\footnotesize ${V}$} -- cycle; 
\draw[thin, dotted] ({\Move + \xA}, {0 + \yA}) -- (\Move + 0, 0) node[right] {\footnotesize ${O}$} -- ({\Move + \xA}, {0 - \yA}); 
\fill[red] ({\Move - \aaa + \LAM*cos(\BETA)}, {0 - \LAM*sin(\BETA)}) circle (0.008) node[below] {\footnotesize ${P}_{\lambda}$}; 
\pgfmathsetmacro\radB{0.15*\radius}; \pgfmathsetmacro\lenC{1.5*\radius}; 
\draw[very thin] ({\Move + \xP + \radB*cos(\THETA - \BETA)}, {0 + \yP + \radB*sin(\THETA - \BETA)}) arc ({\THETA - \BETA} : { - \BETA} : \radB); 
\draw ({\Move + \xP + \radB*cos(\THETA/2 - \BETA)}, {0 + \yP + \radB*sin(\THETA/2 - \BETA)}) node [right] {\footnotesize $\vartheta$}; 
\draw[very thin] (\Move + \xP, \yP) -- ++ ({\THETAa - \BETA} : {1.3*\radius}) node[right] {\footnotesize ${T}_{\lambda, \vartheta_{1}}$}; 
\draw[very thin] (\Move + \xP, \yP) -- ++ ({\THETAc - \BETA} : {1.3*\radius}) node[left] {\footnotesize ${T}_{\lambda, \vartheta_{3}}$}; 
\draw[red, thick] (\Move + \xP, \yP) -- ++ ({\THETA - \BETA} : {1.5*\radius}) node[right] {\footnotesize ${T}_{\lambda, \vartheta}$}; 
\draw[very thin] ({\xR - 0.8*(\xS - \xR)}, {\yR - 0.8*(\yS - \yR)}) -- ({\xR + 3.2*(\xS - \xR)}, {\yR + 3.2*(\yS - \yR)}) node[below = 1ex, right = - 3ex] {\footnotesize ${T}_{\check{\lambda}, \check{\vartheta}} = \hat{T}_{\hat{\lambda}, \hat{\vartheta}}$}; 
\pgfmathsetmacro\THETA{71.24}; \pgfmathsetmacro\THEbb{\THETA - \BETA}; 
\pgfmathsetmacro\THETAa{34.24}; \pgfmathsetmacro\THETAc{2*\THETA - \THETAa}; 
\pgfmathsetmacro\THEaa{\THETAa - \BETA}; \pgfmathsetmacro\THEcc{\THETAc - \BETA}; 
\pgfmathsetmacro\LAM{0.3067*\radius}; \pgfmathsetmacro\Move{2.4*\radius}; 
\pgfmathsetmacro\xP{ - \aaa + \LAM*cos(\BETA)}; \pgfmathsetmacro\yP{0 - \LAM*sin(\BETA)}; 
\pgfmathsetmacro\kb{tan(\THETA - \BETA)}; \pgfmathsetmacro\bb{\yP - \xP*\kb}; 
\pgfmathsetmacro\xVv{ - \aaa*cos(2*\THEbb) - \bb*sin(2*\THEbb)}; 
\pgfmathsetmacro\yVv{ - \aaa*sin(2*\THEbb) + 2*\bb*(cos(\THEbb))^2}; 
\pgfmathsetmacro\xQ{( - tan(\BETA)*\aaa + (\yP - \xP*tan(\THETA - \BETA)))/(tan(\BETA) - tan(\THETA - \BETA))}; 
\pgfmathsetmacro\yQ{(tan(\BETA)*(\xQ + \aaa)}; 
\pgfmathsetmacro\xSS{( - tan(\BETA)*\aaa + (\yP - \xP*tan(\THETAc - \BETA)))/(tan(\BETA) - tan(\THETAc - \BETA))}; 
\pgfmathsetmacro\ySS{(tan(\BETA)*(\xSS + \aaa)}; 
\pgfmathsetmacro\xS{\xSS*cos(2*\THEbb) + \ySS*sin(2*\THEbb) - \bb*sin(2*\THEbb)}; 
\pgfmathsetmacro\yS{\xSS*sin(2*\THEbb) - \ySS*cos(2*\THEbb) + 2*\bb*(cos(\THEbb))^2}; 
\fill[fill = green, draw = black, very thin] (\Move + \xQ, \yQ) -- (\Move + \xS, \yS) -- (\Move + \xP, \yP) -- cycle; 
\fill[fill = yellow, draw = black, very thin] (\Move + \xQ, \yQ) -- (\Move + \xSS, \ySS) -- (\Move + \xP, \yP) -- cycle; 
\draw[] ({\Move + \xA}, { - \yA}) arc ( - \ALPHA : \ALPHA : \radius) -- (\Move - \aaa, 0) node[left] {\footnotesize ${V}$} -- cycle; 
\draw[thin, dotted] ({\Move + \xA}, {0 + \yA}) -- (\Move + 0, 0) node[right] {\footnotesize ${O}$} -- ({\Move + \xA}, {0 - \yA}); 
\fill[red] ({\Move - \aaa + \LAM*cos(\BETA)}, {0 - \LAM*sin(\BETA)}) circle (0.008) node[below] {\footnotesize ${P}_{\lambda}$}; 
\pgfmathsetmacro\radB{0.15*\radius}; \pgfmathsetmacro\lenC{1.5*\radius}; 
\draw[very thin] ({\Move + \xP + \radB*cos(\THETA - \BETA)}, {0 + \yP + \radB*sin(\THETA - \BETA)}) arc ({\THETA - \BETA} : { - \BETA} : \radB); 
\draw ({\Move + \xP + \radB*cos(\THETA/2 - \BETA)}, {0 + \yP + \radB*sin(\THETA/2 - \BETA)}) node [right] {\footnotesize $\vartheta$}; 
\draw[very thin] (\Move + \xP, \yP) -- ++ ({\THETAa - \BETA} : \lenC) node[right] {\footnotesize ${T}_{\lambda, \vartheta_{1}}$}; 
\draw[very thin] (\Move + \xP, \yP) -- ++ ({\THETAc - \BETA} : \lenC*0.7) node[left] {\footnotesize ${T}_{\lambda, \vartheta_{3}}$}; 
\draw[red, thick] (\Move + \xP, \yP) -- ++ ({\THETA - \BETA} : \lenC) node[right] {\footnotesize ${T}_{\lambda, \vartheta}$}; 
\draw[very thin] (\Move + \xVv, \yVv) node[below] {\footnotesize ${T}_{\check{\lambda}, \check{\vartheta}} = \hat{T}_{\hat{\lambda}, \hat{\vartheta}}$} -- ++ ({2*\THETA - 3*\BETA} : \lenC); 
\end{tikzpicture} 
\caption{The moving lines {\color{red}${T}_{\lambda, \vartheta}$} and the moving domains {\color{green}$D_{\lambda, \vartheta, \vartheta_{1}}$}. }
\label{Fig03Lines}
\end{figure}

It is clear that ${D}_{\lambda, \vartheta, \vartheta_1}$ is nonempty if and only if ${T}_{\lambda, \vartheta} \cap \Sigma \neq \emptyset$, that is, $0 < \lambda < \lambda_{M}(\vartheta)$, where
\begin{equation*}
\begin{aligned}
\lambda_{M}(\vartheta) & = \sup\{\lambda > 0: \; {T}_{\lambda, \vartheta} \cap \Sigma \neq \emptyset\}, \quad \vartheta \in ( 0, \, \tfrac{\pi + \beta}{2} ], \\
\lambda_{\mathrm{max}} & = \sup\{ \lambda_{M}(\vartheta): \; \vartheta \in ( 0, \tfrac{\pi + \beta}{2} ] \} = (1 - {a})\sec\tfrac{\beta}{2}. 
\end{aligned}
\end{equation*}
The boundary $\partial {D}_{\lambda, \vartheta, \vartheta_{1}}$ consists of three parts: 
\begin{enumerate}[label = {\rm(\arabic*)}]
\item
$\Gamma_{\lambda, \vartheta, \vartheta_{1}}^{0} = {T}_{\lambda, \vartheta} \cap \partial {D}_{\lambda, \vartheta, \vartheta_{1}}$, which is always nonempty for $0 < \lambda < \lambda_{M}(\vartheta)$. 
\item
$\Gamma_{\lambda, \vartheta, \vartheta_{1}}^{1} = \big(\partial {D}_{\lambda, \vartheta, \vartheta_{1}} \setminus {T}_{\lambda, \vartheta}\big) \cap \big(\Gamma_{D} \cup \Gamma_{D}^{\lambda, \vartheta}\big)$, representing the spherical boundary. 
\item
$\Gamma_{\lambda, \vartheta, \vartheta_{1}}^{2} = \partial {D}_{\lambda, \vartheta, \vartheta_{1}} \setminus \big(\Gamma_{\lambda, \vartheta, \vartheta_{1}}^{0} \cup \Gamma_{\lambda, \vartheta, \vartheta_{1}}^{1}\big)$, the flat boundary, which consists of two pieces: 
$\Gamma^{2A}_{\lambda, \vartheta, \vartheta_{1}} = \Gamma_{\lambda, \vartheta, \vartheta_{1}}^{2} \cap {T}_{\lambda, \vartheta_{1}}$ and
$\Gamma^{2B}_{\lambda, \vartheta, \vartheta_{1}} = \Gamma_{\lambda, \vartheta, \vartheta_{1}}^{2} \cap {T}_{\check{\lambda}, \check{\vartheta}}$. 
\end{enumerate}
See \autoref{Fig03Lines}. Here, ${T}_{\check{\lambda}, \check{\vartheta}} = \hat{T}_{\hat{\lambda}, \hat{\vartheta}}$ denotes the line containing the reflection of the upper Neumann boundary $\Gamma_{N}^{ + }$ with respect to ${T}_{\lambda, \vartheta}$, where $\check{\vartheta} = 2\vartheta - \beta$, $\hat{\vartheta} = \pi - 2\vartheta + 2\beta$, and
\begin{equation} \label{Yao0308}
\hat{\lambda} = \hat{\lambda}(\lambda, \vartheta) = \frac{\lambda\sin\vartheta}{\sin(\vartheta - \beta)}, \quad
\check{\lambda} = \check{\lambda}(\lambda, \vartheta) = \lambda + \frac{\lambda\sin\beta}{\sin(2\vartheta - \beta)}. 
\end{equation}
Note that $\hat\lambda > \check{\lambda}$ if and only if $\hat{\vartheta} < \check{\vartheta}$, i.e., $\vartheta > (\pi + 3\beta)/4$. Observe that $\Gamma_{\lambda, \vartheta, \vartheta_{1}}^{2A} \cap \Gamma_{\lambda, \vartheta, \vartheta_{1}}^{2B}$ contains at most one point, and both $\Gamma_{\lambda, \vartheta, \vartheta_{1}}^{2A} \setminus \Gamma_{\lambda, \vartheta, \vartheta_{1}}^{2B}$ and $\Gamma_{\lambda, \vartheta, \vartheta_{1}}^{2B} \setminus \Gamma_{\lambda, \vartheta, \vartheta_{1}}^{2A}$ are open line segments. Clearly, $\Gamma_{\lambda, \vartheta, \vartheta_{1}}^{1} \cup \Gamma_{\lambda, \vartheta, \vartheta_{1}}^{2B}$ is always nonempty for $0 < \lambda < \lambda_{M}(\vartheta)$. 

In order to guarantee
\begin{equation} \label{Yao0309}
{w}^{\lambda, \vartheta} \leq 0 \quad \text{on } \Gamma_{\lambda, \vartheta, \vartheta_{1}}^{1}, 
\end{equation}
it suffices to require that the spherical boundary $\Gamma_{\lambda, \vartheta, \vartheta_{1}}^{1}$ is contained within the sphere $\partial {B}$. In other words, the upper mixed boundary point ${P}_{ + }$ and the center ${O}$ must not lie on opposite sides of ${T}_{\lambda, \vartheta}$ for small $\lambda$. 
Thus, for $\lambda > 0$, we define two important angles $\vartheta_{A}(\lambda)$ and $\vartheta_{B}(\lambda)$, and a critical value $\lambda_{C}$ such that: 
\begin{enumerate}[label = {\rm(\arabic*)}]
\item
The center ${O}$ belongs to ${T}_{\lambda, \vartheta}$ if and only if $\vartheta = \vartheta_{A}(\lambda)$; 
\item
The upper mixed boundary point ${P}_{ + }$ belongs to ${T}_{\lambda, \vartheta}$ if and only if $\vartheta = \vartheta_{B}(\lambda)$; 
\item
${T}_{\lambda, \vartheta}$ passes through both $O$ and ${P}_{ + }$ if and only if $\vartheta = (\beta + \alpha)/2$ and $\lambda = \lambda_{C}$. 
\end{enumerate}
Here, $\vartheta_{A}(\lambda)$ and $\vartheta_{B}(\lambda)$ depend on $\lambda$ and are strictly increasing for $\lambda \in (0, \infty)$, and 
\begin{align}
\label{Yao0310a}
\vartheta_{A}(\lambda) & = 
\operatorname{arccot}\Big(\tfrac{|a|\cos(\beta/2) - \lambda}{|a|\sin(\beta/2)}\Big), \\
\label{Yao0310b}
\vartheta_{B}(\lambda) & = 
\operatorname{arccot}\Big(\tfrac{{l}_{N}\cos(\beta) - \lambda}{{l}_{N}\sin\beta}\Big), \\
\lambda_{C} & = \tfrac{|a|\sin(\alpha/2)}{\sin((\beta + \alpha)/2)}, 
\end{align}
where ${l}_{N}$ is the length of the lower Neumann boundary $\Gamma_{N}^{ - }$, that is, 
\begin{equation*}
{l}_{N} = \sin\frac{\alpha}{2}\csc\frac{\beta}{2}. 
\end{equation*}
Thus, we always assume $\vartheta \in {J}_{\lambda}$, where the set ${J}_{\lambda}$ is defined by 
\begin{equation} \label{Yao0312}
{J}_{\lambda} = 
\begin{cases}
(0, \, \frac{\pi + \beta}{2}] & \text{if } \lambda \in [\lambda_{C}, \infty), \\
(0, \, \vartheta_{A}(\lambda)] \cup [\vartheta_{B}(\lambda), \, \frac{\pi + \beta}{2}] & \text{if } \lambda \in (0, \lambda_{C}). 
\end{cases}
\end{equation}
For convenience, we also set
\begin{equation} \label{Yao0314}
\begin{aligned}
\vartheta^{A}(\lambda) & = \min\big\{\vartheta_{A}(\lambda), \, \vartheta_{B}(\lambda), \, \tfrac{\pi + \beta}{2}\big\}, 
\\
\vartheta^{B}(\lambda) & = \min\big\{\vartheta_{B}(\lambda), \, \tfrac{\pi + \beta}{2}\big\}, 
\end{aligned}
\end{equation}
so that
\begin{equation*}
{J}_{\lambda} = (0, \, \vartheta^{A}(\lambda)] \cup [\vartheta^{B}(\lambda), \, \frac{\pi + \beta}{2}] \quad \text{for } \lambda \in (0, \infty). 
\end{equation*}
Clearly, $\bigcap_{\lambda > 0} {J}_{\lambda} = (0, \beta/2] \cup [(\alpha + \beta)/2, (\pi + \beta)/2]$. Hence, \eqref{Yao0309} is fulfilled if $\vartheta \in {J}_{\lambda}$. 
Moreover, ${w}^{\lambda, \vartheta}$ satisfies the linear boundary value problem 
\begin{align*}
\Delta {w}^{\lambda, \vartheta} + {c}^{\lambda, \vartheta}({x}){w}^{\lambda, \vartheta} = 0
& \quad \text{in } {D}_{\lambda, \vartheta, \vartheta_{1}}, 
\\
{w}^{\lambda, \vartheta} = 0 & \quad \text{on } \Gamma_{\lambda, \vartheta, \vartheta_{1}}^{0}, 
\\
{w}^{\lambda, \vartheta} < 0 & \quad \text{on } \Gamma_{\lambda, \vartheta, \vartheta_{1}}^{1} \text{ for } \vartheta \neq \vartheta_{A} \text{ or } \lambda \geq \lambda_{C}, 
\\
{w}^{\lambda, \vartheta} = 0 & \quad \text{ on } \Gamma_{\lambda, \vartheta, \vartheta_{1}}^{1} \text{ for } \vartheta = \vartheta_{A}, \; \lambda < \lambda_{C}. 
\end{align*}
Here the coefficient
\begin{equation*}
{c}^{\lambda, \vartheta}({x}) = 
\begin{cases}
\dfrac{{f}({u}^{\lambda, \vartheta}({x})) - {f}({u}({x}))}
{{u}^{\lambda, \vartheta}({x}) - {u}({x})}, 
& {w}^{\lambda, \vartheta}({x}) \neq 0, 
\\
0, 
& {w}^{\lambda, \vartheta}({x}) = 0, 
\end{cases}
\end{equation*}
is uniformly bounded in $\lambda, \vartheta$, say, $|{c}^{\lambda, \vartheta}| < {c}_{0}$, where ${c}_{0} > 0$ is the Lipschitz constant of ${f}$ on $[0, \sup_{\Sigma} {u}]$. 

Similarly to ${T}_{\lambda, \vartheta}$, we denote by $\hat{T}_{\lambda, \vartheta}$ the line
\begin{equation*}
\hat{T}_{\lambda, \vartheta} = \left\{ {x} \in \R^2: \; 
({x}_{1} - {x}_{1\lambda})\sin(\vartheta - \tfrac{\beta}{2}) - ({x}_{2} + {x}_{2\lambda})\cos(\vartheta - \tfrac{\beta}{2}) = 0 \right\}. 
\end{equation*}
The aim is to prove both \eqref{Yao0305b}, \eqref{Yao0305c} (for suitable $\vartheta_{1}$), and
\begin{subequations} \label{Yao0315}
\begin{align}
\label{Yao0315a}
{u}_{{x}_{1}}\sin(\vartheta - \tfrac{\beta}{2}) - {u}_{{x}_{2}}\cos(\vartheta - \tfrac{\beta}{2}) < 0 & \quad \text{on } {T}_{\lambda, \vartheta} \cap \Sigma, \\
\label{Yao0315b}
{u}_{{x}_{1}}\sin(\vartheta - \tfrac{\beta}{2}) + {u}_{{x}_{2}}\cos(\vartheta - \tfrac{\beta}{2}) < 0 & \quad \text{on } \hat{T}_{\lambda, \vartheta} \cap \Sigma, 
\end{align}
\end{subequations}
for $\lambda > 0$ and $\vartheta \in {J}_{\lambda}$. 
The proof of \eqref{Yao0315b} is analogous to that of \eqref{Yao0315a}; the proof of \eqref{Yao0315a} relies on \eqref{Yao0315b}, which is why both are included here. The sign of these directional derivatives in the interior of $\Sigma$ follows directly from \eqref{Yao0305c} and the Hopf boundary lemma. We note that the strict monotonicity property stated in \eqref{Yao0109a} on the Neumann boundary $\Gamma_{N}$ presents a challenging but crucial aspect for the proof; a detailed argument can be found in \autoref{lma307} below. Building on this, we aim to establish the monotonicity of the solution across both $\Sigma$ and $\Gamma_{N}$ as follows: 
\begin{subequations} \label{Yao0316}
\begin{align}
\label{Yao0316a}
{u}_{{x}_{1}}\sin(\vartheta - \tfrac{\beta}{2}) - {u}_{{x}_{2}}\cos(\vartheta - \tfrac{\beta}{2}) < 0 & \quad \text{on } {T}_{\lambda, \vartheta} \cap (\Sigma \cup \Gamma_{N}^{ - }), \\
\label{Yao0316b}
{u}_{{x}_{1}}\sin(\vartheta - \tfrac{\beta}{2}) + {u}_{{x}_{2}}\cos(\vartheta - \tfrac{\beta}{2}) < 0 & \quad \text{on } \hat{T}_{\lambda, \vartheta} \cap (\Sigma \cup \Gamma_{N}^{ + }). 
\end{align}
\end{subequations}

In what follows, we shall omit the subscript $\vartheta_{1}$ when $\vartheta_{1} = \max\{2\vartheta - \pi, 0\}$ or when there is no confusion. Thus, $\Gamma_{\lambda, \vartheta, \vartheta_{1}}^{0}$, $\Gamma_{\lambda, \vartheta, \vartheta_{1}}^{1}$, $\Gamma_{\lambda, \vartheta, \vartheta_{1}}^{2}$, $\Gamma_{\lambda, \vartheta, \vartheta_{1}}^{2A}$, and $\Gamma_{\lambda, \vartheta, \vartheta_{1}}^{2B}$ will be abbreviated as $\Gamma_{\lambda, \vartheta}^{0}$, $\Gamma_{\lambda, \vartheta}^{1}$, $\Gamma_{\lambda, \vartheta}^{2}$, $\Gamma_{\lambda, \vartheta}^{2A}$, and $\Gamma_{\lambda, \vartheta}^{2B}$, respectively. Likewise, the functions $\vartheta^{A}(\lambda)$, $\vartheta^{B}(\lambda)$, $\vartheta_{A}(\lambda)$, and $\vartheta_{B}(\lambda)$ will be shortened to $\vartheta^{A}$, $\vartheta^{B}$, $\vartheta_{A}$, and $\vartheta_{B}$ whenever the dependence on \(\lambda\) is clear from context. 


\subsection{The moving plane method with a priori boundary condition}

\begin{lma} \label{lma301}
Let $\{ \mu_{{t}} \}$, $\{ \phi_{1, {t}} \}$, $\{ \phi_{2, {t}} \}$, $\{ \phi_{3, {t}} \}$, ${t} \in [0, \infty)$, be continuous families of real parameters such that: 
\begin{enumerate}[label = {\rm(A\arabic*)}]
\item \label{Yaoit32a}
$\mu_{{t}} > 0$, $0 \leq \phi_{{i}, {t}} \leq \pi$ for ${i} = 1, 2, 3$, and $\phi_{3, {t}} - \phi_{2, {t}} = \phi_{2, {t}} - \phi_{1, {t}} > 0$. 
\item \label{Yaoit32b}
${T}_{\mu_{{t}}, \phi_{2, {t}}} \cap \Sigma \neq \emptyset$ for ${t} \in [0, 1)$.
\item \label{Yaoit32c} 
one of the following holds: 
\begin{enumerate}[label = {\rm(A3\alph*)}]
\item \label{Yaoit32c1}
${T}_{\mu_{{t}}, \phi_{2, {t}}} \cap \Sigma = \emptyset$ for all ${t} \geq 1$, and $\operatorname{diam}({D}_{\mu_{{t}}, \phi_{2, {t}}, \phi_{1, {t}}}) \to 0$ as ${t} \nearrow 1$. 
\item \label{Yaoit32c2}
${T}_{\mu_{{t}}, \phi_{2, {t}}} \cap \Sigma \neq \emptyset$ and ${w}^{\mu_{{t}}, \phi_{2, {t}}} < 0$ in ${D}_{\mu_{{t}}, \phi_{2, {t}}, \phi_{1, {t}}}$ for ${t} = 1$. 
\end{enumerate}
\item \label{Yaoit32d}
${w}^{\lambda, \vartheta}$ satisfies, for $\vartheta = \phi_{2, {t}}$, $\vartheta_{1} = \phi_{1, {t}}$, $\lambda = \mu_{{t}}$, and all ${t} \in [0, 1)$, 
\begin{subequations} \label{Yao0322}
\begin{align}
&\Delta {w}^{\lambda, \vartheta} + {c}^{\lambda, \vartheta}({x}) {w}^{\lambda, \vartheta} = 0 \quad \text{in } {D}_{\lambda, \vartheta, \vartheta_{1}}, \\
&{w}^{\lambda, \vartheta} = 0 \quad \text{on } \Gamma_{\lambda, \vartheta, \vartheta_{1}}^{0}, \\
\label{Yao0322c}
&{w}^{\lambda, \vartheta} \leq 0 \quad \text{on } \Gamma_{\lambda, \vartheta, \vartheta_{1}}^{1}, \\
\label{Yao0322d}
&\nabla {w}^{\lambda, \vartheta} \cdot \nu \leq 0 \quad \text{on } \Gamma_{\lambda, \vartheta, \vartheta_{1}}^{2A}, \\
\label{Yao0322e}
&\nabla {w}^{\lambda, \vartheta} \cdot \nu \leq 0 \quad \text{on } \Gamma_{\lambda, \vartheta, \vartheta_{1}}^{2B}, \\
&\text{one of \eqref{Yao0322c} and \eqref{Yao0322e} holds strictly at some point}. 
\end{align}
\end{subequations}
\end{enumerate}
Then, for every ${t} \in [0, 1)$, one has
\begin{equation}\label{Yao0323}
\begin{gathered}
{w}^{\mu_{{t}}, \phi_{2, {t}}} < 0 \quad \text{in } {D}_{\mu_{{t}}, \phi_{2, {t}}, \phi_{1, {t}}}, 
\\ {u}_{{x}_{1}}\sin(\phi_{2, {t}} - \tfrac{\beta}{2}) - {u}_{{x}_{2}}\cos(\phi_{2, {t}} - \tfrac{\beta}{2}) < 0 \quad \text{on } {T}_{\mu_{{t}}, \phi_{2, {t}}} \cap \Sigma. 
\end{gathered}
\end{equation}
\end{lma}

\begin{proof}
We shall apply the maximum principle for the mixed problem to derive the negativity of ${w}^{\mu_{{t}}, \phi_{2, {t}}}$ in ${D}_{\mu_{{t}}, \phi_{2, {t}}, \phi_{1, {t}}}$. 

\textbf{Step 1}: 
The initiation of the moving plane method.

By the definition of ${D}_{\lambda, \vartheta, \vartheta_{1}}$, we have $\nu \cdot \mathbf{e}_{\vartheta - \beta/2 - \pi/2} \geq 0$ on $\Gamma_{\lambda, \vartheta, \vartheta_{1}}^{2}$. Let $\eta$ be the small constant in \autoref{lma201MP}. Suppose the case \ref{Yaoit32c1} occurs. By assumptions \ref{Yaoit32a}, \ref{Yaoit32b} and \ref{Yaoit32c1}, there exists $\delta_{0} > 0$ such that for every ${t} \in (1 - \delta_{0}, 1)$, 
\begin{equation*} 
{D}_{\mu_{t}, \phi_{2, {t}}, \phi_{1, {t}}} \subset \{{x}: 0 < ({{x} - {P}_{\mu_{t}}}) \cdot \mathbf{e}_{\phi_{2, {t}} - \beta/2 - \pi/2} < \eta\}.
\end{equation*}
Hence by the maximum principle in \autoref{lma201MP}, one deduces the negativity of ${w}^{\mu_{t}, \phi_{2, {t}}}$ in ${D}_{\mu_{t}, \phi_{2, {t}}, \phi_{1, {t}}}$. It follows that \eqref{Yao0323} holds for all ${t} \in (1 - \delta_{0}, 1)$. Set 
\begin{equation*}
\bar{t} = \inf\{{t} ' \in [0, 1]: \; \text{\eqref{Yao0323} holds for every ${t} \in [{t} ', 1]$}\}. 
\end{equation*}
It follows that $\bar{t}$ is well-defined, with $\bar{t} < 1$ under \ref{Yaoit32c1} and $\bar{t} \leq 1$ under \ref{Yaoit32c2}. 

To obtain the desired result, we proceed by contradiction and suppose that $\bar{t} > 0$. 

\textbf{Step 2}: 
We conclude that \eqref{Yao0323} holds for ${t} = \bar{t}$. 

If $\bar{t} = 1$ the conclusion is trivial, so assume $\bar{t} < 1$. By continuity, 
\begin{equation*}
{w}^{\mu_{t}, \phi_{2, {t}}}({x}) \leq 0 \text{ for } {x} \in {D}_{\mu_{t}, \phi_{2, {t}}, \phi_{1, {t}}}
\end{equation*}
at ${t} = \bar{t}$. Since the boundary condition of ${w}^{\mu_{t}, \phi_{2, {t}}}$ holds strictly at some point on $\partial {D}_{\mu_{t}, \phi_{2, {t}}, \phi_{1, {t}}}$, the strong maximum principle and Hopf lemma imply 
\begin{equation} \label{Yao0326a}
{w}^{\mu_{\bar{t}}, \phi_{2, \bar{t}}}({x}) < 0 \text{ for } {x} \in {D}_{\mu_{\bar{t}}, \phi_{2, \bar{t}}, \phi_{1, \bar{t}}}
\end{equation}
and
\begin{equation*}
{u}_{{x}_{1}}\sin(\phi_{2, \bar{t}} - \tfrac{\beta}{2}) - {u}_{{x}_{2}}\cos(\phi_{2, \bar{t}} - \tfrac{\beta}{2}) < 0 \text{ on } {T}_{\mu_{\bar{t}}, \phi_{2, \bar{t}}} \cap \Sigma. 
\end{equation*}

\textbf{Step 3}: 
We claim that \eqref{Yao0323} holds whenever $\bar{t} - {t} > 0$ is sufficiently small. 

Observe that ${D}_{\mu_{{t}}, \phi_{2, {t}}, \phi_{1, {t}}}$ lies in an infinite sector $\mathcal{A}^{t}$ of angle $\pi + \beta - \phi_{3, {t}} \geq \beta$, and its Neumann boundary $\Gamma_{\mu_{{t}}, \phi_{2, {t}}, \phi_{1, {t}}}^{2A}$ and $\Gamma_{\mu_{{t}}, \phi_{2, {t}}, \phi_{1, {t}}}^{2B}$ lie on $\partial \mathcal{A}^{t}$. Let $\eta$ be the small constant in \autoref{lma205MP}. 
Fix a compact set ${K}$ of ${D}_{\mu_{\bar{t}}, \phi_{2, \bar{t}}, \phi_{1, \bar{t}}}$ so that 
\begin{equation*}
|{D}_{\mu_{\bar{t}}, \phi_{2, \bar{t}}, \phi_{1, \bar{t}}} \setminus {K}| < \eta \beta/2. 
\end{equation*}
By \eqref{Yao0326a} and compactness, ${w}^{\mu_{\bar{t}}, \phi_{2, \bar{t}}} < 0$ in the compact set ${K}$. Hence by continuity there is a small constant $\epsilon_{0} > 0$ such that for all ${t} \in [\bar{t} - \epsilon_{0}, \bar{t}]$, 
\begin{equation*}
|{D}_{\mu_{t}, \phi_{2, {t}}, \phi_{1, {t}}} \setminus {K}| < \eta \beta, \quad {w}^{\mu_{t}, \phi_{2, {t}}} < 0 \text{ in } {K}. 
\end{equation*}

In the remain domain $\mathcal{D}_{t} = {D}_{\mu_{t}, \phi_{2, {t}}, \phi_{1, {t}}} \setminus {K}$, we have
\begin{equation*}
\begin{cases}
\Delta {w}^{\mu_{t}, \phi_{2, {t}}} + {c}^{\mu_{t}, \phi_{2, {t}}}{w}^{\mu_{t}, \phi_{2, {t}}} = 0 & \text{in } \mathcal{D}_{t}, \\
\partial_{\nu}{w}^{\mu_{t}, \phi_{2, {t}}} \leq 0 & \text{on } 
\Gamma_{\mu_{t}, \phi_{2, {t}}, \phi_{1, {t}}}^{2} \subset \partial \mathcal{D}_{t}, \\
{w}^{\mu_{t}, \phi_{2, {t}}} \leq, \not\equiv0 & \text{on } 
\partial \mathcal{D}_{t} \setminus \Gamma_{\mu_{t}, \phi_{2, {t}}, \phi_{1, {t}}}^{2}. 
\end{cases}
\end{equation*}
Applying \autoref{lma205MP} to ${w}^{\mu_{t}, \phi_{2, {t}}}$ in $\mathcal{D}_{t}$ we infer that ${w}^{\mu_{t}, \phi_{2, {t}}} < 0$ in $\mathcal{D}_{t}$. Hence, \eqref{Yao0323} holds for ${t} \in [\bar{t} - \epsilon_{0}, \bar{t}]$. This finishes this step. 

The above three steps contradict the minimality of $\bar{t} > 0$, so $\bar{t} = 0$. Therefore \eqref{Yao0323} holds for all ${t} \in [0, 1)$. 
\end{proof}

To simplify terminology, we say that ${w}^{\lambda, \vartheta, \vartheta_{1}}$ satisfies \eqref{Yao0322} whenever the conditions in \eqref{Yao0322} hold. We say that ${w}^{\lambda, \vartheta, \vartheta_{1}}$ is \emph{negative} if ${w}^{\lambda, \vartheta} < 0$ in ${D}_{\lambda, \vartheta, \vartheta_{1}}$. 


\subsection{The case for $\lambda \geq \lambda_{C}$}

For $\lambda > 0$ and $\vartheta \in \bigl(\beta/2, (\pi + \beta)/2\bigr]$, the functions $\hat{\lambda}$ and $\check{\lambda}$ are defined in \eqref{Yao0308}. We now define $\lambda_{*} = \check{\lambda}$ when $\vartheta \in (\beta/2, (\pi + 3\beta)/4]$, and $\lambda_{*} = \hat{\lambda}$ when $\vartheta \in [(\pi + 3\beta)/4, (\pi + \beta)/2]$. That is
\begin{equation*}
\lambda_{*} = \lambda_{*}(\lambda, \vartheta) = \lambda/\zeta(\vartheta)
\end{equation*}
where
\begin{equation*}
\zeta(\vartheta) = 
\begin{cases}
\frac{\sin(2\vartheta - \beta)}{\sin\beta + \sin(2\vartheta - \beta)} & \text{ for } \vartheta \in [\frac{\beta}{2}, \frac{\pi + 3\beta}{4}], \\
\frac{\sin(\vartheta - \beta)}{\sin\vartheta} & \text{ for } \vartheta \in [\frac{\pi + 3\beta}{4}, \frac{\pi + \beta}{2}], 
\end{cases}
\end{equation*}
see the function $\zeta(\vartheta)$ in \autoref{Fig04Zeta}. Note that for $\beta < \vartheta \leq (\pi + \beta)/2$, we have $\lambda_{*}(\lambda, \vartheta) = \min\{\hat{\lambda}(\lambda, \vartheta), \check{\lambda}(\lambda, \vartheta)\}$. We denote 
\begin{equation} \label{Yao0332zeta}
\beta_{\flat} = \max\Big\{ \frac{\pi + 2\beta}{3}, \, \frac{\pi}{2} \Big\}, \quad \zeta_{\flat} = \max\Big\{\zeta(\vartheta): \; \vartheta \in \big[\frac{\beta}{2}, \, \beta_{\flat}\big]\Big\}
\end{equation}
and
\begin{equation*}
\lambda_{\flat} = \{\lambda' > 0: \; (0, \beta_{\flat}] \subset {J}_{\lambda} \text{ for all } \lambda \geq \lambda'\}. 
\end{equation*}
It is clear that $\lambda_{\flat} = \lambda_{C}$ if $\beta_{\flat} \geq (\alpha + \beta)/2$, and $\lambda_{\flat} < \lambda_{C}$ if $\beta_{\flat} < (\alpha + \beta)/2$. 

\begin{figure}[htp] \centering
\begin{tikzpicture}[scale = 2] 
\pgfmathsetmacro\BETA{53.27*2}; 
\pgfmathsetmacro\THETAa{(180 + \BETA)/2}; \pgfmathsetmacro\ya{sin(\THETAa - \BETA)/sin(\THETAa)}); 
\pgfmathsetmacro\THETAb{(180 + 3*\BETA)/4}; \pgfmathsetmacro\yb{sin(\THETAb - \BETA)/sin(\THETAb)}); 
\pgfmathsetmacro\THETAc{(180 + 2*\BETA)/4}; \pgfmathsetmacro\yc{sin(2*\THETAc - \BETA)/(sin(2*\THETAc - \BETA) + sin(\BETA))}; 
\pgfmathsetmacro\THETAd{\BETA/2}; \pgfmathsetmacro\yd{sin(2*\THETAd - \BETA)/(sin(2*\THETAd - \BETA) + sin(\BETA))}; 
\draw[very thin, red] plot [domain = \THETAb : \THETAa] ({\x/50}, {sin(\x - \BETA)/sin(\x)}); 
\draw[very thin, red] plot [domain = \THETAd : \THETAb] ({\x/50}, {sin(2*\x - \BETA)/(sin(2*\x - \BETA) + sin(\BETA))}); 
\fill(\THETAa/50, \ya) circle (0.5pt); \fill(\THETAb/50, \yb) circle (0.5pt); \fill(\THETAc/50, \yc) circle (0.5pt); 
\fill(\THETAd/50, \yd) circle (0.5pt) node [below] {\scriptsize $\frac{\beta}{2}$}; 
\draw[dashed] (\THETAa/50, \ya) -- (\THETAa/50, {0}) node [below] {\scriptsize $\frac{\pi + \beta}{2}$}; 
\draw[dashed] (\THETAb/50, \yb) -- (\THETAb/50, {0}) node [below] {\scriptsize $\frac{\pi + 3\beta}{4}$}; 
\draw[dashed] (\THETAc/50, \yc) -- (\THETAc/50, {0}) node [below] {\scriptsize $\frac{\pi + 2\beta}{4}$}; 
\draw[ -> ] (\THETAd/2/50, 0) -- (180/50, 0) node [right] {\scriptsize $\vartheta$}; 
\draw[ -> ] (\THETAd/2/50, 0) -- (\THETAd/2/50, 1.2) node [right] {\scriptsize $\zeta(\vartheta)$}; 
\end{tikzpicture} 
\caption{The function of $\zeta(\vartheta)$. }
\label{Fig04Zeta}
\end{figure}
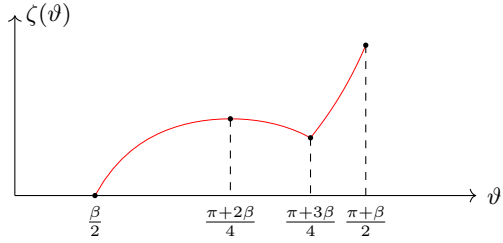

\begin{lma} \label{lma302}
Let $0 < \beta < \alpha \leq \pi$. Suppose there exists $\Lambda > 0$ such that \eqref{Yao0315} holds for all $\vartheta \in (0, (\pi + \beta)/2]$ and $\lambda \geq \Lambda$. Then \eqref{Yao0315} also holds for $\vartheta = \pi/2$ and
\begin{equation} \label{Yao0334}
\lambda \geq \tilde{\Lambda} = \max\left\{ \zeta_{\flat} \Lambda, \, \lambda_{\flat} \right\}. 
\end{equation}
\end{lma}

\begin{proof}
Fix $\vartheta = \pi/2$, $\vartheta_{1} = 0$, and $\vartheta_{3} = \pi$. Observe that both $\Gamma_{\lambda, \vartheta, \vartheta_{1}}^{2A}$ and its reflection with respect to ${T}_{\lambda, \vartheta}$ are contained within the lower Neumann boundary $\Gamma_{N}^{ - }$; see \autoref{Fig03MMP90}. Owing to the Neumann boundary condition satisfied by ${u}$ on $\Gamma_{N}^{ - }$, we see that ${w}^{\lambda, \vartheta}$ satisfies the homogeneous Neumann boundary condition on $\Gamma_{\lambda, \vartheta, \vartheta_{1}}^{2A}$. 

To verify the Neumann-type boundary condition on $\Gamma_{\lambda, \vartheta, \vartheta_{1}}^{2B}$, note that
\begin{equation*}
\Gamma_{\lambda, \vartheta, \vartheta_{1}}^{2B} \subset {T}_{\hat{\lambda}, \, \hat{\vartheta}} = {T}_{\check{\lambda}, \, \check{\vartheta}}, 
\end{equation*}
where
\begin{equation*}
\text{either } \check{\vartheta} = \pi - \beta \in (0, \, (\pi + \beta)/2] \quad \text{or} \quad \hat{\vartheta} = 2\beta \in (0, \, (\pi + \beta)/2]. 
\end{equation*}
Under \eqref{Yao0334}, we have
\begin{equation*}
\lambda_{*} = \lambda/\zeta(\pi/2) \geq \Lambda \quad \text{and} \quad \pi/2 \in {J}_{\lambda}. 
\end{equation*}
Therefore, ${w}^{\lambda, \vartheta, \vartheta_{1}}$ fulfills the strict boundary condition \eqref{Yao0322d} on $\Gamma_{\lambda, \vartheta, \vartheta_{1}}^{2B}$, and thus satisfies \eqref{Yao0322} for all $\lambda \geq \tilde{\Lambda}$. Applying \autoref{lma301} to ${w}^{\lambda, \vartheta, \vartheta_{1}}$ completes the proof. 
\end{proof}

\begin{lma} \label{lma303}
Let $0 < \beta < \alpha \leq \pi$. Suppose there exists $\Lambda > 0$ such that \eqref{Yao0315} holds for all $\vartheta \in (0, (\pi + \beta)/2]$ and $\lambda \geq \Lambda$. 
Then \eqref{Yao0315} also holds for all $\vartheta \in (0, \beta_{\flat}]$ and $\lambda \geq \tilde{\Lambda} = \max\{\zeta_{\flat} \Lambda, \lambda_{\flat}\}$. 
\end{lma}

\begin{proof}
For $\vartheta \in (0, \beta_{\flat}]$ and $\lambda \geq \tilde{\Lambda}$, we have: 
\begin{enumerate}
\item[(1)]
If $\vartheta \in (0, \beta/2]$, then $\vartheta_{3} \leq \beta$ and $\Gamma_{\lambda, \vartheta, \vartheta_{1}}^{2B} = \emptyset$. 
\item[(2)]
If $\vartheta \in (\beta/2, \beta_{\flat}]$, then $\Gamma_{\lambda, \vartheta, \vartheta_{1}}^{2B} \subset {T}_{\hat{\lambda}, \hat{\vartheta}} = {T}_{\check{\lambda}, \check{\vartheta}}$, where either $\check{\vartheta} \in (0, (\pi + \beta)/2]$ or $\hat{\vartheta} \in (0, (\pi + \beta)/2]$, and $\lambda_{*} = \lambda/\zeta(\vartheta) \geq \lambda/\zeta_{\flat}$. 
\end{enumerate}
By assumption, ${w}^{\lambda, \vartheta}$ satisfies the strict boundary condition on $\Gamma_{\lambda, \vartheta}^{2B}$. 

\textbf{Step 1}. 
We claim that \eqref{Yao0315a} holds for every $\lambda \geq \tilde{\Lambda}$ and $\vartheta \in {J}_{{k}}$, where ${k} \in \mathbb{N}^{ + }$ and $\mathbb{N}^{ + }$ denotes the set of all positive integers, 
\begin{equation*}
{J}_{{k}} = \Big\{ \frac{{j}\pi}{2^{{k}}} \in (0, \beta_{\flat}]: \; {j} \in \mathbb{N}^{ + } \Big\}, \quad {k} \in \mathbb{N}^{ + }. 
\end{equation*}
We proceed by induction on ${k}$. By \autoref{lma302}, the assertion holds for $\vartheta \in {J}_{{1}} = \{\pi/2\}$. Assume the claim holds for ${k} = {k}_{0}$ for some ${k}_{0} \in \mathbb{N}^{ + }$. Now let $\vartheta = (2{j} + 1)\pi/2^{{k}_{0} + 1} \in {J}_{{k}_{0} + 1} \setminus {J}_{{k}_{0}}$. Define
\begin{equation*}
\begin{aligned}
& \vartheta_{1} = \frac{{j}\pi}{2^{{k}_{0}}}, \quad \vartheta_{3} = \frac{({j} + 1)\pi}{2^{{k}_{0}}} && \text{if } \vartheta < \pi/2, \\
& \vartheta_{1} = 2\vartheta - \pi, \quad \vartheta_{3} = \pi && \text{if } \vartheta > \pi/2. 
\end{aligned}
\end{equation*}
Then $\vartheta_{1}, \vartheta_{3} \in {J}_{{k}_{0}} \cup \{0\}$, and ${w}^{\lambda, \vartheta}$ satisfies the strict boundary condition on $\Gamma_{\lambda, \vartheta, \vartheta_{1}}^{2A}$. Hence, ${w}^{\lambda, \vartheta}$ satisfies \eqref{Yao0322} for $\lambda \geq \tilde{\Lambda}$. 
By the argument in \autoref{lma301}, \eqref{Yao0305c} and \eqref{Yao0315a} hold for $\lambda \geq \tilde{\Lambda}$; thus, the claim holds for ${k} = {k}_{0} + 1$. 
By induction, the statement holds for all ${k} \in \mathbb{N}^{ + }$. 

\textbf{Step 2}. 
Since ${J}_{\infty}: = \bigcup_{{m} = 1}^{\infty} {J}_{{m}}$ is dense in $(0, \beta_{\flat}]$, we deduce from Step 1 that
\begin{equation} \label{Yao0336}
{u}_{{x}_{1}} \sin(\vartheta - \tfrac{\beta}{2}) - {u}_{{x}_{2}} \cos(\vartheta - \tfrac{\beta}{2}) \leq 0
\quad \text{on } {T}_{\lambda, \vartheta} \cap \Sigma
\end{equation}
holds for all $\vartheta \in (0, \beta_{\flat}]$ and $\lambda \geq \tilde{\Lambda}$. For any $\vartheta \in (0, \beta_{\flat}]$, one can choose $\vartheta_{1}, \vartheta_{3} \in [0, \beta_{\flat}] \cup \{\pi\}$ such that $\vartheta - \vartheta_{1} = \vartheta_{3} - \vartheta > 0$, for example, 
\begin{equation*}
\vartheta_{1} = 
\begin{cases}
2\vartheta - \pi & \text{if } \vartheta \geq \pi/2, \\
\max\{2\vartheta - \beta_{\flat}, 0\} & \text{if } \vartheta < \pi/2. 
\end{cases}
\end{equation*}
Then ${w}^{\lambda, \vartheta}$ satisfies \eqref{Yao0322}. Applying \autoref{lma301}, we obtain
\begin{equation} \label{Yao0337}
{w}^{\lambda, \vartheta} < 0 \quad \text{in } {D}_{\lambda, \vartheta} \quad \text{for all } \vartheta \in [\pi/2, \beta_{\flat}], \; \lambda \geq \tilde{\Lambda}, 
\end{equation}
and
\begin{equation*}
{w}^{\lambda, \vartheta} < 0 \quad \text{in } {D}_{\lambda, \vartheta, \vartheta_{1}} \quad \text{with } \vartheta_{1} = \max\{2\vartheta - \beta_{\flat}, 0\}
\end{equation*}
for $\vartheta \in (0, \pi/2)$ and $\lambda \geq \tilde{\Lambda}$. 
Hence, \eqref{Yao0315a} holds for all $\vartheta \in (0, \beta_{\flat}]$ and $\lambda \geq \tilde{\Lambda}$. 
The proof for \eqref{Yao0315b} is similar. This completes the proof. 
\end{proof}

\begin{lma} \label{lma304}
Let $0 < \beta < \alpha \leq \pi$. Suppose that there exists some $\Lambda > 0$ such that \eqref{Yao0315} holds for $\vartheta \in (0, (\pi + \beta)/2]$ and $\lambda \geq \Lambda$. 
Then \eqref{Yao0315} holds for $\vartheta \in {J}_{\lambda}$ and $\lambda \geq \tilde{\Lambda} = \max\{\zeta_{\flat} \Lambda, \lambda_{\flat}\}$. 
\end{lma}

\begin{proof}
Let $\vartheta = (\pi + \beta)/2$. Then $\hat{\vartheta} = \pi + 2\beta - 2\vartheta = \beta$ and $2\vartheta - \pi = \beta$. 
Thanks to \autoref{lma303}, we know ${w}^{\lambda, (\pi + \beta)/2, \beta}$ satisfies \eqref{Yao0322} for $\lambda \geq \tilde{\Lambda}$. Thus, 
\begin{equation} \label{Yao0338}
{w}^{\lambda, (\pi + \beta)/2} < 0 \text{ in } {D}_{\lambda, (\pi + \beta)/2} \text{ for } \lambda \geq \tilde{\Lambda}. 
\end{equation}
Let us fix $\lambda \geq \tilde{\Lambda}$. For $\vartheta \in (\beta_{\flat}, (\pi + \beta)/2]$ and $\vartheta_{3} = \pi$, we have
\begin{equation*}
\vartheta_{1} = 2\vartheta - \vartheta_{3} \in (0, \beta] \subset (0, \beta_{\flat}]
\end{equation*}
and
\begin{equation*}
\hat{\vartheta} = \pi + 2\beta - 2\vartheta \in [\beta, \beta_{\flat}]
\end{equation*}
where the assumption $\beta_{\flat} \geq (\pi + 2\beta)/3$ in \eqref{Yao0332zeta} is used. 
Then ${w}^{\lambda, \vartheta, 2\vartheta - \pi}$ satisfies \eqref{Yao0322} for $\vartheta \in [\beta_{\flat}, (\pi + \beta)/2] \cap {J}_{\lambda}$. By the definition of ${J}_{\lambda}$ in \eqref{Yao0312}, we know that $[\beta_{\flat}, (\pi + \beta)/2] \cap {J}_{\lambda}$ is a union of (at most) two closed intervals, and $\beta_{\flat}$ is in one interval while $(\pi + \beta)/2$ is in the other interval. 
Combining this with \eqref{Yao0337} and \eqref{Yao0338}, and applying \autoref{lma301}, we deduce the negativity of ${w}^{\lambda, \vartheta, 2\vartheta - \pi}$ and thus \eqref{Yao0315a} holds for $\vartheta \in [\beta_{\flat}, (\pi + \beta)/2] \cap {J}_{\lambda}$. 
\end{proof}

\begin{prop} \label{prop305}
Let $0 < \beta < \alpha \leq \pi$. Then \eqref{Yao0315} holds for every $\vartheta \in {J}_{\lambda}$ and $\lambda \geq \max\{\zeta_{\flat}\lambda_{C}, \lambda_{\flat}\}$. 
\end{prop}

\begin{proof}
By \autoref{lma302}, \autoref{lma303}, \autoref{lma304} and mathematical induction, one can show that \eqref{Yao0315} holds for $\vartheta \in (0, (\pi + \beta)/2]$ and $\lambda \geq \max\{\zeta_{\flat}^{{k}} \lambda_{\mathrm{max}}, \lambda_{C}\}$ for every ${k} \in \mathbb{N}^{ + }$. Hence, \eqref{Yao0315} holds for $\vartheta \in (0, (\pi + \beta)/2]$ and $\lambda \geq \lambda_{C}$. 
Again by \autoref{lma302}, \autoref{lma303}, \autoref{lma304}, we conclude that \eqref{Yao0315} holds for $\vartheta \in {J}_{\lambda}$ and $\lambda \geq \max\{\zeta_{\flat} \lambda_{C}, \lambda_{\flat}\}$. 
\end{proof}

\begin{rmk}
Let $0 < \beta < \alpha \leq \pi$. Then $\max\{\zeta_{\flat} \lambda_{C}, \lambda_{\flat}\} < {l}_{N}$. 
\end{rmk}

\begin{proof}
In the case $\alpha < \pi$, we have $\lambda_{\flat} \leq \lambda_{C} < {l}_{N}$, while in the case $\alpha = \pi$, we have $\beta_{\flat} < (\alpha + \beta)/2$ and $\lambda_{\flat} < \lambda_{C} = {l}_{N}$. Hence $\lambda_{\flat} < {l}_{N}$ and $\max\{\zeta_{\flat} \lambda_{C}, \lambda_{\flat}\} < {l}_{N}$ always hold.
\end{proof}


\subsection{The case $\vartheta_{A}(\lambda) \geq \beta$}

We begin with the strict monotonicity along the Neumann boundary, which plays a key role in the moving plane method. In \cite[Theorem~2.4]{BP89} (see also \cite[Lemma~4]{YCG21} and~\cite{CLY21}), the non-vanishing of the tangential derivative is obtained by first proving ${w}^{\lambda,\pi/2}<0$ and then applying Serrin's boundary point lemma.  However, this approach is not applicable in general, and we will instead use the local analysis developed by Hartman and Wintner~\cite{HW53}.

\begin{lma} \label{lma307}
Let $0 < \beta < \alpha \leq \pi$ and let $\Lambda \in (0, {l}_{N})$. Suppose there exists an angle $\vartheta_{C} \in (0, \pi)$ with $\vartheta_{C} \geq \vartheta_{B}$ and
\begin{equation} \label{Yao0344}
\vartheta_{C} > \pi/2
\end{equation}
such that
\eqref{Yao0315} holds for every $\vartheta \in (0, \vartheta^{A}(\lambda)] \cup [\vartheta_{B}(\lambda), \vartheta_{C}]$ and $\lambda = \Lambda$, where $\vartheta^{A}(\lambda)$ is given in \eqref{Yao0314}. Then we have the strict monotonicity properties along the Neumann boundary, namely, 
\begin{subequations} \label{Yao0345}
\begin{align}
\label{Yao0345a}
&{u}_{{x}_{1}}\cos\frac{\beta}{2} - {u}_{{x}_{2}}\sin\frac{\beta}{2} < 0 \text{ on } {T}_{\lambda, \pi/2} \cap \Gamma_{N}^{ - }, 
\\ \label{Yao0345b}
&{u}_{{x}_{1}}\cos\frac{\beta}{2} + {u}_{{x}_{2}}\sin\frac{\beta}{2} < 0 \text{ on } \hat{T}_{\lambda, \pi/2} \cap \Gamma_{N}^{ + }
\end{align}
\end{subequations}
hold for $\lambda = \Lambda$. 
\end{lma}

\begin{proof}
Set $\lambda = \Lambda$ and define
\begin{equation} \label{Yao0347aRotation}
\begin{gathered}
{x}_{1\lambda} = \lambda\cos\frac{\beta}{2}, \quad {x}_{2\lambda} = - \lambda\sin\frac{\beta}{2}, \\
{v}_{\lambda}({x}) = ({x}_{1} - {x}_{1\lambda}){u}_{{x}_{2}}({x}) - ({x}_{2} - {x}_{2\lambda}){u}_{{x}_{1}}({x}). 
\end{gathered}
\end{equation}
The function ${v}_{\lambda}$ can be interpreted as the rotation function of ${u}$ about the point ${x}_{\lambda} = ({x}_{1\lambda}, {x}_{2\lambda})$. Observe that ${v}_{\lambda}$ vanishes on $\Gamma_{N}^{ - }$ and satisfies the linearized equation
\begin{equation} \label{Yao0347bRotation}
\mathcal{L}[{v}_{\lambda}] = \Delta {v}_{\lambda} + {f}'({u}) {v}_{\lambda} = 0 \quad \text{in } \Sigma. 
\end{equation}
By the assumptions, we have
\begin{equation} \label{Yao0348}
{v}_{\lambda}({r}, \vartheta) > 0 \quad \text{for } \vartheta \in (0, \vartheta^{A}(\lambda)] \cup [\vartheta_{B}(\lambda), \vartheta_{C}]
\end{equation}
where $({r}, \vartheta)$ are the polar coordinates centered at ${x}_{\lambda}$ with the polar axis aligned along $\Gamma_{N}^{ - }$, denoted as ${x}_{\lambda}{P}_{ + }$. According to \autoref{prop207}, there exists a positive integer ${l}$ such that
\begin{equation*}
{v}_{\lambda}({r}, \vartheta) = {C}_{0} {r}^{{l}} \sin({l}\vartheta) + O({r}^{{l} + 1}), 
\end{equation*}
for some ${C}_{0} \in \mathbb{R}^{ + }$, where $O({r}^{{l} + 1})/{r}^{{l} + 1}$ remains bounded as ${r} \to 0$. Using this and \eqref{Yao0348}, it follows that
\begin{equation*}
\sin({l}\vartheta) > 0 \quad \text{for } \vartheta \in (0, \vartheta^{A}(\lambda)), \quad
\sin({l}\vartheta) \geq 0 \quad \text{for } \vartheta \in [\vartheta_{B}(\lambda), \vartheta_{C}]. 
\end{equation*}
Consequently, ${l}\vartheta^{A}(\lambda) \leq \pi$. By \autoref{lma208}\ref{Yaoit23a}, we also have $\vartheta_{B}(\lambda) < 2\vartheta^{A}(\lambda)$, which implies ${l}\vartheta_{B}(\lambda) < 2{l}\vartheta^{A}(\lambda) \leq 2\pi$. Since $\sin({l}\vartheta_{B}(\lambda)) \geq 0$, we deduce ${l}\vartheta_{B}(\lambda) \leq \pi$. Therefore, $\sin({l}\vartheta) \geq 0$ for all $\vartheta \in (0, \vartheta_{C}]$, and in particular ${l}\vartheta_{C} \leq \pi$. Thanks to \eqref{Yao0344}, 
\begin{equation*}
{l} \leq \frac{\pi}{\vartheta_{C}} < 2. 
\end{equation*}
Since ${l}$ is a positive integer, it must be that ${l} = 1$. Accordingly, the outward normal derivative of ${v}_{\lambda}$ at $\bar{x}$ does not vanish and is negative, that is, 
\begin{equation*}
\partial_{\nu} {v}_{\lambda} ({x}_{1\Lambda}, {x}_{2\Lambda}) < 0. 
\end{equation*}
Therefore, \eqref{Yao0345a} (and similarly, \eqref{Yao0345b}) holds for $\lambda = \Lambda$. 
\end{proof}

Before proving \eqref{Yao0315a} for $\lambda < \lambda_{C}$, we illustrate that if \eqref{Yao0315a} holds for $\vartheta = \vartheta_{B}(\lambda)$, then it holds for all $\vartheta \leq \vartheta^{A}(\lambda)$, where $\vartheta^{A}$ and $\vartheta_{B}$ are given in \eqref{Yao0314}.

\begin{lma} \label{lma308}
Let $0 < \beta < \alpha \leq \pi$ and let $\Lambda > 0$. Suppose that \eqref{Yao0315a} holds for every $\vartheta = \vartheta_{B}(\lambda)$ and $\lambda \geq \Lambda$. Then \eqref{Yao0315a} is valid for $\vartheta \in (0, \vartheta^{A}(\lambda)]$ and $\lambda \geq \Lambda$, and
\begin{equation} \label{Yao0351}
{w}^{\lambda, \vartheta_{A}(\lambda)} < 0 \text{ in } {D}_{\lambda, \vartheta_{A}(\lambda), 2\vartheta_{A}(\lambda) - \vartheta_{B}(\lambda)}
\end{equation}
is valid for $\lambda \in [\Lambda, \lambda_{C})$ provided $\Lambda < \lambda_{C}$. 
\end{lma}

\begin{proof}
By \autoref{lma208}\ref{Yaoit23a}, we note that $\frac{1}{2}\vartheta_{B}(\lambda) < \vartheta^{A}(\lambda)$ always holds. 

\textbf{Step 1}. 
\eqref{Yao0315a} holds for $\vartheta \in (0, \vartheta_{B}(\lambda)/2]$ and $\lambda \geq \Lambda$. Indeed, we first claim that for every ${m} \in \mathbb{N}^{ + }$ and $1 \leq {j} \leq 2^{{m} - 1}$, \eqref{Yao0315a} holds for $\vartheta = {j}2^{ - {m}}\vartheta_{B}(\lambda)$ and $\lambda \geq \Lambda$. To show the claim, we take
\begin{equation*}
\vartheta_{1}(\lambda) = 0, \quad
\vartheta_{2}(\lambda) = \tfrac{1}{2}\vartheta_{B}(\lambda), \quad
\vartheta_{3}(\lambda) = \vartheta_{B}(\lambda). 
\end{equation*}
For $\lambda \geq \Lambda$, we see that
$\Gamma_{\lambda, \vartheta_{2}(\lambda)}^{2B}$ is empty, \eqref{Yao0315a} holds for $\vartheta = \vartheta_{3}(\lambda)$, and then ${w}^{\lambda, \vartheta_{2}(\lambda)}$ satisfies \eqref{Yao0322}. From \autoref{lma301}, we can prove
\begin{equation} \label{Yao0352}
{w}^{\lambda, \frac{1}{2}\vartheta_{B}(\lambda)} < 0 \text{ in } {D}_{\lambda, \frac{1}{2}\vartheta_{B}(\lambda)}
\end{equation}
and thus \eqref{Yao0315a} holds for $\vartheta = \vartheta_{B}(\lambda)/2$ and $\lambda \geq \Lambda$. This proves the claim for ${m} = 1$. 

Suppose the claim holds for some ${m}_{0} \geq 1$. Now take $1 \leq {j} \leq 2^{{m}_{0} - 1}$ (then $2{j} - 1 \leq 2^{{m}_{0}}$), and
\begin{equation*}
\vartheta_{2}(\lambda) = \frac{2{j} - 1}{2^{{m}_{0} + 1}}\vartheta_{B}(\lambda), \quad
\vartheta_{1}(\lambda) = \frac{{j} - 1}{2^{{m}_{0}}}\vartheta_{B}(\lambda), \quad
\vartheta_{3}(\lambda) = \frac{{j}}{2^{{m}_{0}}}\vartheta_{B}(\lambda). 
\end{equation*}
Thus, ${w}^{\lambda, \vartheta_{2}(\lambda), \vartheta_{1}(\lambda)}$ satisfies \eqref{Yao0322}. From \autoref{lma301}, \eqref{Yao0315a} holds for $\vartheta = \frac{2{j} - 1}{2^{{m}_{0} + 1}}\vartheta_{B}(\lambda)$ and $\lambda \geq \Lambda$. Therefore, the claim holds for ${m} = {m}_{0} + 1$. By induction, the claim holds for all ${m} \geq 1$. 

Since the set $\{{j}2^{ - {m}}: 1 \leq {j} \leq 2^{{m} - 1}, \, {m} \in \mathbb{N}^{ + }\}$ is dense in $[0, 1/2]$, we conclude that \eqref{Yao0336} holds for $\vartheta \in (0, \vartheta_{B}(\lambda)/2]$. For any ${q}_{2} \in (0, 1/2)$, set ${q}_{1} = \max\{0, 2{q}_{2} - 1/2\}$, ${q}_{3} = \min\{1/2, 2{q}_{2}\}$. Then we can prove the negativity of ${w}^{\lambda, {q}_{2}\vartheta_{B}(\lambda)}$ in ${D}_{\lambda, {q}_{2}\vartheta_{B}(\lambda), {q}_{1}\vartheta_{B}(\lambda)}$, and thus \eqref{Yao0315a} holds for $\vartheta = {q}_{2}\vartheta_{B}(\lambda)$ and $\lambda \geq \Lambda$. Therefore, step 1 is proved. 

\textbf{Step 2}. 
\eqref{Yao0315a} holds for $\lambda \geq \Lambda$ and $\vartheta \in (0, \vartheta^{A}(\lambda)]$. Indeed, fix $\lambda$ and define ${I}_{\lambda, {m}} = (0, (1 - 2^{ - {m}})\vartheta_{B}(\lambda)] \cap (0, \vartheta^{A}(\lambda)]$. Then $(0, \vartheta^{A}(\lambda)] = \cup_{{m} = 1}^{\infty}{I}_{\lambda, {m}}$. By step 1, \eqref{Yao0315a} holds for $\vartheta \in {I}_{\lambda, 1}$. Suppose that \eqref{Yao0315a} holds for every $\vartheta \in {I}_{\lambda, \bar{m}}$ for some $\bar{m} \geq 1$. Let
\begin{equation} \label{Yao0353}
\vartheta \in {I}_{\lambda, \bar{m} + 1} \quad \text{and }\quad \vartheta \geq \vartheta_{B}(\lambda)/2. 
\end{equation}
Then $2\vartheta - \vartheta_{B}(\lambda) \in {I}_{\lambda, \bar{m}} \cup \{0\}$. Thus, ${w}^{\lambda, \vartheta, 2\vartheta - \vartheta_{B}(\lambda)}$ satisfies \eqref{Yao0322}. Combining this with \eqref{Yao0352}, and applying \autoref{lma301}, we conclude that ${w}^{\lambda, \vartheta} < 0$ in ${D}_{\lambda, \vartheta, 2\vartheta - \vartheta_{B}(\lambda)}$ and \eqref{Yao0315a} hold. Therefore, \eqref{Yao0315a} holds for every $\vartheta \in {I}_{\lambda, \bar{m} + 1}$. By mathematical induction, \eqref{Yao0315a} holds for every $\vartheta \in \cup_{{m} = 1}^{\infty}{I}_{\lambda, {m}}$. Moreover, for $\lambda \geq \Lambda$, 
\begin{equation*}
{w}^{\lambda, \vartheta} < 0 \text{ in } {D}_{\lambda, \vartheta, 2\vartheta - \vartheta_{B}(\lambda)} \text{ for } \vartheta \in [\vartheta_{B}(\lambda)/2, \vartheta^{A}(\lambda)].
\end{equation*} 
Furthermore, for $\vartheta \in (0, \vartheta^{A}(\lambda)/2]$ and $\lambda \geq \Lambda$, ${w}^{\lambda, \vartheta}$ satisfies \eqref{Yao0322} in ${D}_{\lambda, \vartheta}$. Hence, ${w}^{\lambda, \vartheta} < 0$ in ${D}_{\lambda, \vartheta}$. 
\end{proof}

Let us define two important constants $\lambda_{\sharp}$ and ${l}_{\perp}$ as follows: 
\begin{equation} \label{Yao0341}
\begin{aligned}
\lambda_{\sharp} & = \{\lambda > 0: \; \beta \in {J}_{\lambda}\} = \sin\tfrac{\alpha - \beta}{2}\csc\beta, \\
{l}_{\perp} & = \inf\{\lambda \in \mathbb{R}^{ + }: \; \vartheta_{B}(\lambda) > \pi/2\} = \max\{{l}_{N}\cos\beta, 0\}. 
\end{aligned}
\end{equation}

\begin{thm} \label{thm309}
Let $0 < \beta < \alpha \leq \pi$. Then the inequality \eqref{Yao0316}: 
\begin{align*}
{u}_{{x}_{1}}\sin(\vartheta - \tfrac{\beta}{2}) - {u}_{{x}_{2}}\cos(\vartheta - \tfrac{\beta}{2}) & < 0 \quad \text{on } {T}_{\lambda, \vartheta} \cap (\Sigma \cup \Gamma_{N}^{ - }), \\
{u}_{{x}_{1}}\sin(\vartheta - \tfrac{\beta}{2}) + {u}_{{x}_{2}}\cos(\vartheta - \tfrac{\beta}{2}) & < 0 \quad \text{on } \hat{T}_{\lambda, \vartheta} \cap (\Sigma \cup \Gamma_{N}^{ + })
\end{align*}
holds for every $\vartheta \in {J}_{\lambda}$ and $\lambda \geq \lambda_{\sharp}$, where $\lambda_{\sharp}$ is given by \eqref{Yao0341}. 
\end{thm} 

\begin{proof}
Let
\begin{equation} \label{Yao0355}
\bar{\Lambda} = \inf\big\{\lambda' > 0: \; \eqref{Yao0315} \text{ holds for all } \vartheta \in {J}_{\lambda} \text{ and } \lambda \geq \lambda' \big\}. 
\end{equation}
From \autoref{prop305}, we see that $\bar{\Lambda}$ is well-defined and $\bar{\Lambda} \leq \lambda_{C}$, $\bar{\Lambda} < {l}_{N}$. We argue by contradiction, and suppose that $\bar{\Lambda} > \lambda_{\sharp}$. The proof will be divided into several steps. 

\textbf{Step 1}. 
\eqref{Yao0315} holds for $\vartheta \in \mathrm{Int}({J}_{\lambda}) \cup \{(\pi + \beta)/2\}$ and $\lambda = \bar{\Lambda}$. 

By continuity, we see that
\begin{align*}
{u}_{{x}_{1}}\sin(\vartheta - \tfrac{\beta}{2}) - {u}_{{x}_{2}}\cos(\vartheta - \tfrac{\beta}{2}) &\leq 0 \quad \text{on } {T}_{\lambda, \vartheta} \cap \Sigma, \\
{u}_{{x}_{1}}\sin(\vartheta - \tfrac{\beta}{2}) + {u}_{{x}_{2}}\cos(\vartheta - \tfrac{\beta}{2}) &\leq 0 \quad \text{on } \hat{T}_{\lambda, \vartheta} \cap \Sigma
\end{align*}
holds for $\vartheta \in {J}_{\lambda}$ and $\lambda = \bar{\Lambda}$. Then the rotation function ${v}_{\lambda}$, given in \eqref{Yao0347aRotation}, is nonnegative and satisfies the linear equation \eqref{Yao0347bRotation}. 
Using the strong maximum principle, along with the strict monotonicity of ${u}$ near the Dirichlet boundary (\autoref{lma206GNN}) and on the Neumann boundary (\autoref{lma307}), we deduce the positivity of ${v}_{\lambda}$. In particular, \eqref{Yao0315a} holds for $\lambda = \bar{\Lambda}$ and $\vartheta \in \mathrm{Int}({J}_{\lambda})$. 

Note that the assumption $\lambda \geq \bar{\Lambda} > \lambda_{\sharp}$ guarantees that $\beta \in \mathrm{Int}({J}_{\lambda})$. Hence \eqref{Yao0315} holds for $\vartheta = \beta$, and ${w}^{\lambda, (\pi + \beta)/2}$ satisfies \eqref{Yao0322} in ${D}_{\lambda, (\pi + \beta)/2, \beta}$. By \autoref{lma301}, 
\begin{equation*}
{w}^{\lambda, (\pi + \beta)/2} < 0 \quad \text{in } {D}_{\lambda, (\pi + \beta)/2}. 
\end{equation*}
Therefore, \eqref{Yao0315a} holds for $\vartheta = (\pi + \beta)/2$ and $\lambda = \bar{\Lambda}$. 

We remark that \eqref{Yao0315} holds for $\vartheta = \vartheta_{A}(\lambda), \vartheta_{B}(\lambda)$ and $\lambda = \bar{\Lambda}$, but the proof is contained in the following steps. 

\textbf{Step 2}. 
The tangential derivatives of ${u}$ along the Neumann boundary $\Gamma_{N}$ do not vanish, i.e., \eqref{Yao0345} holds for $\lambda \geq \bar{\Lambda}$. This follows from \autoref{lma307}. 

We are now ready to prove \eqref{Yao0315} provided that $\bar{\Lambda} - \lambda > 0$ is sufficiently small. 

\textbf{Step 3}. 
\eqref{Yao0315} holds for $\vartheta = \vartheta_{B}(\lambda)$ and $0 \leq \bar{\Lambda} - \lambda \ll 1$. 

The key is to check the boundary condition on $\Gamma^{2A}$ when $\vartheta$ is close to $\vartheta_{B}$. Note that ${u}$ has strict monotonicity in the interior (see step 1) and on the Neumann boundary $\Gamma_{N}$ (see step 2). Together with the fact that $\bar{\Lambda} < {l}_{N}$, we see ${u}_{{x}_{1}} < 0$ on ${T}_{\bar{\Lambda}, (\pi + \beta)/2} \cap \overline{\Sigma}$. By continuity, there exist small constants $\delta_{1} \in (0, (\pi - \beta)/2)$ and $\varepsilon_{1} \in (0, \bar{\Lambda} - \lambda_{\sharp}]$ such that \eqref{Yao0315} holds for $(\lambda, \vartheta) \in \mathcal{S}_{1}$, where
\begin{equation*}
\mathcal{S}_{1} = [\bar{\Lambda} - \varepsilon_{1}, \bar{\Lambda}] \times [(\pi + \beta)/2 - \delta_{1}, (\pi + \beta)/2 + \delta_{1}]. 
\end{equation*}
From \ref{Yaoit23b} and \ref{Yaoit23d} of \autoref{lma208}, we have
\begin{equation*}
2\vartheta_{B}(\lambda) - \tfrac{\pi + \beta}{2} \leq \vartheta_{A}(\lambda) \text{ for } \lambda \in (0, {l}_{\perp}], 
\end{equation*}
and
\begin{equation*}
\delta_{2} = \min \Big\{ \min_{\lambda \in [\lambda_{\sharp}, {l}_{N}]} \frac{\pi + \vartheta_{A}(\lambda) - 2\vartheta_{B}(\lambda)}{3}, \frac{\delta_{1}}{3}, \frac{\beta}{6}, \vartheta_{B}(\lambda_{\sharp}) - \beta \Big\} > 0. 
\end{equation*}
Let $\vartheta^{B}$ be as in \eqref{Yao0314}. Set $\vartheta_{2, \lambda} = \vartheta^{B}(\lambda)$, 
\begin{equation*}
\vartheta_{1, \lambda} = 2\vartheta_{2, \lambda} - \vartheta_{3, \lambda}, \quad
\vartheta_{3, \lambda} = \begin{cases}
\pi & \text{if } \lambda \geq {l}_{\perp}, \\
\min\{2\vartheta_{2, \lambda}, (\pi + \beta)/2 + \delta_{1}\} & \text{if } \lambda < {l}_{\perp}. 
\end{cases}
\end{equation*}
Then for $\lambda \in [\lambda_{\sharp}, {l}_{N}]$, we have
\begin{equation} \label{Yao0359}
\vartheta_{1, \lambda} \in [0, \vartheta_{A}(\lambda) - 3\delta_{2}], \quad
\vartheta_{3, \lambda} \in [\vartheta_{B}(\lambda) + \delta_{2}, (\pi + \beta)/2 + \delta_{1}] \cup \{\pi\}. 
\end{equation}
Since ${u}$ is strictly monotone in the interior (see step 1), on the Neumann boundary (see step 2), and near the Dirichlet boundary (see \autoref{lma206GNN}), by continuity we deduce that
\eqref{Yao0315} holds for $\lambda$ and $\vartheta$ such that
\begin{equation*}
\vartheta \in [\delta_{2}, \vartheta_{A}(\lambda) - \delta_{2}] \cup [\vartheta_{B}(\lambda) + \delta_{2}, (\pi + \beta)/2 + \delta_{1}], \quad \lambda \in [\bar{\Lambda} - \varepsilon_{2}, {l}_{N}], 
\end{equation*}
where $\varepsilon_{2} \in (0, \varepsilon_{1}]$ is a small positive constant such that (for further use)
\begin{equation} \label{Yao0360a}
\check{\lambda}(\lambda, \vartheta_{B}(\lambda)) > \bar{\Lambda} \text{ for } \lambda \geq \bar{\Lambda} - \varepsilon_{2}. 
\end{equation}

As in the proof of \autoref{lma308}, we know \eqref{Yao0315} holds for $(\lambda, \vartheta) \in [\bar{\Lambda} - \varepsilon_{2}, \infty) \times (0, \delta_{2}]$. Thus, \eqref{Yao0315} holds for $(\lambda, \vartheta) \in \mathcal{S}_{2}$, where
\begin{equation*}
\mathcal{S}_{2} = \left\{
\begin{aligned}
(\lambda, \vartheta): \; 
& \vartheta \in (0, \vartheta_{A}(\lambda) - \delta_{2}] \cup [\vartheta_{B}(\lambda) + \delta_{2}, (\pi + \beta)/2 + \delta_{1}] \\
& \text{and } \lambda \in [\bar{\Lambda} - \varepsilon_{2}, {l}_{N}]
\end{aligned}\right\}. 
\end{equation*}
Combining this with \eqref{Yao0359} and \eqref{Yao0360a}, we conclude that for every $\lambda \geq \bar{\Lambda} - \varepsilon_{2}$, the function ${w}^{\lambda, \vartheta_{2, \lambda}}$ satisfies \eqref{Yao0322} on ${D}_{\lambda, \vartheta_{2, \lambda}, \vartheta_{1, \lambda}}$. By \autoref{lma301}, 
\begin{equation} \label{Yao0361}
{w}^{\lambda, \vartheta_{2, \lambda}} < 0 \quad \text{in } {D}_{\lambda, \vartheta_{2, \lambda}, \vartheta_{1, \lambda}} \text{ for } \lambda \geq \bar{\Lambda} - \varepsilon_{2}. 
\end{equation}
Therefore, \eqref{Yao0315a} (and similarly \eqref{Yao0315b}) holds for $\vartheta = \vartheta_{B}(\lambda)$ and $\lambda \in [\bar{\Lambda} - \varepsilon_{2}, {l}_{N}]$. 

\textbf{Step 4}. 
\eqref{Yao0315} is valid for $\vartheta \in (0, \vartheta_{A}(\lambda)]$ and $\lambda \in [\bar{\Lambda} - \varepsilon_{2}, {l}_{N}]$. This is done in \autoref{lma308}. 

\textbf{Step 5}. 
\eqref{Yao0315} holds when $\vartheta \in [\vartheta_{B}(\lambda), (\pi+\beta)/2]$ and $\bar{\Lambda} - \lambda$ is positive and sufficiently small. In fact, we distinguish two cases. 

Step 5.1. The case $\lambda_{C} < {l}_{N}$, i.e., $\alpha < \pi$. 
Since $\hat{\lambda}(\lambda, \vartheta_{B}(\lambda)) = {l}_{N} > \lambda_{C}$ and $\check{\lambda}(\lambda, \vartheta) > \lambda$, there exist $\delta_{3} \in (0, \delta_{2})$ and $\varepsilon_{3} \in (0, \varepsilon_{2})$ such that
\begin{equation} \label{Yao0360b}
\begin{gathered}
\hat{\lambda}(\lambda, \vartheta_{B}(\lambda) + \delta_{3}) \geq \lambda_{C} \quad \text{for all } \lambda \in [\bar{\Lambda} - \varepsilon_{2}, {l}_{N}], 
\\
\check{\lambda}(\lambda, \vartheta_{B}(\lambda) + {t}) > \bar{\Lambda} \quad \text{for all } \lambda \in [\bar{\Lambda} - \varepsilon_{3}, {l}_{N}], \; {t} \in [0, \delta_{3}].
\end{gathered}
\end{equation}
Recall that $2\vartheta_{B}(\lambda) - \beta \geq \vartheta_{B}(\lambda) + \delta_{2}$ and that~\eqref{Yao0315} holds for $(\lambda, \vartheta) \in \mathcal{S}_{2}$. Using~\eqref{Yao0359} and~\eqref{Yao0360b}, we infer that for every ${t} \in [0, \delta_{3}]$ and $\lambda \geq \bar{\Lambda} - \varepsilon_{3}$, the function ${w}^{\lambda, \vartheta_{2, \lambda} + {t}, \vartheta_{1, \lambda} + 2t}$ satisfies the boundary condition in~\eqref{Yao0322} on both $\Gamma_{\lambda, \vartheta_{2, \lambda} + {t}, \vartheta_{1, \lambda} + 2{t}}^{2A}$ and $\Gamma_{\lambda, \vartheta_{2, \lambda} + {t}, \vartheta_{1, \lambda} + 2{t}}^{2B}$. By~\eqref{Yao0361}, we may apply \autoref{lma301} to deduce that for ${t} \in [0, \delta_{3}]$,
\begin{equation*}
{w}^{\lambda, \vartheta_{2, \lambda} + {t}} < 0 \quad \text{in}~{D}_{\lambda, \vartheta_{2, \lambda} + {t}, \vartheta_{1, \lambda} + 2{t}} \quad \text{for}~\lambda \geq \bar{\Lambda} - \varepsilon_{3}.
\end{equation*}
Therefore, \eqref{Yao0315a} holds for $\vartheta \in [\vartheta_{B}(\lambda), \vartheta_{B}(\lambda) + \delta_{3}]$ and $\lambda \in [\bar{\Lambda} - \varepsilon_{3}, {l}_{N}]$. For $\lambda \in [\bar{\Lambda} - \varepsilon_{3}, \bar{\Lambda}]$, set
\[
\mathcal{D} = \left\{ {x} \in \Sigma :\; \vartheta_{B}(\lambda) + \delta_{3} < \sigma_{\lambda}({x}) < (\pi + \beta)/2 \right\}.
\]
Then the function ${v}_{\lambda}$, given in~\eqref{Yao0347aRotation}, belongs to $C^{1}(\overline{\mathcal{D}})$ and satisfies the linear equation \eqref{Yao0347bRotation} and ${v}_{\lambda} > 0$ on $\partial \mathcal{D}$. Recall that
\begin{equation*}
\Delta {u}_{{x}_{1}} + {f}'({u}) {u}_{{x}_{1}} = 0 \quad \text{and} \quad {u}_{{x}_{1}} < 0 \quad \text{in}~\overline{\mathcal{D}}.
\end{equation*}
Applying the maximum principle from~\cite{BNV94} to ${v}_{\lambda}$, we deduce the positivity of ${v}_{\lambda}$ in $\mathcal{D}$. Therefore, \eqref{Yao0315a} holds for $\vartheta \in [\vartheta_{B}(\lambda) + \delta_{3}, (\pi+\beta)/2]$ and $\lambda \in [\bar{\Lambda} - \varepsilon_{3}, \bar{\Lambda}]$. In summary, \eqref{Yao0315a} holds for $\vartheta \in [\vartheta_{B}(\lambda), (\pi+\beta)/2]$ and $\lambda \geq \bar{\Lambda} - \varepsilon_{3}$.

Step 5.2. The case $\lambda_{C} = {l}_{N}$, i.e., $\alpha = \pi$. 
Let ${l}_{*} \in (0, {l}_{N})$ and $\vartheta_{\lambda}$ be as given in \autoref{lma501} below. There are three possibilities:
\begin{equation}
\bar{\Lambda} \in ({l}_{*}, {l}_{N}), \quad \bar{\Lambda} \in (0, {l}_{*}), \quad \bar{\Lambda} = {l}_{*}.
\end{equation}

{Case} $\bar{\Lambda} \in ({l}_{*}, {l}_{N})$. 
In this setting, $\check{\lambda}(\bar{\Lambda}, \vartheta_{B}(\bar{\Lambda})) > {l}_{N}$. Hence, 
\[
\hat{\lambda}(\bar{\Lambda}, \vartheta_{B}(\bar{\Lambda})) = {l}_{N} > \bar{\Lambda}, \quad
0 < \hat{\vartheta}(\vartheta_{B}(\bar{\Lambda})) < (\pi + \beta)/2 = \vartheta_{A}(\hat{\lambda}(\bar{\Lambda}, \vartheta_{B}(\bar{\Lambda}))).
\]
It follows by continuity that for $0 < \vartheta - \vartheta_{B}(\lambda) \ll 1$ and $0 < \bar{\Lambda} - \lambda \ll 1$, 
\begin{equation*}
\bar{\Lambda} < \hat{\lambda}(\lambda, \vartheta) \leq {l}_{N}, \quad
0 < \hat{\vartheta}(\vartheta) < \vartheta_{A}(\hat{\lambda}(\lambda, \vartheta)).
\end{equation*}

{Case} $\bar{\Lambda} \in (0, {l}_{*})$. 
Here, $\check{\lambda}(\bar{\Lambda}, \vartheta_{B}(\bar{\Lambda})) < {l}_{N}$ and $\check{\vartheta}(\vartheta_{B}(\bar{\Lambda})) = \vartheta_{B}(\check{\lambda}(\bar{\Lambda}, \vartheta_{B}(\bar{\Lambda}))) < (\pi+\beta)/2$. It follows by continuity that for $0 < \vartheta - \vartheta_{B}(\lambda) \ll 1$ and $0 < \bar{\Lambda} - \lambda \ll 1$, 
\begin{equation*}
\bar{\Lambda} < \check{\lambda}(\lambda, \vartheta) \leq {l}_{N}, \quad
\vartheta_{B}(\check{\lambda}(\lambda, \vartheta)) \leq \check{\vartheta}(\vartheta) < (\pi+\beta)/2.
\end{equation*}

{Case} $\bar{\Lambda} = {l}_{*}$. 
By~\eqref{Yao0505} and continuity, for $0 < \bar{\Lambda} - \lambda \ll 1$, and $\vartheta \in [\vartheta_{B}(\lambda), \vartheta_{\lambda}]$.
\begin{gather*}
\vartheta_{B}(\lambda) < \vartheta_{\lambda} < \vartheta_{B}(\lambda) + \delta_{2}, 
\\
\bar{\Lambda} < \hat{\lambda}(\lambda, \vartheta) \leq {l}_{N}, \quad
\bar{\Lambda} < \check{\lambda}(\lambda, \vartheta) \leq {l}_{N}. 
\end{gather*}

Thus, there exist a constant $\varepsilon_{4} > 0$ and a function $\psi_{\lambda}$ (depending on $\lambda$) such that 
(i) $\vartheta_{B}(\lambda) < \psi_{\lambda} < \vartheta_{B}(\lambda) + \delta_{2}$ for $\lambda \in [\bar{\Lambda} - \varepsilon_{4}, \bar{\Lambda})$; 
(ii) ${w}^{\lambda, \vartheta}$ satisfies the boundary condition on $\Gamma_{\lambda, \vartheta}^{2B}$ for
\begin{equation} \label{Yao0363}
\vartheta \in [\vartheta_{B}(\lambda), \psi_{\lambda}], \quad \lambda \in [\bar{\Lambda} - \varepsilon_{4}, \bar{\Lambda}].
\end{equation}
From~\eqref{Yao0359}, ${w}^{\lambda, \vartheta}$ also satisfies the boundary condition on $\Gamma_{\lambda, \vartheta, 2\vartheta - \vartheta_{3, \lambda}}^{2A}$. Therefore,
\begin{equation*}
{w}^{\lambda, \vartheta, 2\vartheta - \vartheta_{3, \lambda}} < 0 \quad \text{in}~{D}_{\lambda, \vartheta, 2\vartheta - \vartheta_{3, \lambda}}
\end{equation*}
and \eqref{Yao0315a} holds whenever~\eqref{Yao0363} is satisfied.
Finally, by applying the maximum principle to the angular derivative function ${v}_{\lambda}$, we deduce that \eqref{Yao0315a} holds for $\vartheta \in [\psi_{\lambda}, \pi/2 + \beta/2]$ and $\lambda \in [\bar{\Lambda} - \varepsilon_{4}, \bar{\Lambda}]$.

In both cases $\alpha < \pi$ and $\alpha = \pi$, we have shown that \eqref{Yao0315a} holds for $\vartheta \in [\vartheta_{B}(\lambda), (\pi+\beta)/2]$ and $0 \leq \bar{\Lambda} - \lambda \ll 1$. Similarly, \eqref{Yao0315b} holds for $\vartheta \in [\vartheta_{B}(\lambda), (\pi+\beta)/2]$ and $0 \leq \bar{\Lambda} - \lambda \ll 1$.

Combining all the steps above, we infer that \eqref{Yao0315} holds for $\vartheta \in {J}_{\lambda}$ and $0 \leq \bar{\Lambda} - \lambda \ll 1$. This contradicts the minimality of $\bar{\Lambda}$, and hence $\bar{\Lambda} \leq \lambda_{\sharp}$. Therefore, \eqref{Yao0315} holds for $\vartheta \in {J}_{\lambda}$ and $\lambda \geq \lambda_{\sharp}$. This completes the proof of the theorem.
\end{proof}


\section{The proof for $\pi/2 \leq \beta \leq 2\pi/3$} \label{Sec04Large}

The assumption that $\lambda \geq \lambda_{\sharp}$ in \autoref{thm309} is used in two situations: (1) $2\vartheta_{B} - \pi < \vartheta_{A}$; (2) the proof of \eqref{Yao0315} with $\vartheta = (\pi + \beta)/2$. However, we do not check the negativity of ${u}_{{x}_{1}}$ and only show \eqref{Yao0315} for $\vartheta \in {J}_{\lambda} \cap (0, \vartheta_{B}]$. 
Note that ${J}_{\lambda} \cap (0, \vartheta_{B}] = (0, \vartheta_{A}] \cup \{\vartheta_{B}(\lambda)\}$ for $\lambda < \lambda_{C}$. 

\begin{lma} \label{lma401}
Let $0 < \beta < \alpha \leq \pi$. Then
\begin{equation} \label{Yao0402}
{w}^{\lambda, \vartheta^{B}(\lambda)} < 0 \text{ in } {D}_{\lambda, \vartheta^{B}(\lambda)}
\end{equation}
and \eqref{Yao0316} hold for every $\vartheta \in {J}_{\lambda} \cap (0, \vartheta_{B}(\lambda)]$ and $\lambda \in \mathbb{R}^{ + } \cap [\lambda_{\star\star}, \infty)$, where
\begin{equation} \label{Yao0403}
\lambda_{\star\star} = \inf\{\lambda' > 0: \; \pi/2 \leq \vartheta_{B}(\lambda) < (\pi + \vartheta_{A}(\lambda))/2 \text{ for every } \lambda > \lambda'\}. 
\end{equation}
\end{lma}

\begin{proof}
Set
\begin{equation}
\bar{\Lambda} = \inf\{\lambda' > 0: \; \eqref{Yao0315} \text{ holds for } \vartheta \in {J}_{\lambda} \cap (0, \vartheta_{B}(\lambda)], \lambda \geq \lambda' \}. 
\end{equation}
From \autoref{thm309}, $\bar{\Lambda}$ is well-defined and $0 \leq \bar{\Lambda} \leq \lambda_{\sharp}$. We argue indirectly and suppose that $\bar{\Lambda} > \lambda_{\star\star}$. Note that the definition of $\lambda_{\star\star}$ guarantees that
\begin{equation*}
0 \leq 2\vartheta_{B}(\lambda) - \pi < \vartheta_{A}(\lambda) \text{ for all } \lambda \geq \lambda_{\star\star}. 
\end{equation*}
Based on this and the definition of $\bar{\Lambda}$, we deduce by \autoref{lma301} that
\begin{equation*}
{w}^{\lambda, \vartheta^{B}(\lambda)} < 0 \text{ in } {D}_{\lambda, \vartheta^{B}(\lambda)}
\end{equation*}
and \eqref{Yao0315a} (and similarly \eqref{Yao0315b}) hold for $\vartheta = \vartheta^{B}(\lambda)$, $\lambda \geq \bar{\Lambda}$. Thanks to \autoref{lma308}, we derive that \eqref{Yao0315} holds for $0 < \vartheta \leq \vartheta^{A}(\lambda)$ and $\lambda \geq \bar{\Lambda}$. Applying \autoref{lma307}, we obtain the strict monotonicity along the Neumann boundary, i.e., \eqref{Yao0345} holds for $\lambda \geq \bar{\Lambda}$. By the same arguments as in step 3 of \autoref{thm309}, we can show \eqref{Yao0402} and \eqref{Yao0315} hold for $\vartheta = \vartheta_{B}(\lambda)$ and $\lambda \geq \bar{\Lambda} - \varepsilon$ for some $\varepsilon > 0$. Finally, from \autoref{lma308}, we conclude that \eqref{Yao0315} holds for $0 < \vartheta \leq \vartheta^{A}(\lambda)$ and $\lambda \geq \bar{\Lambda} - \varepsilon$. This contradicts the definition of $\bar{\Lambda}$. Thus, $\bar{\Lambda} = \lambda_{\star\star}$, and hence \eqref{Yao0402} and \eqref{Yao0316} hold for every $\vartheta \in {J}_{\lambda} \cap (0, \vartheta_{B}(\lambda)]$ and $\lambda > \lambda_{\star\star}$. Moreover, \eqref{Yao0402} and \eqref{Yao0316} hold for every $\vartheta \in {J}_{\lambda} \cap (0, \vartheta_{B}(\lambda)]$ and $\lambda = \lambda_{\star\star}$ if $\lambda_{\star\star} > 0$. 
\end{proof}

\begin{rmk}
Based on \autoref{lma208}\ref{Yaoit23c} and its proof, the constant $\lambda_{\star\star}$ in \eqref{Yao0403} has the following properties: 
\begin{enumerate}[label = {\rm(\arabic*)}]
\item
$\lambda_{\star\star}$ can be characterized by
\begin{equation*}
\lambda_{\star\star} = \inf\{\lambda > 0: \; \pi/2 \leq \vartheta_{B}(\lambda) < (\pi + \vartheta_{A}(\lambda))/2\}; 
\end{equation*}
\item
$\lambda_{\star\star} = {l}_{\perp}$ if $\beta \leq 2\pi/3$, and $\lambda_{\star\star} = 0$ if $\pi/2 \leq \beta \leq 2\pi/3$. 
\end{enumerate}
\end{rmk}

\begin{thm} \label{thm403}
Let $0 < \beta < \alpha \leq \pi$ and
\begin{equation} \label{Yao0407}
\pi/2 \leq \beta \leq 2\pi/3. 
\end{equation}
Then the solution ${u}$ of \eqref{Yao0105} satisfies \eqref{Yao0316} for every $\vartheta \in {J}_{\lambda} \cap (0, \vartheta_{B}(\lambda)]$ and $\lambda > 0$. Furthermore, ${u}$ has the even symmetry and monotonicity properties as stated in \ref{Yaoit11a}, \ref{Yaoit11b}, and \ref{Yaoit11c} of \autoref{thm101main}. 
\end{thm}

\begin{proof}
We divide the proof into three parts. 

\textbf{Part 1}. 
\eqref{Yao0316} holds for every $\vartheta \in {J}_{\lambda} \cap (0, \vartheta_{B}(\lambda)]$ and $\lambda > 0$. In fact, by the assumption \eqref{Yao0407} and \autoref{lma208}\ref{Yaoit23c}, 
\begin{equation*}
0 < 2\vartheta_{B}(\lambda) - \pi < \vartheta_{A}(\lambda) \text{ for all } \lambda > 0. 
\end{equation*}
Hence $\lambda_{\star\star} = 0$. Then this part follows from \autoref{lma401}. 

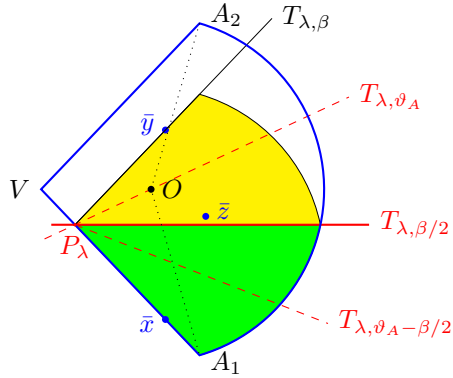
\begin{figure}[htp] \centering
\begin{tikzpicture}[scale = 2.3] 
\pgfmathsetmacro\ALPHA{73.73}; \pgfmathsetmacro\BETAs{46.35}; \pgfmathsetmacro\radius{1.00}; 
\pgfmathsetmacro\aaa{\radius*sin(abs(\BETAs - \ALPHA))/sin(\BETAs)}; 
\pgfmathsetmacro\xA{\aaa + \radius*cos(\ALPHA)}; \pgfmathsetmacro\yA{\radius*sin(\ALPHA)}; 
\pgfmathsetmacro\LenNeu{\radius*sin(\ALPHA)/sin(\BETAs)}; 
\pgfmathsetmacro\LAM{0.6161*\radius*sin(\ALPHA - \BETAs)/sin(2*\BETAs)}; 
\pgfmathsetmacro\xP{\LAM*cos( - \BETAs)}; \pgfmathsetmacro\yP{\LAM*sin( - \BETAs)}; 
\fill[fill = green, draw = black, very thin] 
(\xA, - \yA) arc ({ - \ALPHA} : {asin(\yP/\radius)} : \radius) -- (\xP, \yP) -- cycle; 
\fill[fill = yellow, draw = black, very thin] 
(\xA, 2*\yP + \yA) arc ({\ALPHA} : { - asin(\yP/\radius)} : \radius) -- (\xP, \yP) -- cycle; 
\pgfmathsetmacro\JiaoA{atan(\aaa*sin(\BETAs)/(\aaa*cos(\BETAs) - \LAM))}; 
\pgfmathsetmacro\Anglea{0 - \BETAs}; \pgfmathsetmacro\Angleb{\JiaoA - \BETAs - \BETAs}; 
\pgfmathsetmacro\Anglec{2*\JiaoA - 2*\BETAs - \BETAs}; \pgfmathsetmacro\Angled{\JiaoA - \BETAs}; 
\pgfmathsetmacro\Anglee{2*\BETAs - \BETAs}; 
\pgfmathsetmacro\lenx{(\LenNeu - \LAM)*0.7262}; 
\fill[blue] ({\xP + \lenx*cos(\Anglea)}, {\yP + \lenx*sin(\Anglea)}) circle (0.02) node[below = 2pt, left] { \footnotesize $\bar{x}$ }; 
\fill[blue] ({\xP + \lenx*cos(\Anglec)}, {\yP + \lenx*sin(\Anglec)}) circle (0.02) node[above = 2pt, right] { \footnotesize $\bar{z}$ }; 
\fill[blue] ({\xP + \lenx*cos(\Anglee)}, {\yP + \lenx*sin(\Anglee)}) circle (0.02) node[above = 2pt, left] { \footnotesize $\bar{y}$ }; 
\pgfmathsetmacro\ssa{0.4}; \pgfmathsetmacro\ssb{3.6}; 
\draw[red, dashed] ({ - \ssa*(\aaa - \xP) + \xP}, { - \ssa*(0 - \yP) + \yP}) -- ({\ssb*(\aaa - \xP) + \xP}, {\ssb*(0 - \yP) + \yP}) node[right] {\footnotesize ${T}_{\lambda, \vartheta_{A}}$}; 
\draw[red, thick] ({\xP - \xP*0.7}, {\yP}) -- ({\xP + \radius*1.7}, {\yP}) node[right] {\footnotesize ${T}_{\lambda, \beta/2}$}; 
\draw[red, dashed] ({\xP}, {\yP}) -- ({\xP + 1.57*\radius*cos(\Angleb)}, {\yP + 1.57*\radius*sin(\Angleb)}) node[right] {\footnotesize ${T}_{\lambda, \vartheta_{A} - \beta/2}$}; 
\draw[black, very thin] ({\xP}, {\yP}) -- ({\xP + 1.65*\radius*cos(\BETAs)}, {\yP + 1.65*\radius*sin(\BETAs)}) node[right] {\footnotesize ${T}_{\lambda, \beta}$}; 
\draw[blue, thick] ({\LenNeu*cos(\BETAs)}, {0 - \LenNeu*sin(\BETAs)}) arc ({ - \ALPHA} : {\ALPHA} : \radius) -- (0, 0) -- cycle; 
\draw[dotted] (\xA, \yA) -- (\aaa, 0) -- (\xA, - \yA); 
\fill (\aaa, 0) circle (0.02) node[right] { \footnotesize ${O}$ }; 
\fill[red] (\xP, \yP) circle (0.01) node[below] { \footnotesize ${P}_{\lambda}$ }; 
\node[left] at (0, 0) {\footnotesize ${V}$ }; 
\node[above = 3pt, right] at (\BETAs : \LenNeu) {\footnotesize ${A}_{2}$ }; 
\node[below = 3pt, right] at ( - \BETAs : \LenNeu) {\footnotesize ${A}_{1}$ }; 
\end{tikzpicture} 
\caption{$u({\color{cyan}\bar{x}}) < u({\color{cyan}\bar{y}}) < u({\color{cyan}\bar{z}})$ and ${w}^{\lambda, \beta/2} < 0$ in {\color{green}${D}_{\lambda, \beta/2, 0}$}. }
\label{Fig05Sym}
\end{figure}

\textbf{Part 2}. 
The even symmetry property of ${u}$. Indeed, we present two approaches to demonstrate the even symmetry of ${u}$. 

\emph{Method 1 (moving plane method)}. We will show
\begin{equation} \label{Yao0408}
{w}^{\lambda, \beta/2} < 0 \text{ in } \overline{{D}_{\lambda, \beta/2}} \setminus {T}_{\lambda, \beta/2}
\end{equation}
for every $\lambda > 0$. 
Indeed, note that $\Gamma_{\lambda, \beta/2}^{2B}$ is always empty. For $\lambda \geq \lambda_{\sharp}$, we have $(0, \beta] \subset {J}_{\lambda}$, so ${w}^{\lambda, \beta/2}$ satisfies \eqref{Yao0322}, and \eqref{Yao0408} holds. For $\lambda < \lambda_{\sharp}$, we will show 
\begin{equation} \label{Yao0409}
{w}^{\lambda, \beta/2} < 0 \text{ on } \Gamma_{\lambda, \beta/2}^{1} \cup \Gamma_{\lambda, \beta/2}^{2A}
\end{equation}
by the approach in \cite[Lemma 13]{YCG21}. 
Let us fix any $\lambda \in (0, \lambda_{\sharp})$ and any point $\bar{x} \in \Gamma_{\lambda, \beta/2}^{2A}$. Let $\bar{y} = \bar{x}^{\lambda, \beta/2}$ denote the reflection of $\bar{x}$ (with respect to ${T}_{\lambda, \beta/2}$), and let $\bar{z} = \bar{y}^{\lambda, \vartheta_{A}(\lambda)}$ be the reflection of $\bar{y}$ (with respect to ${T}_{\lambda, \vartheta_{A}(\lambda)}$); see \autoref{Fig05Sym}. Then $\bar{x} \in {T}_{\lambda, 0}$, $\bar{z} \in {T}_{\lambda, 2\vartheta_{A}(\lambda) - \beta}$ and $\bar{y} \in {T}_{\lambda, \beta} \cap \Sigma$. By \autoref{lma308}, 
\begin{gather*}
{w}^{\lambda, \vartheta_{A}(\lambda)} < 0 \text{ in } {D}_{\lambda, \vartheta_{A}(\lambda), 2\vartheta_{A}(\lambda) - \vartheta_{B}(\lambda)}, \\
{w}^{\lambda, \vartheta_{A}(\lambda) - \beta/2} < 0 \text{ in } \overline{{D}_{\lambda, \vartheta_{A}(\lambda) - \beta/2}} \setminus {T}_{\lambda, \vartheta_{A}(\lambda) - \beta/2}. 
\end{gather*}
It follows that
\begin{equation*}
{u}(\bar{x}) < {u}(\bar{z}) \quad \text{and }\quad {u}(\bar{z}) < {u}(\bar{y}), 
\end{equation*}
hence ${w}^{\lambda, \beta/2} < 0$ on $\Gamma_{\lambda, \beta/2}^{2A}$. Therefore, \eqref{Yao0409} holds for all $\lambda > 0$. Following the standard argument of the moving plane method (or as in the proof of \autoref{lma301}), we conclude that \eqref{Yao0408} holds for every $\lambda > 0$. By continuity, 
\begin{equation*}
{u}({x}_{1}, {x}_{2}) \leq {u}({x}_{1}, - {x}_{2}) \text{ in } \Sigma^{ - } = \{{x} \in \Sigma: \; {x}_{2} < 0\}. 
\end{equation*}
Similarly, moving planes along the upper Neumann boundary yields ${u}({x}_{1}, - {x}_{2}) \leq {u}({x}_{1}, {x}_{2})$ in $\Sigma^{ - }$, leading to the even symmetry of ${u}$ with respect to ${x}_{2}$, i.e., ${u}({x}) = {u}({x}_{1}, - {x}_{2})$ for ${x} \in \Sigma$. 

\emph{Method 2 (uniqueness of overdetermined problem)}. The proof is derived from step 2 in the proof of \cite[Theorem 2.9]{CLY21}. By taking $\vartheta = \beta/2$ in \eqref{Yao0315}, we get
\begin{equation*}
\begin{aligned}
{u}_{{x}_{2}} < 0 \text{ in } \Sigma^{ - } = \{{x} \in \Sigma: \; {x}_{2} < 0\}, \\
{u}_{{x}_{2}} > 0 \text{ in } \Sigma^{ + } = \{{x} \in \Sigma: \; {x}_{2} > 0\}. 
\end{aligned}
\end{equation*}
By continuity, 
\begin{equation} \label{Yao0411}
{u}_{{x}_{2}}({x}) = 0 \text{ for } {x} \in \Sigma \text{ with } {x}_{2} = 0. 
\end{equation}
Set ${v}({x}) = {u}({x}_{1}, - {x}_{2})$. 
Then ${u}$ and ${v}$ satisfy the same equation on the half domain $\Sigma^{ - } = \Sigma \cap \{{x}_{2} < 0\}$, and the same Dirichlet and Neumann data on $\{{x}_{2} = 0\}$: 
\begin{equation}
{v} = {u}, \quad {v}_{{x}_{2}} = {u}_{{x}_{2}} = 0 \quad \text{on} \quad \Sigma \cap \{{x}_{2} = 0\}. 
\end{equation}
Applying \cite[Theorem 1]{FV13}, we deduce that ${v} - {u}$ vanishes in $\Sigma^{ - }$, and hence the symmetry follows. 

\textbf{Part 3}. 
The monotonicity property of ${u}$ in the variable ${x}_{1}$. By the even symmetry in ${x}_{2}$, we can regard ${u}$ as a solution in the half domain $\Sigma^{ - }$ with Neumann boundary condition \eqref{Yao0411} on the symmetry axis $\{{x}_{2} = 0\}$. Since the domain is $\Sigma^{ - }$, the restriction $\lambda \geq \lambda_{\sharp}$ in \autoref{thm309} can be dropped, and by the same argument as in \autoref{thm309} one can show ${u}_{{x}_{1}} < 0$ in $\Sigma^{ - }$. For details, see \cite[Theorem 1.2 or Section 3.1]{CLY21}. 
\end{proof}


\section{The proof for $\beta \leq \pi/3$} \label{Sec05Small}

In~\cite{CLY21}, the key steps are to show that
\begin{equation} \label{Yao0502a}
{w}^{\lambda, \pi/2} < 0 \quad \text{in } {D}_{\lambda, \pi/2}
\end{equation}
and
\begin{equation} \label{Yao0502b}
{u}_{{x}_{1}}\cos\frac{\beta}{2} - {u}_{{x}_{2}}\sin\frac{\beta}{2} < 0 \quad \text{on } {T}_{\lambda, \pi/2} \cap \Sigma
\end{equation}
hold for all $\lambda > 0$. However, in the general case, \eqref{Yao0502b} (and thus \eqref{Yao0502a}) may not hold; in fact, a necessary condition for \eqref{Yao0502b} is $\pi/2 \in {J}_{\lambda}$. If $\alpha + \beta > \pi$ and $\vartheta_{B}(\lambda) > \pi/2$ and $\vartheta_{A}(\lambda) < \pi/2$, then ${J}_{\lambda}$ does not contain $\pi/2$, and \eqref{Yao0502b} fails. 

In this section, we focus on the case $\beta \leq \pi/3$ and show that
\begin{equation*}
{w}^{\lambda, \pi/2} < 0 \quad \text{in } {D}_{\lambda, \pi/2}
\end{equation*}
holds for $\lambda \leq {l}_{N}/2$. Based on a crucial observation, we will establish a continuous family linking ${w}^{{l}_{*}, (\pi + 3\beta)/4} < 0$ and ${w}^{{l}_{N}/2, \pi/2} < 0$, where ${l}_{*}$ is given in \autoref{lma501}. 

We begin this section by solving the equation $\check{\lambda}(\lambda, \vartheta) = {l}_{N}$. 

\begin{lma} \label{lma501}
Let ${l}_{*} \in (0, {l}_{N})$ be the constant such that
\begin{equation*}
\check{\lambda}({l}_{*}, \vartheta_{B}({l}_{*})) = {l}_{N}. 
\end{equation*}
For $0 < {\lambda} < {l}_{N}/(1 + \sin\beta)$, denote by $\vartheta_{\lambda} \in (\pi/4 + \beta/2, \pi/2 + \beta/2)$ the unique constant such that
\begin{equation} \label{Yao0504}
\check{\lambda}(\lambda, \vartheta_{\lambda}) = {l}_{N}. 
\end{equation}
Then $\vartheta_{\lambda}$ is strictly decreasing and continuous, $\hat{\lambda}(\lambda, \vartheta_{\lambda})$ is strictly increasing, $\vartheta_{{l}_{*}} = \vartheta_{B}({l}_{*}) = (\pi + 3\beta)/4$, and
\begin{equation} \label{Yao0505}
\vartheta_{\lambda} > \frac{\pi + 3\beta}{4} > \vartheta_{B}(\lambda) \text{ and } \check{\lambda}(\lambda, \vartheta) < {l}_{N}  \text{ for } \vartheta \in (\vartheta_{B}(\lambda), \vartheta_{\lambda}), \lambda \in (0, {l}_{*}). 
\end{equation}
\end{lma}

\begin{proof}
Since the function $\lambda \mapsto \check{\lambda}(\lambda, \vartheta_{B}(\lambda))$ is strictly increasing, continuous, and bijective from $(0, {l}_{N})$ onto $(0, \infty)$, there exists a unique ${l}_{*} \in (0, {l}_{N})$ such that $\check{\lambda}({l}_{*}, \vartheta_{B}({l}_{*})) = {l}_{N}$. Moreover, $\vartheta_{B}({l}_{*}) = (\pi + 3\beta)/4$ and
\begin{equation*}
\check{\lambda}(\lambda, \vartheta_{B}(\lambda)) > {l}_{N}, \quad \lambda \in ({l}_{*}, {l}_{N}) \quad \text{and} \quad
\check{\lambda}(\lambda, \vartheta_{B}(\lambda)) < {l}_{N}, \quad \lambda \in (0, {l}_{*}). 
\end{equation*}
It is easy to see that
\begin{equation*}
\begin{aligned}
\check{\lambda}(\lambda, \vartheta) & \text{ is strictly decreasing in } (\tfrac{\beta}{2}, \tfrac{\pi + 2\beta}{4}], \\
\check{\lambda}(\lambda, \vartheta) & \text{ is strictly increasing in } [\tfrac{\pi + 2\beta}{4}, \tfrac{\pi + \beta}{2}), 
\end{aligned}
\end{equation*}
and $\check{\lambda}(\lambda, \pi/2 + \beta/2) = \infty$, 
\begin{equation*}
\check{\lambda}(\lambda, (\pi + 2\beta)/4) < {l}_{N} \quad \text{if } 0 < \lambda < \frac{{l}_{N}}{1 + \sin\beta}. 
\end{equation*}
Thus, there exists a strictly decreasing and continuous map $\lambda \mapsto \vartheta_{\lambda}$ from $(0, {l}_{N}/(1 + \sin\beta))$ onto $(\pi/4 + \beta/2, \pi/2 + \beta/2)$ such that $\check{\lambda}(\lambda, \vartheta_{\lambda}) = {l}_{N}$. It is obvious that $\vartheta_{{l}_{*}} = (\pi + 3\beta)/4 = \vartheta_{B}({l}_{*})$, 
\begin{equation*}
\vartheta_{\lambda} > \vartheta_{{l}_{*}} = \vartheta_{B}({l}_{*}) > \vartheta_{B}(\lambda) \quad \text{for } \lambda \in (0, {l}_{*}). 
\end{equation*}
Since $\vartheta \mapsto \check{\lambda}(\lambda, \vartheta)$ is first decreasing then increasing, and
\begin{equation*}
\check{\lambda}(\lambda, \vartheta_{\lambda}) = {l}_{N}, \quad \check{\lambda}(\lambda, \vartheta_{B}(\lambda)) < {l}_{N} \quad \text{for } \lambda \in (0, {l}_{*}), 
\end{equation*}
we deduce that
\begin{equation*}
\check{\lambda}(\lambda, \vartheta) < {l}_{N} \quad \text{for } \vartheta \in (\vartheta_{B}(\lambda), \vartheta_{\lambda}), \; \lambda \in (0, {l}_{*}). 
\end{equation*}
Therefore, the proof is complete. 
\end{proof}

\begin{lma} \label{lma502}
Let $0 < \beta \leq \pi/3$. Then 
\begin{subequations} \label{Yao0508a}
\begin{gather} \label{Yao0508a1}
{w}^{\lambda, \vartheta_{B}(\lambda)} < 0 \text{ in } {D}_{\lambda, \vartheta_{B}(\lambda)} \quad \text{for } \lambda \in [{l}_{*}, {l}_{\perp}], 
\\ \label{Yao0508a2}
{w}^{\lambda, \vartheta_{\lambda}} < 0 \text{ in } {D}_{\lambda, \vartheta_{\lambda}} \quad \text{for } \lambda \in [{l}_{N}/2, {l}_{*}] 
\end{gather}
\end{subequations}
In particular, 
\begin{equation} \label{Yao0508b}
{w}^{{l}_{N}/2, \pi/2} < 0 \quad \text{in } {D}_{{l}_{N}/2, \pi/2}. 
\end{equation} 
\end{lma}

\begin{proof}
By the definition of $\lambda_{\sharp}$ and $\vartheta_{\lambda}$ in \eqref{Yao0341} and \eqref{Yao0504}, we have $\lambda_{\sharp} \leq {l}_{N}/2$ and $\vartheta_{{l}_{N}/2} = \pi/2$. Note that the assumption $0 < \beta \leq \pi/3$ implies ${l}_{*} \geq {l}_{N}/2$ and $\vartheta_{{l}_{*}} = (\pi + 3\beta)/4 \in (0, \pi/2]$. 

\textbf{Step 1}. 
\eqref{Yao0508a} holds on $\Gamma^{2A}$. In fact, we will verify that 
\begin{equation} \label{Yao0513}
{w}^{\lambda, \vartheta}({x}) = {u}({x}) - {u}({x}^{\lambda, \vartheta}) < 0 \text{ for } {x} \in \Gamma_{\lambda, \vartheta}^{2A}.
\end{equation}
holds for any fixed $\vartheta \in [\vartheta_{B}(\lambda), \pi/2]$ and $\lambda \in [\lambda_{\sharp}, {l}_{\perp}]$. We follow the approach in \cite[Lemma 13]{YCG21} and prove that
\begin{equation*}
{u}({x}) < {u}({z}_{1}) < {u}({z}_{2}) < \cdots < {u}({z}_{k}) < {u}({x}^{\lambda, \vartheta}). 
\end{equation*}
for several points ${z}_{1}, {z}_{2}, \ldots, {z}_{k}$. 
By the proof of \autoref{thm309}, we have: 
\begin{subequations}
\begin{align}
\label{Yao0511a}
{w}^{\lambda, \vartheta_{A}(\lambda)} < 0 & \quad \text{in } {D}_{\lambda, \vartheta_{A}(\lambda), 2\vartheta_{A}(\lambda) - \vartheta_{B}(\lambda)} \quad \text{for } \lambda_{\sharp} \leq \lambda < \lambda_{C}, \\
\label{Yao0511b}
{w}^{\lambda, \vartheta_{B}(\lambda)} < 0 & \quad \text{in } \overline{{D}_{\lambda, \vartheta_{B}(\lambda), 2\vartheta_{B}(\lambda) - (\pi + \beta)/2}} \setminus {T}_{\lambda, \vartheta_{B}(\lambda)} \quad \text{for } \lambda_{\sharp} \leq \lambda < {l}_{\perp}, \\
\label{Yao0511c}
{w}^{\lambda, (\pi + \beta)/2} < 0 & \quad \text{in } \overline{{D}_{\lambda, (\pi + \beta)/2}} \setminus {T}_{\lambda, (\pi + \beta)/2} \quad \text{for } \lambda_{\sharp} \leq \lambda, 
\end{align}
\end{subequations}
and monotonicity: 
\begin{subequations} \label{Yao0511d}
\begin{align}
\label{Yao0511d1}
\vartheta \in {S}_{r}^{1} \mapsto {u}({r}, \vartheta) & \text{ is strictly increasing}, \\
\label{Yao0511d2}
\vartheta \in {S}_{r}^{2} \mapsto {u}({r}, \vartheta) & \text{ is strictly increasing}, 
\end{align}
\end{subequations}
where ${r} = |{x} - {P}_{\lambda}|$ and $\vartheta = \sigma_{\lambda}({x})$ is the polar coordinates as in \eqref{Yao0306} and ${S}_{r}^{1} = \{\vartheta \in [0, \vartheta^{A}]: ({r}, \vartheta) \in \overline{\Sigma}\}$ and ${S}_{r}^{2} = \{\vartheta \in [\vartheta_{B}, (\pi + \beta)/2]: ({r}, \vartheta) \in \overline{\Sigma}\}$. Note that $\Sigma \cap \{{x}: |{x} - {P}_{\lambda}| = {r}, \sigma_{\lambda}({x}) \in (0, \pi/2 + \beta)\}$ is connected for any $\lambda > 0$. 

We focus on the case $\vartheta_{A} < \vartheta_{B}$ (the case $\vartheta_{A} \geq \vartheta_{B}$ is similar). 
Let ${x}$ be an arbitrary fixed point on $\Gamma_{\lambda, \vartheta}^{2A}$. Since $\vartheta \leq \pi/2$, we know $\Gamma_{\lambda, \vartheta}^{2A} \subset \Gamma_{N}^{ - }$. Then ${x}$ and its reflection ${x}^{\lambda, \vartheta}$ are denoted by $({r}, \psi_{1})$ and $({r}, \psi_{2})$ where $\psi_{1} = 0$ and $\psi_{2} = 2\vartheta$. Let $\gamma = (\pi + \beta)/2$, and denote $\psi_{2}' = 2\gamma - \psi_{2}$. There are four cases: 

Case 1: $\psi_{2}' = 2\gamma - 2\vartheta \in (0, \vartheta_{A}]$. We observe that the point $({r}, \psi_{2}')$ belongs to $\Sigma$ since $\psi_{2} \geq \gamma$ and $|{P}_{\lambda}{P}_{ - }| < |{P}_{\lambda}{P}_{ + }|$. By \eqref{Yao0511c} and \eqref{Yao0511d1},
\begin{equation*}
{u}({r}, \psi_{2}) > {u}({r}, \psi_{2}') > {u}({r}, \psi_{1}). 
\end{equation*}

Case 2: $\psi_{2}' \in (\vartheta_{A}, \vartheta_{B}]$. Then $\psi_{2}'' = 2\vartheta_{A} - \psi_{2}' \geq 2\vartheta_{A} - \vartheta_{B} > 0$ by \autoref{lma208}\ref{Yaoit23a}. By \eqref{Yao0511c}, \eqref{Yao0511a} and \eqref{Yao0511d1}
\begin{equation*}
{u}({r}, \psi_{2}) > {u}({r}, \psi_{2}') > {u}({r}, \psi_{2}'') > {u}({r}, \psi_{1}). 
\end{equation*}

Case 3: $\psi_{2}' \in (\vartheta_{B}, \gamma)$. Then by \eqref{Yao0511c}, \eqref{Yao0511d2}, \eqref{Yao0511a} and \eqref{Yao0511d1}
\begin{equation*}
{u}({r}, \psi_{2}) > {u}({r}, \psi_{2}') > {u}({r}, \vartheta_{B}) > {u}({r}, 2\vartheta_{A} - \vartheta_{B}) > {u}({r}, \psi_{1}). 
\end{equation*}

Case 4: $\psi_{2} \in [\vartheta_{B}, \gamma]$. Then by \eqref{Yao0511d2}, \eqref{Yao0511a} and \eqref{Yao0511d1}
\begin{equation*}
{u}({r}, \psi_{2}) \geq {u}({r}, \vartheta_{B}) > {u}({r}, 2\vartheta_{A} - \vartheta_{B}) > {u}({r}, \psi_{1}). 
\end{equation*}

In all cases, we obtain ${u}({r}, 0) < {u}({r}, 2\vartheta)$, so \eqref{Yao0513} holds. Step 1 is finished. 

\textbf{Step 2}. 
We claim that \eqref{Yao0508a} holds. Set 
\begin{equation*}
\phi_{2,\lambda} = \vartheta_{B}(\lambda) \text{ for } \lambda \in [{l}_{*}, {l}_{\perp}] \text{ and }
\phi_{2,\lambda} = \vartheta_{\lambda} \text{ for } \lambda \in [{l}_{N}/2, {l}_{*}]
\end{equation*}
From \autoref{thm309}, ${w}^{\lambda, \phi_{2,\lambda}}$ satisfies the Neumann boundary on $\Gamma_{\lambda, \phi_{2,\lambda}}^{2B}$ for $\lambda \in [{l}_{N}/2, {l}_{\perp}]$. From Step 1, ${w}^{\lambda, \phi_{2,\lambda}}$ satisfies
\begin{equation*}
\begin{cases}
\Delta {w}^{\lambda, \phi_{2,\lambda}} + {c}^{\lambda, \phi_{2,\lambda}}{w}^{\lambda, \phi_{2,\lambda}} = 0 & \text{in } {D}_{\lambda, \phi_{2,\lambda}}, \\
{w}^{\lambda, \phi_{2,\lambda}} = 0 & \text{on } \Gamma_{\lambda, \phi_{2,\lambda}}^{0}, \\
{w}^{\lambda, \phi_{2,\lambda}} < 0 & \text{on } \Gamma_{\lambda, \phi_{2,\lambda}}^{1} \cup \Gamma_{\lambda, \phi_{2,\lambda}}^{2A}, \\
\nabla {w}^{\lambda, \phi_{2,\lambda}} \cdot \nu < 0 & \text{on } \Gamma_{\lambda, \phi_{2,\lambda}}^{2B}.
\end{cases}
\end{equation*}
In \eqref{Yao0322d}, ${w}^{\lambda, \vartheta}$ satisfies a Neumann-type boundary condition on $\Gamma^{2A}$, while here ${w}^{\lambda, \phi_{2,\lambda}}$ satisfies a Dirichlet-type boundary condition on $\Gamma^{2A}$. From \autoref{lma401}, ${w}^{\lambda, \pi/2}$ is negative in ${D}_{\lambda, \pi/2}$ for $\lambda = {l}_{\perp}$. By the same arguments as in \autoref{lma301}, one can show the negativity of ${w}^{\lambda, \phi_{2,\lambda}}$ for $\lambda \in [{l}_{N}/2, {l}_{\perp}]$. 
This completes the proof. 
\end{proof}

\begin{thm} 
Let $0 < \beta < \alpha \leq \pi$ and
\begin{equation} \label{Yao0515}
0 < \beta \leq \pi/3. 
\end{equation}
Then the solution ${u}$ of \eqref{Yao0105} satisfies
\eqref{Yao0316} for every $\vartheta \in {J}_{\lambda}$ and $\lambda > 0$. 
Furthermore, ${u}$ has the even symmetry and monotonicity properties as stated in \ref{Yaoit11a}, \ref{Yaoit11b}, and \ref{Yaoit11c} of \autoref{thm101main}. 
\end{thm}

\begin{proof}
The proof is the same as \autoref{thm309}, except for the proof of \eqref{Yao0315} with $\vartheta = (\pi + \beta)/2$. In fact, the assumption that $\lambda \geq \lambda_{\sharp}$ in \autoref{thm309} is used in two situations: 
(1) $2\vartheta_{B} - \pi < \vartheta_{A}$; (2) the proof of \eqref{Yao0315} with $\vartheta = (\pi + \beta)/2$. 
However, in the situation $0 < \beta \leq \pi/3$, we know that $2\vartheta_{B}(\lambda) - \pi < \vartheta_{A}(\lambda)$ is valid for all $\lambda > 0$; see \autoref{lma208}\ref{Yaoit23c}. It remains to check the negativity of ${u}_{{x}_{1}}$. 

Let $\bar{\Lambda}$ be as in \eqref{Yao0355}, i.e., 
\begin{equation*}
\bar{\Lambda} = \inf\big\{\lambda' > 0: \; \eqref{Yao0315} \text{ holds for all } \vartheta \in {J}_{\lambda} \text{ and } \lambda \geq \lambda' \big\}. 
\end{equation*}
To show $\bar{\Lambda} = 0$, we observe that all steps in \autoref{thm309} are valid except the monotonicity in ${x}_{1}$. 
To complete the proof, we only need to check that
\begin{equation} \label{Yao0516}
{u}_{{x}_{1}} < 0 \quad \text{in } {T}_{\lambda, (\pi + \beta)/2} \cap \Sigma
\end{equation}
for $\lambda = \bar{\Lambda}$. In fact, \autoref{thm309} implies that $\bar{\Lambda} \leq \lambda_{\sharp}$ and thus $\bar{\Lambda} \leq {l}_{N}/2$. For $\lambda \in [\bar{\Lambda}\cos\beta, {l}_{N}/2]$, 
\begin{equation*}
\hat{\lambda}(\lambda, \pi/2) = \lambda \sec\beta \in [\bar{\Lambda}, {l}_{N}], \quad
\check{\lambda}(\lambda, \pi/2) = 2\lambda \in [\bar{\Lambda}, {l}_{N}], 
\end{equation*}
where the condition \eqref{Yao0515} is used. Hence, ${w}^{\lambda, \pi/2}$ satisfies \eqref{Yao0322}. Combining this with \eqref{Yao0508b} in \autoref{lma502}, and applying \autoref{lma301}, we conclude that
\begin{equation*}
{w}^{\lambda, \pi/2} < 0 \quad \text{in } {D}_{\lambda, \pi/2}
\end{equation*}
and \eqref{Yao0315a} (and similarly \eqref{Yao0315b}) hold for $\vartheta = \pi/2$, $\bar{\Lambda}\cos\beta \leq \lambda \leq {l}_{N}/2$. By adding \eqref{Yao0315a} and \eqref{Yao0315b}, we see that ${u}_{{x}_{1}} < 0$ holds for all points in the closure of the intersection of $\Sigma_{\bar{\Lambda}\cos\beta, \pi/2} \setminus \Sigma_{{l}_{N}/2, \pi/2}$ and $\Sigma_{\bar{\Lambda}, \beta + \pi/2} \setminus \Sigma_{{l}_{N}/(2\cos\beta), \beta + \pi/2}$. 
In particular, \eqref{Yao0516} is true for $\lambda \in [\bar{\Lambda}, {l}_{N}/2]$. 

Thus, step 1 of \autoref{thm309} is established. All other steps in the proof of \autoref{thm309} also remain valid. Therefore, $\bar{\Lambda} = 0$ and \eqref{Yao0315} holds for every $\vartheta \in {J}_{\lambda}$ and $\lambda > 0$. 

The symmetry property of ${u}$ follows from part 2 of \autoref{thm403}. 
\end{proof}


\section{The proof for $\pi/3 < \beta < \pi/2$} \label{Sec05Mid}

As in the proof for the case $\beta \in [\pi/2, 2\pi/3]$, the monotonicity properties \eqref{Yao0315} with $\vartheta = \vartheta_{B}(\lambda)$ can be obtained when $\lambda \geq {l}_{\perp}$. However, this approach does not work when $\lambda < {l}_{\perp}$. To overcome this difficulty, a new approach is needed. The main observations are as follows: 
\begin{itemize}
\item Establish the negativity of ${w}^{\lambda, \vartheta_{B}(\lambda)}$ in a domain larger than ${D}_{\lambda, \vartheta_{B}(\lambda)}$; 
\item Establish the monotonicity \eqref{Yao0315} on a subset of ${T}_{\lambda, \vartheta} \cap \Sigma$ for $\vartheta \not \in {J}_{\lambda}$. 
\end{itemize}

\subsection{The case $\lambda \geq {l}_{\perp}$}

At the beginning of this section, we improve \autoref{lma401} and obtain \eqref{Yao0316} for $\lambda \geq {l}_{\perp}$ and $\vartheta \in {J}_{\lambda} \cap (0, \bar{\omega}_{\lambda}]$ with a larger constant $\bar{\omega}_{\lambda}$ defined by
\begin{equation} \label{Yao0602}
\bar{\omega}_{\lambda} = 
\begin{cases}
\vartheta^{B}(\lambda) & \text{if } \lambda > {l}_{*}, \\
\min\left\{ \frac{\pi + \vartheta_{A}(\lambda)}{2}, \vartheta_{\lambda} \right\} & \text{if } \lambda \leq {l}_{*}, 
\end{cases}
\end{equation}
where $\vartheta_{\lambda}$ is defined in \autoref{lma501}. 

\begin{lma} \label{lma601}
Let $\vartheta_{\lambda}$ be given in \autoref{lma501} and let $\bar{\omega}_{\lambda}$ be as in \eqref{Yao0602}. Then we have: 
\begin{enumerate}[label = {\rm(\alph*)}]
\item
The map $\lambda \in (0, \infty) \mapsto \bar{\omega}_{\lambda}$ is continuous, $\bar{\omega}_{\lambda}$ is increasing-decreasing over $(0, {l}_{*}]$, and
\begin{gather*}
\bar{\omega}_{\lambda} > \vartheta_{B}(\lambda), \quad \lambda \in (0, {l}_{*}) \text{ when } \beta \leq \frac{2\pi}{3}, \\
\bar{\omega}_{\lambda} > \frac{\pi + 3\beta}{4}, \quad \lambda \in (0, {l}_{*}) \text{ when } \beta \leq \frac{\pi}{2}. 
\end{gather*}

\item \label{Yaoit62b}
$\vartheta_{B}(\lambda) + \vartheta_{\lambda} > (\pi + 3\beta)/2$ for $\lambda \in (0, {l}_{*})$. 

\item
The map $\lambda \in (0, {l}_{*}] \mapsto \hat{\lambda}(\lambda, \bar{\omega}_{\lambda})$ is strictly increasing whenever $\beta \leq \pi/2$. 
\end{enumerate}
\end{lma}

\begin{proof}
From \autoref{lma501}, $\vartheta_{\lambda}$ is decreasing for $0 < \lambda < {l}_{N}/(1 + \sin\beta)$. Together with the fact that $\vartheta_{A}(\lambda)$ is increasing in $\lambda$, we deduce that $\pi + \vartheta_{A}(\lambda) - 2\vartheta_{\lambda}$ is strictly increasing for $0 < {\lambda} < {l}_{N}/(1 + \sin\beta)$. 
Since $\lim_{\lambda \to 0} \vartheta_{A}(\lambda) = \beta/2$ and $\lim_{\lambda \to 0} \vartheta_{\lambda} = (\pi + \beta)/2$, we have
\begin{equation*}
\pi + \vartheta_{A}(\lambda) - 2\vartheta_{\lambda} < 0 \text{ for } 0 < \lambda \ll 1. 
\end{equation*}
Note that $\vartheta_{B}({l}_{*}) = \vartheta_{{l}_{*}} = (\pi + 3\beta)/4$ and $\vartheta_{A}({l}_{*}) \geq \beta/2 + \operatorname{arccot}(2\cos(\beta/2))$, and
\begin{equation*}
\pi + \vartheta_{A}(\lambda) - 2\vartheta_{\lambda} \geq (\pi - \beta) - \arctan(2\sin\tfrac{\pi - \beta}{2}) > 0 \text{ at } \lambda = {l}_{*}. 
\end{equation*}
It follows that $\pi + \vartheta_{A}(\lambda) - 2\vartheta_{\lambda} = 0$ has a unique zero in $(0, {l}_{N}/(1 + \sin\beta)]$, and this unique zero belongs to $(0, {l}_{*})$. Moreover, $\bar{\omega}_{\lambda}$ depends continuously on $\lambda \in \mathbb{R}^{ + }$, $\bar{\omega}_{\lambda}$ is increasing then decreasing on $(0, {l}_{*}]$, and
\begin{equation*}
\bar{\omega}_{\lambda} > \frac{\pi + 3\beta}{4}, \quad \lambda \in (0, {l}_{*}) \text{ when } \beta \leq \frac{\pi}{2}. 
\end{equation*}
By \autoref{lma208}\ref{Yaoit23c}, 
\begin{equation*}
\bar{\omega}_{\lambda} > \vartheta_{B}(\lambda), \quad \lambda \in (0, {l}_{*}] \text{ when } \beta \leq \frac{2\pi}{3}. 
\end{equation*}

Now we estimate the lower bound of the sum $\vartheta_{B}(\lambda) + \vartheta_{\lambda}$. According to the definitions, we can write
\begin{align*}
\cot(\vartheta_{B}(\lambda)) = \frac{\cos\beta - \frac{\sin(2\vartheta_{\lambda} - \beta)}{\sin(2\vartheta_{\lambda} - 2\beta)} \frac{\sin(\vartheta_{\lambda} - \beta)}{\sin\vartheta_{\lambda}}}{\sin\beta} = 
\cot\beta - \frac{\sin(2\vartheta_{\lambda} - \beta)} {2\sin\beta \sin\vartheta_{\lambda} \cos(\vartheta_{\lambda} - \beta)}. 
\end{align*}
Thus, 
\begin{align*}
&\left( \cot(\vartheta_{B}(\lambda)) - \cot\left( \tfrac{\pi + 3\beta}{2} - \vartheta_{\lambda} \right) \right) \sin\beta \sin\vartheta_{\lambda} \cos(\vartheta_{\lambda} - \beta)\cos(\vartheta_{\lambda} - \tfrac{3\beta}{2})\\
 = &
\cos(\vartheta_{\lambda} - \tfrac{\beta}{2})\sin\vartheta_{\lambda}\cos(\vartheta_{\lambda} - \beta)
 - \tfrac{1}{2} \sin(2\vartheta_{\lambda} - \beta)\cos(\vartheta_{\lambda} - \tfrac{3\beta}{2}) \\
 = &
\cos(\vartheta_{\lambda} - \tfrac{\beta}{2})\left( \sin\vartheta_{\lambda}\cos(\vartheta_{\lambda} - \beta)
 - \sin(\vartheta_{\lambda} - \tfrac{\beta}{2})\cos(\vartheta_{\lambda} - \tfrac{3\beta}{2}) \right) \\
 = &
\cos(\vartheta_{\lambda} - \tfrac{\beta}{2}) \left( \tfrac{1}{2}\sin(2\vartheta_{\lambda} - \beta) - \tfrac{1}{2}\sin(2\vartheta_{\lambda} - 2\beta) \right) \\
 = &
\cos(\vartheta_{\lambda} - \tfrac{\beta}{2})\cos(2\vartheta_{\lambda} - \tfrac{3\beta}{2})\sin\tfrac{\beta}{2}. 
\end{align*}
Recalling that $(\pi + 3\beta)/4 < \vartheta_{\lambda} < (\pi + \beta)/2$ for $\lambda \in (0, {l}_{*})$, we get $\cot(\vartheta_{B}(\lambda)) < \cot((\pi + 3\beta)/2 - \vartheta_{\lambda})$, and hence
\begin{equation*}
\vartheta_{\lambda} + \vartheta_{B}(\lambda) > \frac{\pi + 3\beta}{2} \text{ for } \lambda \in (0, {l}_{*}). 
\end{equation*}

Lastly, we turn to the monotonicity of $\hat{\lambda}(\lambda, \bar{\omega}_{\lambda})$. One can compute that $\hat{\lambda}(\lambda, (\pi + \vartheta_{A}(\lambda))/2) = |{V}{O}| \cdot \hbar(\vartheta_{A}(\lambda))$ where
\begin{equation*}
\hbar(\vartheta) = \frac{\sin(\vartheta - \frac{\beta}{2})}{\sin\vartheta} \cdot \frac{\sin\frac{\pi - \vartheta}{2}}{ \sin\frac{\pi + \vartheta - 2\beta}{2}} = \frac{\sin(\vartheta - \frac{\beta}{2})}{\sin\beta + \sin(\vartheta - \beta)}. 
\end{equation*}
Note that
\begin{equation*}
\frac{\partial }{\partial \vartheta} \hbar(\vartheta) = \frac{\cos(\vartheta - \frac{\beta}{2}) \sin\beta - \sin\frac{\beta}{2}}{(\sin\beta + \sin(\vartheta - \beta))^{2}} > 0 \text{ for } \tfrac{\beta}{2} < \vartheta < \beta,
\end{equation*}
where $\beta \leq \pi/2$ is used here. It immediately follows that the function $\lambda \mapsto \hat{\lambda}(\lambda, (\pi + \vartheta_{A}(\lambda))/2)$ is strictly increasing when $\vartheta_{A}(\lambda) \leq \beta$, i.e., $\lambda \in (0, \lambda_{\sharp}]$. It is clear that $\vartheta_{\lambda} < (\pi + \beta)/2 \leq (\pi + \vartheta_{A}(\lambda))/2$ and $\bar{\omega}_{\lambda} = \vartheta_{\lambda}$ for $\lambda \in [\lambda_{\sharp}, {l}_{*}]$. Combining these properties and the fact that $\hat{\lambda}(\lambda, \vartheta_{\lambda})$ is strictly increasing for $\lambda \in (0, {l}_{*}]$, we conclude that
\begin{equation*}
\hat{\lambda}(\lambda, \bar{\omega}_{\lambda}) = \max \left\{ \hat{\lambda}\Big(\lambda, \frac{\pi + \vartheta_{A}(\lambda)}{2}\Big), \; \hat{\lambda}(\lambda, \vartheta_{\lambda}) \right\}
\end{equation*}
is strictly increasing for $\lambda \in (0, {l}_{*}]$ when $\beta \leq \pi/2$. 
\end{proof}

\begin{defn}
Define $\imath(\lambda)$ such that
\begin{equation*}
\imath(\lambda) = \hat{\lambda}(\lambda, \bar{\omega}_{\lambda}), \quad \lambda \in (0, \infty)
\end{equation*}
and denote by $\jmath$ the inverse function of $\imath: (0, {l}_{*}] \to \mathbb{R}^{ + }$, i.e., $\jmath(\lambda)$ is determined by
\begin{equation*}
\imath(\jmath(\lambda)) = \lambda, \quad \lambda \in (0, {l}_{N}]. 
\end{equation*}
\end{defn}

\begin{defn}
For convenience and further use, define
\begin{equation} \label{Yao0607}
\begin{aligned}
\mathfrak{S}^{\Lambda} & = \mathfrak{S}^{\Lambda}_{1} \cap \mathfrak{S}^{\Lambda}_{2}, \\
\mathfrak{S}^{\Lambda}_{1} & = \{(\lambda, \vartheta): \; \lambda \geq \Lambda, 0 < \vartheta \leq \vartheta^{A}(\lambda)\}, \\
\mathfrak{S}^{\Lambda}_{2} & = \{(\lambda, \vartheta): \; \Lambda \leq \lambda \leq \hat{\lambda}(\lambda, \vartheta), \; \imath(\Lambda) \leq \hat{\lambda}(\lambda, \vartheta) \leq {l}_{N} \}
\end{aligned}
\end{equation}
for every $\Lambda > 0$, where $\vartheta^{A}(\lambda) = \min\{\vartheta_{A}(\lambda), \vartheta_{B}(\lambda), (\pi + \beta)/2\}$ as in \eqref{Yao0314}. 
\end{defn}

\begin{rmks}
By this definition of $\mathfrak{S}^{\Lambda}_{2}$, it follows that $\mathfrak{S}^{\Lambda}_{2}$ is empty for $\Lambda \geq {l}_{N}$, and
\eqref{Yao0315} holds for $(\lambda, \vartheta) \in \mathfrak{S}^{\Lambda}_{2}$ if and only if \eqref{Yao0315} holds for $\lambda$ and $\vartheta$ satisfying
\begin{equation}
\Lambda \leq \min\{\lambda, \hat{\lambda}(\lambda, \vartheta) \} \leq {l}_{N}, \quad \imath(\Lambda) \leq \max\{\lambda, \hat{\lambda}(\lambda, \vartheta) \} \leq {l}_{N}. 
\end{equation}
\end{rmks}

\begin{lma} \label{lma604A}
Let $\Lambda \in (0, {l}_{*})$. 
Suppose that \eqref{Yao0315} holds for $\vartheta \in {J}_{\lambda} \cap (0, \bar{\omega}_{\lambda}]$ and $\lambda \in [\Lambda, \infty)$. 
Then \eqref{Yao0315} holds for $(\lambda, \vartheta) \in \mathfrak{S}^{\Lambda}$, where $\mathfrak{S}^{\Lambda}$ is given in \eqref{Yao0607}. 
\end{lma}

\begin{proof}
The case $\vartheta \leq \vartheta^{A}(\lambda)$ is proved in \autoref{lma308}. 
It suffices to show that \eqref{Yao0315a} holds when $\vartheta > \bar{\omega}_{\lambda}$ and $(\lambda, \vartheta) \in \mathfrak{S}^{\Lambda}_{2}$, 
which is equivalent to \eqref{Yao0315b} holding for $\vartheta \in [(\pi + \beta)/2, \pi + \beta - \bar{\omega}_{\jmath(\lambda)}]$ and $\lambda \in [\imath(\Lambda), {l}_{N})$. 

To this end, denote by ${\Omega}^{\lambda}$ the triangle domain enclosed by $\hat{T}_{\lambda, \vartheta_{B}(\lambda)}$, $\hat{T}_{\lambda, \pi + \beta - \bar{\omega}_{\jmath(\lambda)}}$ (which is the line ${T}_{\jmath(\lambda), \bar{\omega}_{\jmath(\lambda)}}$), and $\Gamma_{N}^{ - }$ for $\lambda \leq {l}_{N}$. 
Set $\bar{{x}}^{\lambda} = (\bar{{x}}_{1}^{\lambda}, \bar{{x}}_{2}^{\lambda})$ with $\bar{{x}}_{1}^{\lambda} = \lambda \cos(\beta/2)$ and $\bar{{x}}_{2}^{\lambda} = \lambda \sin(\beta/2)$, and define
\begin{equation*}
{v}^{\lambda}({x}) = ({x}_{1} - \bar{{x}}_{1}^{\lambda}){u}_{{x}_{2}}({x}) - ({x}_{2} - \bar{{x}}_{2}^{\lambda}) {u}_{{x}_{1}}({x}). 
\end{equation*}
It is clear that the angular derivative function ${v}^{\lambda}$ satisfies the linear equation
\begin{equation*}
\mathcal{L}[{v}^{\lambda}] = [\Delta + f'({u})]{v}^{\lambda} = 0 \quad \text{in } \Sigma. 
\end{equation*}
By the assumption of this lemma, there holds
\begin{equation*}
{v}^{{l}_{N}} < 0 \quad \text{in } {\Omega}^{{l}_{N}}
\end{equation*}
and
\begin{equation} \label{Yao0610b}
{v}^{\lambda} \leq, \not\equiv 0 \quad \text{on } \partial {\Omega}^{\lambda}
\end{equation}
for $\imath(\Lambda) \leq \lambda \leq {l}_{N}$. 

Let $\eta$ be the positive constant such that the maximum principle holds for the operator $\mathcal{L}$ in domains of volume less than $\eta$ (e.g., Proposition 1.1 of~\cite{BN91}). Choose any fixed compact set ${K} \subset {\Omega}^{{l}_{N}}$ so that the measure $|{\Omega}^{{l}_{N}} \setminus {K}| < \eta/2$. Since ${v}^{{l}_{N}} < 0$ in ${K}$, there exists a small positive constant $\epsilon_{1}$ such that ${K} \subset {\Omega}^{\lambda}$ and
\begin{equation*}
|{\Omega}^{\lambda} \setminus {K}| < \eta, \quad {v}^{\lambda} < 0 \quad \text{in } {K}
\end{equation*}
for $\lambda \in [{l}_{N} - \epsilon_{1}, {l}_{N})$. 
With this and \eqref{Yao0610b}, we have
\begin{equation*}
\begin{cases}
\mathcal{L}[{v}^{\lambda}] = 0 & \text{in } {\Omega}^{\lambda} \setminus {K}, \\
{v}^{\lambda} \leq, \not\equiv 0 & \text{on } \partial ({\Omega}^{\lambda} \setminus {K}). 
\end{cases}
\end{equation*}
It then follows from the maximum principle in \cite[Proposition 1.1]{BN91} that ${v}^{\lambda} < 0$ in ${\Omega}^{\lambda} \setminus {K}$ and hence 
\begin{equation} \label{Yao0611}
{v}^{\lambda} < 0 \quad \text{in } {\Omega}^{\lambda} \text{ for } \lambda \in [{l}_{N} - \epsilon_{1}, {l}_{N}).
\end{equation} 

Let $(\bar{\lambda}, {l}_{N})$ be the largest open interval of $\lambda$ such that \eqref{Yao0611} holds. We claim that $\bar{\lambda} \leq \imath(\Lambda)$. Suppose, for contradiction, that $\bar{\lambda} > \imath(\Lambda)$. By continuity, ${v}^{\bar{\lambda}} \leq 0$ in ${\Omega}^{\bar{\lambda}}$. Since ${v}^{\bar{\lambda}} \not\equiv 0$ on $\partial {\Omega}^{\bar{\lambda}}$, the strong maximum principle yields
\begin{equation*}
{v}^{\bar{\lambda}} < 0 \quad \text{in } {\Omega}^{\bar{\lambda}}. 
\end{equation*}
By repeating the above argument, we see that \eqref{Yao0611} holds for $\lambda \in [\bar{\lambda} - \epsilon_{2}, \bar{\lambda})$ for some small $\epsilon_{2} > 0$, contradicting the maximality of $(\bar{\lambda}, {l}_{N})$. Hence, $\bar{\lambda} \leq \imath(\Lambda)$, and so \eqref{Yao0611} holds for $\lambda \in [\imath(\Lambda), {l}_{N}]$. 

This establishes that \eqref{Yao0315a} holds for $(\lambda, \vartheta) \in \mathfrak{S}^{\Lambda}_{2}$. Similarly, one can deduce \eqref{Yao0315b} for $(\lambda, \vartheta) \in \mathfrak{S}^{\Lambda}_{2}$. 
\end{proof}

\begin{lma} \label{lma604B}
Let $0 < \beta < \alpha \leq \pi$ and
\begin{equation*}
\beta \in (\pi/3, \pi/2). 
\end{equation*}
Then \eqref{Yao0316} holds for $\vartheta \in {J}_{\lambda} \cap (0, \bar{\omega}_{\lambda}]$ and $\lambda \geq {l}_{\perp}$. Moreover, \eqref{Yao0316} holds for
$(\lambda, \vartheta) \in \mathfrak{S}^{\Lambda}$ with $\Lambda = {l}_{\perp}$ where $\mathfrak{S}^{\Lambda}$ is defined in \eqref{Yao0607}. 
\end{lma}

\begin{proof}
From \autoref{lma401}, we know that \eqref{Yao0316} holds for $\vartheta \in (0, \vartheta^{A}(\lambda)] \cup \{\vartheta_{B}(\lambda)\}$ and $\lambda \geq {l}_{\perp}$, and
\begin{equation} \label{Yao0614b}
{w}^{\lambda, \vartheta_{B}(\lambda)} < 0 \text{ in } {D}_{\lambda, \vartheta_{B}(\lambda)} \text{ for } \lambda = {l}_{*}
\end{equation}
since ${l}_{*} \in [{l}_{\perp}, \infty)$ and $\beta > \pi/3$. 

\textbf{Step 1}. 
We claim that \eqref{Yao0316} holds for $(\lambda, \vartheta) \in \mathfrak{S}^{\Lambda_{1}}_{2}$ with
\begin{equation} \label{Yao0615a}
\Lambda_{1} = \min\big\{ \lambda \in [{l}_{\perp}, {l}_{*}]: \; 
\lambda_{*}(\lambda, \vartheta) \geq \lambda_{\sharp} \text{ for every } \vartheta \in [\vartheta_{B}(\lambda), \bar{\omega}_{\lambda}]
\big\}, 
\end{equation}
where we remark that $\Lambda_{1} < {l}_{*}$. 
In fact, let $\lambda$ and $\vartheta$ satisfy
\begin{equation*}
\lambda \in [\Lambda_{1}, {l}_{*}] \quad \text{and }\quad \vartheta \in [\vartheta_{B}(\lambda), \bar{\omega}_{\lambda}]. 
\end{equation*}
Then we have
\begin{equation*}
2\vartheta - \pi \in [0, \vartheta_{A}(\lambda)], \quad
\lambda_{\sharp} \leq \hat{\lambda}(\lambda, \vartheta) \leq {l}_{N}, \lambda_{\sharp} \leq \check{\lambda}(\lambda, \vartheta) \leq {l}_{N}. 
\end{equation*}
Following by \autoref{lma401}, ${w}^{\lambda, \vartheta}$ satisfies the boundary condition on $\Gamma_{\lambda, \vartheta}^{2A}$. 
Following by \autoref{thm309}, ${w}^{\lambda, \vartheta}$ satisfies the strict boundary condition on $\Gamma_{\lambda, \vartheta}^{2B}$. Therefore, ${w}^{\lambda, \vartheta}$ satisfies \eqref{Yao0322}. Combining this with \eqref{Yao0614b}, and using \autoref{lma301}, one gets the negativity of ${w}^{\lambda, \vartheta}$ and \eqref{Yao0315a}. Similarly, one obtains \eqref{Yao0315b}. 
Applying \autoref{lma604A}, we conclude that \eqref{Yao0316} holds for $(\lambda, \vartheta) \in \mathfrak{S}^{\Lambda_{1}}_{2}$. 

\textbf{Step 2}. 
We claim that \eqref{Yao0316} holds for $(\lambda, \vartheta) \in \mathfrak{S}^{{l}_{\perp}}_{2}$. 

In fact, let $\Lambda_{1}$ be the constant in \eqref{Yao0615a} and define
\begin{equation*}
\Lambda_{k + 1} = \max \Big\{ \jmath(\Lambda_{k}), \, \frac{\Lambda_{k}}{1 + \sin\beta}, \, \frac{\imath(\Lambda_{k})}{1 + 2\sin(\beta/2)}, \, {l}_{\perp} \Big\}, \quad {k} \in \mathbb{N}. 
\end{equation*}
Let
\begin{equation} \label{Yao0616b}
\lambda \in [\Lambda_{2}, {l}_{*}] \quad \text{and} \quad \vartheta \in [\vartheta_{B}(\lambda), \bar{\omega}_{\lambda}]. 
\end{equation}
By the definition of $\bar{\omega}_{\lambda}$ and $\Lambda_{2}$, we have
\begin{equation*}
\Lambda_{1} \leq \lambda_{*}(\lambda, \vartheta) \leq {l}_{N}, \quad \imath(\Lambda_{1}) \leq \lambda^{*}(\lambda, \vartheta) \leq {l}_{N}, 
\end{equation*}
where $\lambda_{*}$ and $\lambda^{*}$ are given by
\begin{equation}
\begin{aligned}
\lambda_{*}(\lambda, \vartheta) & = \min\{\hat{\lambda}(\lambda, \vartheta), \; \check{\lambda}(\lambda, \vartheta)\}, \\
\lambda^{*}(\lambda, \vartheta) & = \max\{\hat{\lambda}(\lambda, \vartheta), \; \check{\lambda}(\lambda, \vartheta)\}. 
\end{aligned}
\end{equation}
From step 1, ${w}^{\lambda, \vartheta}$ satisfies the strict boundary condition on $\Gamma_{\lambda, \vartheta}^{2B}$. 
Recalling that $\bar{\omega}_{\lambda} \leq (\pi + \vartheta_{A})/2$ and $\Lambda_{2} \geq {l}_{\perp}$, we have
\begin{equation}
0 \leq 2\vartheta - \pi \leq \vartheta_{A}(\lambda). 
\end{equation}
Combining this with \autoref{lma401}, we know ${w}^{\lambda, \vartheta}$ satisfies the boundary condition on $\Gamma_{\lambda, \vartheta}^{2A}$. In a word, ${w}^{\lambda, \vartheta}$ satisfies \eqref{Yao0322}. Combining this with \eqref{Yao0614b}, we can apply \autoref{lma301} to conclude that
\begin{equation*}
{w}^{\lambda, \vartheta} < 0 \text{ in } {D}_{\lambda, \vartheta}
\end{equation*}
and \eqref{Yao0315a} holds for $\lambda$ and $\vartheta$ satisfying \eqref{Yao0616b}. Similarly, \eqref{Yao0315b} holds for $\lambda$ and $\vartheta$ satisfying \eqref{Yao0616b}. Applying \autoref{lma604A}, we conclude that \eqref{Yao0316} holds for $(\lambda, \vartheta) \in \mathfrak{S}^{\Lambda_{2}}_{2}$. 

Repeating this process again and again, one concludes that \eqref{Yao0316} holds for $(\lambda, \vartheta) \in \mathfrak{S}^{\lambda_{k}}_{2}$ for every ${k} \in \mathbb{N}^{ + }$. 
Due to the fact that $\Lambda_{1} < {l}_{*}$ and $\bar{\omega}_{\lambda} > (\pi + 3\beta)/4$, it follows that $\Lambda_{k + 1} \leq \Lambda_{k}$ with equality holds if and only if $\Lambda_{k} = {l}_{\perp}$. Moreover, $\Lambda_{k} = {l}_{\perp}$ for some ${k} \in \mathbb{N}^{ + }$. This leads to the conclusion that \eqref{Yao0316} holds for $(\lambda, \vartheta) \in \mathfrak{S}^{{l}_{\perp}}_{2}$. 
\end{proof}

\subsection{The case $\lambda < {l}_{\perp}$}

We observe that \eqref{Yao0315a} does not hold on the full line segment ${T}_{\lambda, \vartheta} \cap \Sigma$ when
\begin{equation} \label{Yao0621}
\vartheta_{A}(\lambda) < \vartheta < \vartheta_{B}(\lambda) \quad \text{and} \quad 0 < \lambda < \lambda_{C}. 
\end{equation}
In fact, since ${u}$ is always monotone near the Dirichlet boundary (see \autoref{lma206GNN}), we know that \eqref{Yao0315a} fails in a neighborhood of ${T}_{\lambda, \vartheta} \cap \Gamma_{D}$. On the other hand, we will prove strict monotonicity near the Neumann boundary (see \autoref{lma307}), and then, by continuity, we deduce that \eqref{Yao0315a} holds in a neighborhood of ${T}_{\lambda, \vartheta} \cap \Gamma_{N}^{ - }$. In summary, for $\lambda$ and $\vartheta$ satisfying \eqref{Yao0621}, the inequality \eqref{Yao0315a} fails near ${T}_{\lambda, \vartheta}\cap\Gamma_{ {D}}$ but holds near ${T}_{\lambda, \vartheta}\cap\Gamma_{N}^{ - }$. 

The following result illustrates that the monotonicity property \eqref{Yao0315a} holds on a subset of ${T}_{\lambda, \vartheta} \cap \Sigma$, which will be crucial in what follows. 

\begin{lma} \label{lma606}
Let $\Lambda \in (0, \lambda_{C})$. 
Suppose that \eqref{Yao0315a} holds for every $\vartheta \in {J}_{\lambda} \cap (0, \vartheta_{B}(\lambda)]$ and $\lambda \geq \Lambda$. 
Then
\begin{equation} \label{Yao0622}
{u}_{{x}_{1}}\sin(\vartheta - \tfrac{\beta}{2}) - {u}_{{x}_{2}}\cos(\vartheta - \tfrac{\beta}{2}) < 0 \text{ on } {T}_{\lambda, \vartheta} \cap \Sigma_{\Lambda, \vartheta_{A}(\Lambda)}
\end{equation}
holds for $\vartheta \in (\vartheta_{A}(\lambda), \vartheta_{B}(\lambda))$ and $\lambda \in (\Lambda, \lambda_{C})$, where
$\Sigma_{\Lambda, \vartheta_{A}(\Lambda)}$ is the "right cap" cut out by ${T}_{\Lambda, \vartheta_{A}(\Lambda)}$ from $\Sigma$, i.e., $\Sigma_{\Lambda, \vartheta_{A}(\Lambda)}$ is the domain enclosed by ${T}_{\Lambda, \vartheta_{A}(\Lambda)}$, $\Gamma_{N}^{ - }$, and $\Gamma_{D}$; see \autoref{Fig06Portion}. 
\end{lma}

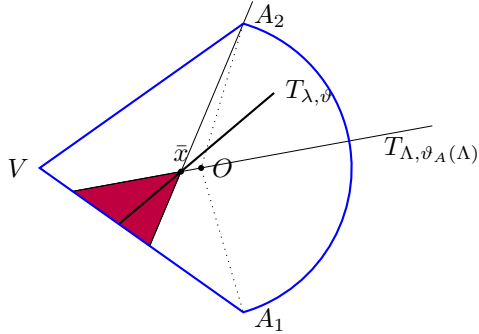
\begin{figure}[htp] \centering
\begin{tikzpicture}[scale = 2] 
\pgfmathsetmacro\ALPHA{73.73}; \pgfmathsetmacro\BETA{35.35}; \pgfmathsetmacro\radius{1.00}; 
\pgfmathsetmacro\aaa{\radius*sin(abs(\BETA - \ALPHA))/sin(\BETA)}; 
\pgfmathsetmacro\LenNeu{\radius*sin(\ALPHA)/sin(\BETA)}; 
\pgfmathsetmacro\xA{\aaa + \radius*cos(\ALPHA)}; \pgfmathsetmacro\yA{\radius*sin(\ALPHA)}; 
\pgfmathsetmacro\Jiaoa{(\ALPHA + \BETA)*0.94}; 
\pgfmathsetmacro\LAMa{\LenNeu*sin(\Jiaoa - 2*\BETA)/sin(\Jiaoa)}; 
\pgfmathsetmacro\LAMb{\LAMa*0.3}; 
\pgfmathsetmacro\Jiaob{atan(\aaa*sin(\BETA)/(\aaa*cos(\BETA) - \LAMb))}; 
\pgfmathsetmacro\Lenc{(\LAMa - \LAMb)*sin(\Jiaob)/sin(\Jiaoa - \Jiaob)}; 
\pgfmathsetmacro\xP{\LAMb*cos( - \BETA)}; \pgfmathsetmacro\yP{\LAMb*sin( - \BETA)}; 
\pgfmathsetmacro\xQ{\LAMa*cos( - \BETA)}; \pgfmathsetmacro\yQ{\LAMa*sin( - \BETA)}; 
\pgfmathsetmacro\xR{\xQ + \Lenc*cos(\Jiaoa - \BETA)}; \pgfmathsetmacro\yR{\yQ + \Lenc*sin(\Jiaoa - \BETA)}; 
\pgfmathsetmacro\xS{\xQ + 0.4*(\xP - \xQ)}; \pgfmathsetmacro\yS{\yQ + 0.4*(\yP - \yQ)}; 
\fill[fill = purple, draw = black, very thin] (\xP, \yP) -- (\xQ, \yQ) -- (\xR, \yR) -- cycle; 
\draw[very thin] (\xP, \yP) -- ({(\aaa - \xP)*2.8 + \xP}, {(0 - \yP)*2.8 + \yP}) node[below] {\footnotesize ${T}_{\Lambda, \vartheta_{A}(\Lambda)}$}; 
\draw[very thin] (\xQ, \yQ) -- ({(\xA - \xQ)*1.1 + \xQ}, {(\yA - \yQ)*1.1 + \yQ}); 
\draw[thick] (\xS, \yS) -- ({(\xR - \xS)*2.5 + \xS}, {(\yR - \yS)*2.5 + \yS}) node[right] {\footnotesize ${T}_{\lambda, \vartheta}$}; 
\fill[] (\xR, \yR) circle (0.02) node[above] { \footnotesize $\bar{x}$ }; 
\draw[blue, thick] (\xA, - \yA) arc ({ - \ALPHA} : {\ALPHA} : \radius) -- (0, 0) -- cycle; 
\draw[dotted] (\xA, \yA) -- (0 + \aaa, 0) -- (\xA, - \yA); 
\fill (0 + \aaa, 0) circle (0.02) node[right] {\footnotesize ${O}$ }; 
\node[left] at (0, 0) {\footnotesize ${V}$ }; 
\node[above = 3pt, right] at (\BETA : \LenNeu) {\footnotesize ${A}_{2}$ }; 
\node[below = 3pt, right] at ( - \BETA : \LenNeu) {\footnotesize ${A}_{1}$ }; 
\end{tikzpicture} 
\caption{Directional derivative on portion of ${T}_{\lambda, \vartheta}\cap\Sigma$ and the triangle domain {\color{purple}$\mathcal{T}$}. } 
\label{Fig06Portion}
\end{figure}

\begin{proof}
Let $\bar{{x}} = (\bar{{x}}_{1}, \bar{{x}}_{2})$ be the intersection point of ${T}_{\Lambda, \vartheta_{A}(\Lambda)}$ and ${T}_{\lambda, \vartheta}$; see \autoref{Fig06Portion}. 
Let $\mathcal{T}$ be the triangle formed by the line ${T}_{\Lambda, \vartheta_{A}(\Lambda)}$, $\Gamma_{N}^{ - }$, and the line passing through $(\bar{{x}}_{1}, \bar{{x}}_{2})$ and the upper mixed point. 
It is clear that the rotation function
\begin{equation*}
{v} = ({x}_{1} - \bar{{x}}_{1}){u}_{{x}_{2}} - ({x}_{2} - \bar{{x}}_{2}){u}_{{x}_{1}}
\end{equation*}
satisfies the linear equation
\begin{equation} \label{Yao0623a}
\Delta {v} + f'({u}){v} = 0. 
\end{equation}
By the assumption, ${v}$ satisfies the Dirichlet boundary condition
\begin{equation} \label{Yao0623b}
{v} \leq 0, \not\equiv 0 \text{ on } \partial \mathcal{T}. 
\end{equation}
Now, let us fix $\bar{\nu}$ as the unit vector orthogonal to ${T}_{\Lambda, \vartheta_{A}(\Lambda)}$, i.e., 
\begin{equation*}
\bar{\nu} = (\sin(\vartheta_{A}(\Lambda) - \tfrac{\beta}{2}), - \cos(\vartheta_{A}(\Lambda) - \tfrac{\beta}{2})). 
\end{equation*}
Note that $\vartheta_{A}(\Lambda) \leq \vartheta_{A}(\lambda)$ for $\lambda \geq \Lambda$. 
By the assumption that \eqref{Yao0315a} holds for every $\vartheta \in {J}_{\lambda} \cap (0, \vartheta_{B}(\lambda)]$ and $\lambda \geq \Lambda$, we deduce that
\begin{equation*}
\mathcal{L}[\nabla {u} \cdot \bar{\nu}] = [\Delta + f'({u})](\nabla {u} \cdot \bar{\nu}) = 0, \quad \nabla {u} \cdot \bar{\nu} < 0
\quad \text{in} \quad \Sigma_{\Lambda, \vartheta_{A}(\Lambda)}. 
\end{equation*}
From~\cite{BNV94}, the operator $\mathcal{L} = \Delta + f'({u})$ satisfies the maximum principle in the subdomain $\mathcal{T}$. Applying the maximum principle in~\cite{BNV94} to \eqref{Yao0623a}-\eqref{Yao0623b}, we deduce the negativity of ${v}$ in $\mathcal{T}$. In particular, \eqref{Yao0622} holds. This completes the proof of the lemma. 
\end{proof}

Let $\Sigma_{*}$ be the mirror image of $\Sigma$ with respect to the lower Neumann boundary $\Gamma_{N}^{ - }$, and let $\widetilde{\Sigma}$ be the interior of the closure of $\Sigma \cup \Sigma_{*}$. 
Similar to ${D}_{\lambda, \vartheta}$, we now set
\begin{equation*}
\widetilde{D}_{\lambda, \vartheta} = 
\Bigg\{
\begin{aligned}
& {x} \in \widetilde{\Sigma}: \; {x}^{\lambda, \vartheta} \in \widetilde{\Sigma} \quad \text{and }\quad
\\
& ({x}_{1} - {x}_{1\lambda})\sin(\vartheta - \tfrac{\beta}{2}) - ({x}_{2} - {x}_{2\lambda})\cos(\vartheta - \tfrac{\beta}{2}) > 0
\end{aligned}
\Bigg\}
\end{equation*}
where ${x}_{1\lambda}$ and ${x}_{2\lambda}$ are given in \eqref{Yao0303}; see \autoref{Fig07DoubleDomains}. 
Note that the reflections of ${T}_{0, \beta}$ and ${T}_{0, - \beta}$ with respect to ${T}_{\lambda, \vartheta}$ are ${T}_{{h}_{\lambda, \vartheta}^{1}, 2\vartheta - \beta}$ and ${T}_{{h}_{\lambda, \vartheta}^{3}, 2\vartheta + \beta - \pi}$, respectively. Here, 
\begin{equation*}
\begin{gathered}
{h}_{\lambda, \vartheta}^{0} = \hat{\lambda}(\lambda, \vartheta), \quad {h}_{\lambda, \vartheta}^{1} = \check{\lambda}(\lambda, \vartheta), \\
{h}_{\lambda, \vartheta}^{2} = \hat{\lambda}(\lambda, \pi - \vartheta), \quad
{h}_{\lambda, \vartheta}^{3} = \check{\lambda}(\lambda, \pi - \vartheta). 
\end{gathered}
\end{equation*}

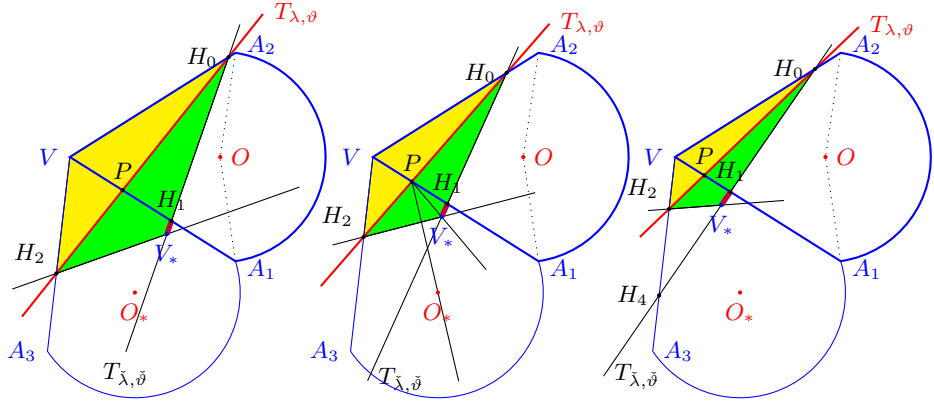
\begin{figure}[htp] \centering
\begin{tikzpicture}[scale = 1.4] 
\pgfmathsetmacro\ALPHA{163.46/2}; \pgfmathsetmacro\BETA{64.38/2}; \pgfmathsetmacro\radius{1.00}; 
\pgfmathsetmacro\aaa{\radius*sin(abs(\BETA - \ALPHA))/sin(\BETA)}; 
\pgfmathsetmacro\xA{\aaa + \radius*cos(\ALPHA)}; \pgfmathsetmacro\yA{\radius*sin(\ALPHA)}; 
\pgfmathsetmacro\LenNeu{\radius*sin(\ALPHA)/sin(\BETA)}; 
\pgfmathsetmacro\THETAB{\BETA*2 + (90 - \BETA*2)*0.72}; \pgfmathsetmacro\THETA{\THETAB + (90 - \THETAB)*0.12}; 
\pgfmathsetmacro\LAM{\LenNeu*sin(\THETAB - 2*\BETA)/sin(\THETAB)}; 
\pgfmathsetmacro\xP{\LAM*cos( - \BETA)}; \pgfmathsetmacro\yP{\LAM*sin( - \BETA)}; 
\pgfmathsetmacro\HHl{\LAM*sin(\THETA)/sin(\THETA - 2*\BETA)}; 
\pgfmathsetmacro\xHl{\HHl*cos(\BETA)}; \pgfmathsetmacro\yHl{\HHl*sin(\BETA)}; 
\pgfmathsetmacro\HHa{\LAM*(1 + sin(2*\BETA)/sin(2*\THETA - 2*\BETA))}; 
\pgfmathsetmacro\xHa{\HHa*cos( - \BETA)}; \pgfmathsetmacro\yHa{\HHa*sin( - \BETA)}; 
\pgfmathsetmacro\HHb{\LAM*sin(\THETA)/sin(\THETA + 2*\BETA)}; 
\pgfmathsetmacro\xHb{\HHb*cos( - 3*\BETA)}; \pgfmathsetmacro\yHb{\HHb*sin( - 3*\BETA)}; 
\pgfmathsetmacro\LENv{2*\LAM*sin(\THETA)}; 
\pgfmathsetmacro\xVv{\LENv*cos(\THETA - \BETA - 90)}; \pgfmathsetmacro\yVv{\LENv*sin(\THETA - \BETA - 90)}; 
\pgfmathsetmacro\xOo{\aaa*cos( - 2*\BETA)}; \pgfmathsetmacro\yOo{\aaa*sin( - 2*\BETA)}; 
\fill[fill = yellow, draw = black, very thin] 
(0, 0) -- (\BETA : \HHl) -- ({ - 3*\BETA} : \HHb) -- cycle; 
\fill[fill = green, draw = black, very thin] 
(\BETA : \HHl) -- ++ ({2*\THETA - 3*\BETA} : { - \HHl}) -- ({ - 3*\BETA} : \HHb) -- cycle; 
\draw[red, thick] ({1.2*(\xHb - \xHl) + \xHl}, {1.2*(\yHb - \yHl) + \yHl}) -- ({1.2*(\xHl - \xHb) + \xHb}, {1.2*(\yHl - \yHb) + \yHb}) node[right] {\scriptsize ${T}_{\lambda, \vartheta}$}; 
\draw[] ({1.2*(\xHl - \xHa) + \xHa}, {1.2*(\yHl - \yHa) + \yHa}) -- ({1.8*(\xHa - \xHl) + \xHl}, {1.8*(\yHa - \yHl) + \yHl}) node[below] {\scriptsize ${T}_{\check{\lambda}, \check{\vartheta}}$}; 
\draw[] ({1.4*(\xHb - \xVv) + \xVv}, {1.4*(\yHb - \yVv) + \yVv}) -- ({2.2*(\xVv - \xHb) + \xHb}, {2.2*(\yVv - \yHb) + \yHb}); 
\draw[blue, thick] (\xA, - \yA) node[below = 3pt, right] {\scriptsize ${A}_{1}$ } arc ({ - \ALPHA} : {\ALPHA} : \radius) node[above = 3pt, right] {\scriptsize ${A}_{2}$ } -- (0, 0) node[left] {\scriptsize ${V}$ } -- cycle; 
\draw[blue] (\xA, - \yA) arc ({\ALPHA - 2*\BETA} : { - \ALPHA - 2*\BETA} : \radius) node[left] {\scriptsize ${A}_{3}$ } -- (0, 0) ; 
\draw[dotted] (\xA, \yA) -- (\aaa, 0) -- (\xA, - \yA); 
\draw[very thick, purple, line width = 2pt] (\xVv, \yVv) -- (\xHa, \yHa); 
\fill (\xP, \yP) circle (0.02) node[above] { \scriptsize ${P}$ }; 
\fill (\xHl, \yHl) circle (0.02) node[left] { \scriptsize ${H}_{0}$ }; 
\fill (\xHa, \yHa) circle (0.02) node[above] { \scriptsize ${H}_{1}$ }; 
\fill (\xHb, \yHb) circle (0.02) node[above left] { \scriptsize ${H}_{2}$ }; 
\fill[red] (\xOo, \yOo) circle (0.02) node[below] { \scriptsize ${O}_{*}$ }; 
\fill[blue] (\xVv, \yVv) circle (0.02) node[below] { \scriptsize ${V}_{*}$ }; 
\fill[red] (\aaa, 0) circle (0.02) node[right] { \scriptsize ${O}$ }; 
\end{tikzpicture}
\hspace{ - 4ex}
\begin{tikzpicture}[scale = 1.4] 
\pgfmathsetmacro\ALPHA{163.46/2}; \pgfmathsetmacro\BETA{64.38/2}; \pgfmathsetmacro\radius{1.00}; 
\pgfmathsetmacro\aaa{\radius*sin(abs(\BETA - \ALPHA))/sin(\BETA)}; 
\pgfmathsetmacro\xA{\aaa + \radius*cos(\ALPHA)}; \pgfmathsetmacro\yA{\radius*sin(\ALPHA)}; 
\pgfmathsetmacro\LenNeu{\radius*sin(\ALPHA)/sin(\BETA)}; 
\pgfmathsetmacro\THETAB{\BETA*2 + (90 - \BETA*2)*0.52}; \pgfmathsetmacro\THETA{\THETAB + (90 - \THETAB)*0.28}; \pgfmathsetmacro\LAM{\LenNeu*sin(\THETAB - 2*\BETA)/sin(\THETAB)}; 
\pgfmathsetmacro\xP{\LAM*cos( - \BETA)}; \pgfmathsetmacro\yP{\LAM*sin( - \BETA)}; 
\pgfmathsetmacro\HHl{\LAM*sin(\THETA)/sin(\THETA - 2*\BETA)}; 
\pgfmathsetmacro\xHl{\HHl*cos(\BETA)}; \pgfmathsetmacro\yHl{\HHl*sin(\BETA)}; 
\pgfmathsetmacro\HHa{\LAM*(1 + sin(2*\BETA)/sin(2*\THETA - 2*\BETA))}; 
\pgfmathsetmacro\xHa{\HHa*cos( - \BETA)}; \pgfmathsetmacro\yHa{\HHa*sin( - \BETA)}; 
\pgfmathsetmacro\HHb{\LAM*sin(\THETA)/sin(\THETA + 2*\BETA)}; 
\pgfmathsetmacro\xHb{\HHb*cos( - 3*\BETA)}; \pgfmathsetmacro\yHb{\HHb*sin( - 3*\BETA)}; 
\pgfmathsetmacro\LENv{2*\LAM*sin(\THETA)}; 
\pgfmathsetmacro\xVv{\LENv*cos(\THETA - \BETA - 90)}; \pgfmathsetmacro\yVv{\LENv*sin(\THETA - \BETA - 90)}; 
\pgfmathsetmacro\xOo{\aaa*cos( - 2*\BETA)}; \pgfmathsetmacro\yOo{\aaa*sin( - 2*\BETA)}; 
\fill[fill = yellow, draw = black, very thin] 
(0, 0) -- (\BETA : \HHl) -- ({ - 3*\BETA} : \HHb) -- cycle; 
\fill[fill = green, draw = black, very thin] 
(\BETA : \HHl) -- ++ ({2*\THETA - 3*\BETA} : { - \HHl}) -- ({ - 3*\BETA} : \HHb) -- cycle; 
\draw[] (\xP, \yP) -- ({1.8*(\xOo - \xP) + \xP}, {1.8*(\yOo - \yP) + \yP}); 
\draw[] (\xP, \yP) -- ({2.5*(\xVv - \xP) + \xP}, {2.5*(\yVv - \yP) + \yP}); 
\draw[red, thick] ({1.3*(\xHb - \xHl) + \xHl}, {1.3*(\yHb - \yHl) + \yHl}) -- ({1.3*(\xHl - \xHb) + \xHb}, {1.3*(\yHl - \yHb) + \yHb}) node[right] {\scriptsize ${T}_{\lambda, \vartheta}$}; 
\draw[] ({1.2*(\xHl - \xHa) + \xHa}, {1.2*(\yHl - \yHa) + \yHa}) -- ({2.35*(\xHa - \xHl) + \xHl}, {2.35*(\yHa - \yHl) + \yHl}) node[right] {\scriptsize ${T}_{\check{\lambda}, \check{\vartheta}}$}; 
\draw[] ({1.4*(\xHb - \xVv) + \xVv}, {1.4*(\yHb - \yVv) + \yVv}) -- ({2.2*(\xVv - \xHb) + \xHb}, {2.2*(\yVv - \yHb) + \yHb}); 
\draw[blue, thick] (\xA, - \yA) node[below = 3pt, right] {\scriptsize ${A}_{1}$ } arc ({ - \ALPHA} : {\ALPHA} : \radius) node[above = 3pt, right] {\scriptsize ${A}_{2}$ } -- (0, 0) node[left] {\scriptsize ${V}$ } -- cycle; 
\draw[blue] (\xA, - \yA) arc ({\ALPHA - 2*\BETA} : { - \ALPHA - 2*\BETA} : \radius) node[left] {\scriptsize ${A}_{3}$ } -- (0, 0) ; 
\draw[dotted] (\xA, \yA) -- (\aaa, 0) -- (\xA, - \yA); 
\draw[very thick, purple, line width = 2pt] (\xVv, \yVv) -- (\xHa, \yHa); 
\fill (\xP, \yP) circle (0.02) node[above] { \scriptsize ${P}$ }; 
\fill (\xHl, \yHl) circle (0.02) node[left] { \scriptsize ${H}_{0}$ }; 
\fill (\xHa, \yHa) circle (0.02) node[above] { \scriptsize ${H}_{1}$ }; 
\fill (\xHb, \yHb) circle (0.02) node[above left] { \scriptsize ${H}_{2}$ }; 
\fill[red] (\xOo, \yOo) circle (0.02) node[below] { \scriptsize ${O}_{*}$ }; 
\fill[blue] (\xVv, \yVv) circle (0.02) node[below] { \scriptsize ${V}_{*}$ }; 
\fill[red] (\aaa, 0) circle (0.02) node[right] { \scriptsize ${O}$ }; 
\end{tikzpicture}
\hspace{ - 4ex}
\begin{tikzpicture}[scale = 1.4] 
\pgfmathsetmacro\ALPHA{163.46/2}; \pgfmathsetmacro\BETA{64.38/2}; \pgfmathsetmacro\radius{1.00}; 
\pgfmathsetmacro\aaa{\radius*sin(abs(\BETA - \ALPHA))/sin(\BETA)}; 
\pgfmathsetmacro\xA{\aaa + \radius*cos(\ALPHA)}; \pgfmathsetmacro\yA{\radius*sin(\ALPHA)}; 
\pgfmathsetmacro\LenNeu{\radius*sin(\ALPHA)/sin(\BETA)}; 
\pgfmathsetmacro\THETAB{\BETA*2 + (90 - \BETA*2)*0.38}; \pgfmathsetmacro\THETA{\THETAB + (90 - \THETAB)*0.12}; 
\pgfmathsetmacro\LAM{\LenNeu*sin(\THETAB - 2*\BETA)/sin(\THETAB)}; 
\pgfmathsetmacro\xP{\LAM*cos( - \BETA)}; \pgfmathsetmacro\yP{\LAM*sin( - \BETA)}; 
\pgfmathsetmacro\HHl{\LAM*sin(\THETA)/sin(\THETA - 2*\BETA)}; 
\pgfmathsetmacro\xHl{\HHl*cos(\BETA)}; \pgfmathsetmacro\yHl{\HHl*sin(\BETA)}; 
\pgfmathsetmacro\HHa{\LAM*(1 + sin(2*\BETA)/sin(2*\THETA - 2*\BETA))}; 
\pgfmathsetmacro\xHa{\HHa*cos( - \BETA)}; \pgfmathsetmacro\yHa{\HHa*sin( - \BETA)}; 
\pgfmathsetmacro\HHb{\LAM*sin(\THETA)/sin(\THETA + 2*\BETA)}; 
\pgfmathsetmacro\xHb{\HHb*cos( - 3*\BETA)}; \pgfmathsetmacro\yHb{\HHb*sin( - 3*\BETA)}; 
\pgfmathsetmacro\LENv{2*\LAM*sin(\THETA)}; 
\pgfmathsetmacro\xVv{\LENv*cos(\THETA - \BETA - 90)}; \pgfmathsetmacro\yVv{\LENv*sin(\THETA - \BETA - 90)}; 
\pgfmathsetmacro\xOo{\aaa*cos( - 2*\BETA)}; \pgfmathsetmacro\yOo{\aaa*sin( - 2*\BETA)}; 
\fill[fill = yellow, draw = black, very thin] 
(0, 0) -- (\BETA : \HHl) -- ({ - 3*\BETA} : \HHb) -- cycle; 
\fill[fill = green, draw = black, very thin] 
(\BETA : \HHl) -- ++ ({2*\THETA - 3*\BETA} : { - \HHl}) -- ({ - 3*\BETA} : \HHb) -- cycle; 
\draw[red, thick] ({1.2*(\xHb - \xHl) + \xHl}, {1.2*(\yHb - \yHl) + \yHl}) -- ({1.3*(\xHl - \xHb) + \xHb}, {1.3*(\yHl - \yHb) + \yHb}) node[right] {\scriptsize ${T}_{\lambda, \vartheta}$}; 
\draw[] ({1.2*(\xHl - \xHa) + \xHa}, {1.2*(\yHl - \yHa) + \yHa}) -- ({2.5*(\xHa - \xHl) + \xHl}, {2.5*(\yHa - \yHl) + \yHl}) node[right] {\scriptsize ${T}_{\check{\lambda}, \check{\vartheta}}$}; 
\draw[] ({1.4*(\xHb - \xVv) + \xVv}, {1.4*(\yHb - \yVv) + \yVv}) -- ({2.2*(\xVv - \xHb) + \xHb}, {2.2*(\yVv - \yHb) + \yHb}); 
\draw[blue, thick] (\xA, - \yA) node[below = 3pt, right] {\scriptsize ${A}_{1}$ } arc ({ - \ALPHA} : {\ALPHA} : \radius) node[above = 3pt, right] {\scriptsize ${A}_{2}$ } -- (0, 0) node[left] {\scriptsize ${V}$ } -- cycle; 
\draw[blue] (\xA, - \yA) arc ({\ALPHA - 2*\BETA} : { - \ALPHA - 2*\BETA} : \radius) node[right] {\scriptsize ${A}_{3}$ } -- (0, 0) ; 
\draw[dotted] (\xA, \yA) -- (\aaa, 0) -- (\xA, - \yA); 
\draw[very thick, purple, line width = 2pt] (\xVv, \yVv) -- (\xHa, \yHa); 
\pgfmathsetmacro\HHd{\LAM*sin(90 + 2*\BETA - \THETA)/sin(90 - \THETA)}; 
\pgfmathsetmacro\xHd{\HHd*cos( - 3*\BETA)}; \pgfmathsetmacro\yHd{\HHd*sin( - 3*\BETA)}; 
\fill (\xHd, \yHd) circle (0.02) node[left] { \scriptsize ${H}_{4}$ }; 
\fill (\xP, \yP) circle (0.02) node[above] { \scriptsize ${P}$ }; 
\fill (\xHl, \yHl) circle (0.02) node[left] { \scriptsize ${H}_{0}$ }; 
\fill (\xHa, \yHa) circle (0.02) node[above] { \scriptsize ${H}_{1}$ }; 
\fill (\xHb, \yHb) circle (0.02) node[above left] { \scriptsize ${H}_{2}$ }; 
\fill[red] (\xOo, \yOo) circle (0.02) node[below] { \scriptsize ${O}_{*}$ }; 
\fill[blue] (\xVv, \yVv) circle (0.02) node[below] { \scriptsize ${V}_{*}$ }; 
\fill[red] (\aaa, 0) circle (0.02) node[right] { \scriptsize ${O}$ }; 
\end{tikzpicture} 
\caption{The moving domains $\tilde{D}_{\lambda, \vartheta}$ for $\lambda \leq {l}_{\perp}$. }
\label{Fig07DoubleDomains}
\end{figure}

Under the condition $\beta \in (\pi/3, \pi/2)$, one has
\begin{equation*}
{l}_{\perp} < {l}_{*} < \frac{{l}_{N}}{2}. 
\end{equation*}
Combining this with \eqref{Yao0505}, we have
\begin{equation*}
{h}_{\lambda, \vartheta}^{i} \leq {l}_{N}, \quad i = 0, 1, 2, 3, 
\end{equation*}
provided $\lambda$ and $\vartheta$ satisfy
\begin{equation*}
\vartheta_{B}(\lambda) \leq \vartheta \leq \vartheta_{\lambda}, \quad \vartheta_{B}(\lambda) \leq \pi - \vartheta \leq \vartheta_{\lambda}, \quad 0 < \lambda \leq {l}_{\perp}, 
\end{equation*}
or equivalently (by using $\vartheta_{\lambda} + \vartheta_{B}(\lambda) > \pi$ from \autoref{lma601}\ref{Yaoit62b}), 
\begin{equation*}
\vartheta_{B}(\lambda) \leq \vartheta \leq \pi - \vartheta_{B}(\lambda), \quad 0 < \lambda \leq {l}_{\perp}. 
\end{equation*}
Note that $\vartheta_{B}(\lambda) \leq \pi/2$ for $\lambda \leq {l}_{\perp}$, and $\vartheta \leq \pi/2$ implies that ${h}_{\lambda, \vartheta}^{1} \leq {h}_{\lambda, \vartheta}^{3}$. Therefore, we will focus on
\begin{equation} \label{Yao0625c}
\vartheta_{B}(\lambda) \leq \vartheta \leq \frac{\pi}{2}, \quad 0 < \lambda \leq {l}_{\perp}. 
\end{equation}

Under \eqref{Yao0625c}, the moving domain $\widetilde{D}_{\lambda, \vartheta}$ is an open triangle enclosed by ${T}_{\lambda, \vartheta}$, ${T}_{{h}_{\lambda, \vartheta}^{1}, 2\vartheta - \beta}$, and ${T}_{{h}_{\lambda, \vartheta}^{3}, 2\vartheta + \beta - \pi}$. The boundary of $\widetilde{D}_{\lambda, \vartheta}$ consists of three parts (see \autoref{Fig07DoubleDomains}): 
\begin{enumerate}[label = {\rm(\arabic*)}]
\item
$\widetilde{\Gamma}_{\lambda, \vartheta}^{0} = \partial \widetilde{D}_{\lambda, \vartheta} \cap {T}_{\lambda, \vartheta}$; 
\item
$\widetilde{\Gamma}_{\lambda, \vartheta}^{1} = \partial \widetilde{D}_{\lambda, \vartheta} \cap {T}_{{h}_{\lambda, \vartheta}^{3}, 2\vartheta + \beta - \pi}$ (this part is contained in $\Sigma_{*}$); 
\item
$\widetilde{\Gamma}_{\lambda, \vartheta}^{2} = \partial \widetilde{D}_{\lambda, \vartheta} \cap {T}_{{h}_{\lambda, \vartheta}^{1}, 2\vartheta - \beta}$, and this boundary has two subparts: $\widetilde{\Gamma}_{\lambda, \vartheta}^{2A} = \widetilde{\Gamma}_{\lambda, \vartheta}^{2} \cap \Sigma$ and $\widetilde{\Gamma}_{\lambda, \vartheta}^{2B} = \widetilde{\Gamma}_{\lambda, \vartheta}^{2} \setminus \Sigma$. 
\end{enumerate}

The even extension of ${u}$ along $\Gamma_{N}^{ - }$ is still denoted by ${u}$, and the difference function ${w}^{\lambda, \vartheta}$ is still defined as in \eqref{Yao0305b}. 
We will show that
\begin{equation} \label{Yao0626a}
{w}^{\lambda, \vartheta} < 0 \text{ in } \widetilde{D}_{\lambda, \vartheta}
\end{equation}
and
\begin{equation} \label{Yao0626b}
{u}_{{x}_{1}}\sin(\vartheta - \tfrac{\beta}{2}) - {u}_{{x}_{2}}\cos(\vartheta - \tfrac{\beta}{2}) < 0 \text{ on } {T}_{\lambda, \vartheta} \cap \widetilde{\Sigma}
\end{equation}
for $\lambda$ and $\vartheta$ satisfying \eqref{Yao0625c}. 

Now we are ready to establish the monotonicity in \eqref{Yao0316} for $\vartheta_{B}(\lambda) \leq \vartheta \leq \bar{\omega}_{\lambda}$, $\lambda > 0$, and hence the symmetry result for $\pi/3 < \beta < \pi/2$. 

\begin{thm} 
Let $0 < \beta < \alpha \leq \pi$ and
\begin{equation} \label{Yao0627}
\beta \in (\pi/3, \pi/2). 
\end{equation}
Then the solution ${u}$ of \eqref{Yao0105} satisfies
\eqref{Yao0316} for $(\lambda, \vartheta) \in \mathfrak{S}^{\Lambda}$ for every $\Lambda > 0$. 
Consequently, ${u}$ possesses the even symmetry and monotonicity properties as stated in \ref{Yaoit11a}, \ref{Yaoit11b}, and \ref{Yaoit11c} of \autoref{thm101main}. 
\end{thm}

\begin{proof}
From \autoref{lma401}, we have
\begin{equation*}
{w}^{\lambda, \vartheta_{B}(\lambda)} < 0 \text{ in } {D}_{\lambda, \vartheta_{B}(\lambda), 2\vartheta_{B}(\lambda) - \pi} \text{ for } \lambda \geq {l}_{\perp}, 
\end{equation*}
and in particular, 
\begin{equation} \label{Yao0614c}
{w}^{\lambda, \vartheta_{B}(\lambda)} < 0 \text{ in } \widetilde{D}_{\lambda, \vartheta_{B}(\lambda)} \text{ for } \lambda = {l}_{\perp}. 
\end{equation}
From \autoref{lma604B}, \eqref{Yao0316} holds for $(\lambda, \vartheta) \in \mathfrak{S}^{{l}_{\perp}}$. 
Hence, we define
\begin{equation} \label{Yao0628}
\bar{\Lambda} = \inf\left\{\Lambda \in (0, {l}_{*}]: \; \eqref{Yao0316} \text{ holds for } (\lambda, \vartheta) \in \mathfrak{S}^{\Lambda} \right\}. 
\end{equation}
To prove the assertion (i.e., $\bar{\Lambda} = 0$), we argue by contradiction, assuming $\bar{\Lambda} > 0$ and thus $\bar{\Lambda} \in (0, {l}_{\perp}]$. 

\textbf{Step 1}. 
We claim that \eqref{Yao0316a} holds when $\vartheta \in [\vartheta_{B}(\lambda), \pi - \vartheta_{B}(\lambda)]$ and $\lambda$ is close to $\bar{\Lambda}$. 

First, we note that
\begin{equation} \label{Yao0629}
{h}_{\lambda, \vartheta}^{i} > \lambda, \quad
{h}_{\lambda, \vartheta}^{0} > \imath(\lambda), \quad
\max\{ {h}_{\lambda, \vartheta}^{2}, {h}_{\lambda, \vartheta}^{3} \} > \imath(\lambda)
\end{equation}
for $\vartheta \in [\vartheta_{B}(\lambda), \pi/2]$ and $\lambda \in (0, {l}_{\perp}]$. Indeed, it is clear that ${h}_{\lambda, \vartheta}^{i} > \lambda$ for $i = 0, 1, 3$, while ${h}_{\lambda, \vartheta}^{2} > \lambda$ since $\vartheta > \pi - \beta - \vartheta$ and $\beta > \pi/3$. 
Recalling that
\begin{equation*}
\bar{\omega}_{\lambda} > \frac{\pi + 3\beta}{4} > \frac{\pi}{2} \geq \vartheta_{B}(\lambda) \quad \text{for } \lambda \in (0, {l}_{\perp}], 
\end{equation*}
it follows by the definitions that
\begin{gather*}
{h}_{\lambda, \vartheta}^{0} = \hat{\lambda}(\lambda, \vartheta) > \hat{\lambda}(\lambda, \tfrac{\pi}{2}) > \imath(\lambda), 
\\
\max\{{h}_{\lambda, \vartheta}^{2}, {h}_{\lambda, \vartheta}^{3}\} \geq \hat{\lambda}(\lambda, \tfrac{\pi + 3\beta}{4}) > \imath(\lambda). 
\end{gather*}

From \eqref{Yao0629}, there exists a positive constant $\bar{\Lambda}_{1}$ (close to and less than $\bar{\Lambda}$) such that
\begin{gather}
\label{Yao0631a}
\bar{\Lambda} < {h}_{\lambda, \vartheta}^{1} < {l}_{N}, \quad \imath(\bar{\Lambda}) < {h}_{\lambda, \vartheta}^{0} \leq {l}_{N}, 
\\
\label{Yao0631b}
\min\{ {h}_{\lambda, \vartheta}^{2}, {h}_{\lambda, \vartheta}^{3}\} > \bar{\Lambda}, \quad \imath(\bar{\Lambda}) < \max\{{h}_{\lambda, \vartheta}^{2}, {h}_{\lambda, \vartheta}^{3}\} \leq {l}_{N}
 \end{gather}
whenever $(\lambda, \vartheta) \in \Xi^{\bar{\Lambda}_{1}}$, where
\begin{equation*}
\Xi^{\Lambda} = \{(\lambda, \vartheta): \; \vartheta \in [\vartheta_{B}(\lambda), \tfrac{\pi}{2}], \; \lambda \in [\Lambda, {l}_{\perp}] \}. 
\end{equation*}
From \eqref{Yao0631b} and the definition of $\bar{\Lambda}$ in \eqref{Yao0628}, 
\begin{equation} \label{Yao0634a}
\nabla {w}^{\lambda, \vartheta} \cdot \mathbf{e}_{2\vartheta + \beta/2 + \pi/2} = \nabla {u} \cdot \mathbf{e}_{2\vartheta + \beta/2 + \pi/2} < 0 \text{ on } \widetilde{\Gamma}_{\lambda, \vartheta}^{1}
\end{equation}
for $(\lambda, \vartheta) \in \Xi^{\bar{\Lambda}_{1}}$, where $\mathbf{e}_{2\vartheta + \beta/2 + \pi/2}$ denotes the outer normal direction to $\widetilde{\Gamma}_{\lambda, \vartheta}^{1}$. 
Similarly, from \eqref{Yao0631a} and the definition of $\bar{\Lambda}$, 
\begin{equation} \label{Yao0634b}
\nabla {w}^{\lambda, \vartheta} \cdot \mathbf{e}_{2\vartheta - 3\beta/2 - \pi/2} = \nabla {u} \cdot \mathbf{e}_{2\vartheta - 3\beta/2 - \pi/2} < 0 \text{ on } \widetilde{\Gamma}_{\lambda, \vartheta}^{2A}
\end{equation}
for $(\lambda, \vartheta) \in \Xi^{\bar{\Lambda}_{1}}$, where $\mathbf{e}_{2\vartheta - 3\beta/2 - \pi/2}$ is the outer normal to $\widetilde{\Gamma}_{\lambda, \vartheta}^{2}$. 

Now we turn to the boundary condition for ${w}^{\lambda, \vartheta}$ on $\widetilde{\Gamma}_{\lambda, \vartheta}^{2B}$, which is more delicate. 
Let ${A}_{1}$ and ${A}_{2}$ be the lower and upper mixed boundary points of $\Sigma$, respectively. The points ${P}$ and ${H}_{0}$ denote the intersections of ${T}_{\lambda, \vartheta}$ and $\Gamma_{N}^{\mp}$. Moreover, ${V}_{*}$ is the reflection of the vertex ${V}$ with respect to ${T}_{\lambda, \vartheta}$, and ${O}_{*}$ is the reflection of the center ${O}$ with respect to $\Gamma_{N}^{ - }$. 
Let ${H}_{1}$ be the intersection point between $\widetilde{\Gamma}_{\lambda, \vartheta}^{2}$ and $\Gamma_{N}^{ - }$. It is clear that $\angle {V}_{*}{H}_{1}{A}_{1} = \pi + \beta - 2\vartheta$; see \autoref{Fig07DoubleDomains}. 
From \eqref{Yao0627}, 
\begin{equation*}
\tfrac{\pi + \beta}{2} - \left( \pi + \beta - 2\vartheta \right) = 2 (\vartheta - \beta) + \tfrac{3\beta - \pi}{2} > 0. 
\end{equation*}
Note that $\widetilde{\Gamma}_{\lambda, \vartheta}^{2B}$ is the (compact) line segment ${H}_{1}{V}_{*}$, which lies in ${T}_{{h}_{\lambda, \vartheta}^{1}, 2\vartheta - \beta} \cap (\Sigma_{*} \cup \Gamma_{N}^{ - })$ with $0 < \pi + \beta - 2\vartheta < (\pi + \beta)/2$. 
Now let
\begin{equation*}
\vartheta \in [\vartheta_{B}(\lambda), \tfrac{\pi}{2}], \quad \lambda \in [\bar{\Lambda}, {l}_{\perp}]. 
\end{equation*}
There are three cases: 

\textit{Case 1}: $\angle {V}_{*}{H}_{1}{A}_{1} \leq \vartheta_{A}({h}_{\lambda, \vartheta}^{1})$, see the left panel of \autoref{Fig07DoubleDomains}. 
Note that \eqref{Yao0631a} implies ${h}_{\lambda, \vartheta}^{1} \in (\bar{\Lambda}, {l}_{N})$. By \autoref{lma604B}, we have
\begin{equation} \label{Yao0635}
{u}_{{x}_{1}} \sin(2\vartheta - \tfrac{3\beta}{2}) - {u}_{{x}_{2}} \cos(2\vartheta - \tfrac{3\beta}{2}) < 0 \text{ on } \widetilde{\Gamma}_{\lambda, \vartheta}^{2B}, 
\end{equation}
where the left-hand side gives the outer normal derivative of ${u}$ along $\widetilde{\Gamma}_{\lambda, \vartheta}^{2}$. 

\textit{Case 2}: $\vartheta_{A}({h}_{\lambda, \vartheta}^{1}) < \angle {V}_{*}{H}_{1}{A}_{1} < \vartheta_{B}({h}_{\lambda, \vartheta}^{1})$, see the middle panel of \autoref{Fig07DoubleDomains}. 
In this case, ${h}_{\lambda, \vartheta}^{4} > {l}_{N}$, ${h}_{\lambda, \vartheta}^{1} \leq {l}_{N}$. Thus, $\pi + \beta - 2\vartheta < \pi/2$, i.e., $\angle {H}_{0}{P}{A}_{1} = \vartheta > (\pi + 2\beta)/4$. 
So, 
\begin{equation*}
\angle {A}_{1}{P}{V}_{*} = \pi - 2\angle {Q}{P}{A}_{1} \leq \pi - \frac{\pi + 2\beta}{2} = \frac{\pi}{2} - \beta < \frac{\beta}{2}
\end{equation*}
(using \eqref{Yao0627}) and then
\begin{equation*}
\angle {A}_{1}{P}{V}_{*} < \angle {A}_{1}{P}{O}_{*}. 
\end{equation*}
Thus, ${V}_{*}$ and $\widetilde{\Gamma}_{\lambda, \vartheta}^{2B}$ are contained in the right cap cut by the line $PO_{*}$ from $\Sigma_{*}$. By \autoref{lma606}, \eqref{Yao0635} holds. 

\textit{Case 3}: $\angle {V}_{*}{H}_{1}{A}_{1} \geq \vartheta_{B}({h}_{\lambda, \vartheta}^{1})$, see the right panel of \autoref{Fig07DoubleDomains}. 
Here, ${h}_{\lambda, \vartheta}^{4} \leq {l}_{N}$, where ${h}_{\lambda, \vartheta}^{4} = \hat{\lambda}({h}_{\lambda, \vartheta}^{1}, \pi + \beta - 2\vartheta)$. Let ${H}_{4}$ be the intersection of ${T}_{{h}_{\lambda, \vartheta}^{1}, 2\vartheta - \beta}$ and ${T}_{0, - \beta}$, as shown in the right panel of \autoref{Fig07DoubleDomains}. Then ${h}_{\lambda, \vartheta}^{4}$ equals the length of the segment ${V}{H}_{4}$. 
Since ${P}$ is the incenter of triangle ${V}{H}_{0}{H}_{4}$, 
\begin{equation*}
\angle {H}_{4}{P}{A}_{1} = \left(\frac{\pi}{2} + \angle {H}_{0}{V}{H}_{4}\right) - \angle {H}_{0}{P}{A}_{1} = \frac{\pi}{2} + \beta - \vartheta, 
\end{equation*}
and
\begin{equation}
{h}_{\lambda, \vartheta}^{4} = \hat{\lambda}(\lambda, \frac{\pi}{2} + \beta - \vartheta) \quad \text{for } \vartheta \in (0, \frac{\pi}{2}). 
\end{equation}
Hence, using $\vartheta > \vartheta_{B}(\lambda)$, $\bar{\omega}_{\lambda}$, and $\lambda \in [\bar{\Lambda}, {l}_{\perp}]$, we obtain
\begin{equation*}
{h}_{\lambda, \vartheta}^{4} = \hat{\lambda}(\lambda, \frac{\pi}{2} + \beta - \vartheta) > \hat{\lambda}(\lambda, \frac{\pi}{2}) > \imath(\lambda) \geq \imath(\bar{\Lambda}). 
\end{equation*}
It follows that
\begin{equation*}
{h}_{\lambda, \vartheta}^{4} \in (\imath(\bar{\Lambda}), {l}_{N}], \quad {h}_{\lambda, \vartheta}^{1} \in (\bar{\Lambda}, {l}_{N}). 
\end{equation*}
Thus, by the definition of $\bar{\Lambda}$ in \eqref{Yao0628}, \eqref{Yao0635} holds. 

In all three cases, \eqref{Yao0635} is valid for $\vartheta \in [\vartheta_{B}(\bar{\Lambda}), \pi/2]$, $\lambda \in [\bar{\Lambda}, {l}_{\perp}]$. By continuity, there exists $\bar{\Lambda}_{2} \in (0, \bar{\Lambda})$ (with $\bar{\Lambda}_{2} > \bar{\Lambda}_{1}$) such that
\begin{equation} \label{Yao0634c}
{u}_{{x}_{1}} \sin(2\vartheta - \tfrac{3\beta}{2}) - {u}_{{x}_{2}} \cos(2\vartheta - \tfrac{3\beta}{2}) < 0 \text{ on } \widetilde{\Gamma}_{\lambda, \vartheta}^{2B}
\end{equation}
holds for $(\lambda, \vartheta) \in \Xi^{\bar{\Lambda}_{2}}$. 

Combining \eqref{Yao0634a}, \eqref{Yao0634b}, and \eqref{Yao0634c}, we know ${w}^{\lambda, \vartheta}$ satisfies
\begin{equation*}
\begin{cases}
\Delta {w}^{\lambda, \vartheta} + {c}^{\lambda}({x}) {w}^{\lambda, \vartheta} = 0 & \text{in } \widetilde{D}_{\lambda, \vartheta}, \\
{w}^{\lambda, \vartheta} = 0 & \text{on } \widetilde{\Gamma}_{\lambda, \vartheta}^{0}, \\
\nabla {w}^{\lambda, \vartheta} \cdot \nu < 0 & \text{on } \widetilde{\Gamma}_{\lambda, \vartheta}^{1} \cup \widetilde{\Gamma}_{\lambda, \vartheta}^{2}
\end{cases}
\end{equation*}
for $(\lambda, \vartheta) \in \Xi^{\bar{\Lambda}_{2}}$. 
Thanks to \eqref{Yao0614c}, one can use the same argument as in \autoref{lma301} to show that \eqref{Yao0626a} and \eqref{Yao0626b} hold for $(\lambda, \vartheta) \in \Xi^{\bar{\Lambda}_{2}}$. Noting that \eqref{Yao0626b} holds for points in $\Sigma$ and its reflection $\Sigma_{*}$, it follows that \eqref{Yao0316a} holds for $\vartheta \in [\vartheta_{B}(\lambda), \pi - \vartheta_{B}(\lambda)]$ and $\lambda \in [\bar{\Lambda}_{2}, \bar{\Lambda}]$. 

\textbf{Step 2}. 
\eqref{Yao0316a} holds for $(\lambda, \vartheta) \in \mathfrak{S}_{1}^{\bar{\Lambda}_{2}}$. 
This step combines step 1 and \autoref{lma308}. 

\textbf{Step 3}. 
\eqref{Yao0316a} holds for $\lambda \in [\bar{\Lambda}_{3}, {l}_{\perp}]$ and $\vartheta \in [\pi/2, \bar{\omega}_{\lambda}]$ where
\begin{equation*}
\bar{\Lambda}_{3} = \max\Big\{ \jmath(\bar{\Lambda}), \frac{\bar{\Lambda}}{1 + \sin\beta}, \frac{\imath(\bar{\Lambda})}{1 + 2\sin(\beta/2)}, \bar{\Lambda}_{2} \Big\}. 
\end{equation*}
This step follows by the same arguments as in step 2 of \autoref{lma604B}. 

\textbf{Step 4}. 
We conclude the proof. 
From steps 1 -- 3, we obtain \eqref{Yao0316a} for $\lambda \in [\bar{\Lambda}_{3}, {l}_{\perp}]$ and $\vartheta \in {J}_{\lambda} \cap (0, \bar{\omega}_{\lambda}]$. 
The same process yields \eqref{Yao0316b} for $\lambda \in [\bar{\Lambda}_{4}, {l}_{\perp}]$ and $\vartheta \in {J}_{\lambda}\cap (0, \bar{\omega}_{\lambda}]$ for some $\bar{\Lambda}_{4} \in (0, \bar{\Lambda})$. Thanks to \autoref{lma604A}, \eqref{Yao0316} holds for $(\lambda, \vartheta) \in \mathfrak{S}^{\Lambda}$ with $\Lambda = \max\{\bar{\Lambda}_{3}, \bar{\Lambda}_{4}\} \in (0, \bar{\Lambda})$. 
This contradicts the minimality of $\bar{\Lambda}$, and therefore $\bar{\Lambda} = 0$. Hence, \eqref{Yao0316} holds for every $\vartheta \in {J}_{\lambda}$ and $\lambda > 0$. 

As a direct consequence, we have ${u}_{{x}_{1}} < 0$ in $\Sigma$, and ${x}_{2}{u}_{{x}_{2}} < 0$ in $\Sigma \cap \{{x}_{2} \neq 0\}$. The symmetry property of ${u}$ follows from part 2 of \autoref{thm403}. 
\end{proof}



\section*{Acknowledgments}
The authors declare no conflicts of interest.
This work was supported by Guangdong Basic and Applied Basic Research Foundation (Grant No. 2025A1515011856) and the National Natural Science Foundation of China (Grant No. 12001543).




\begin{thebibliography}{10}

\bibitem{AF03}
Robert~A. Adams and John J.~F. Fournier.
\newblock {\em Sobolev Spaces}, volume 140 of {\em Pure and Applied Mathematics
(Amsterdam)}.
\newblock Elsevier/Academic Press, Amsterdam, second edition, 2003.

\bibitem{AR25}
Nausica Aldeghi and Jonathan Rohleder.
\newblock On the first eigenvalue and eigenfunction of the {L}aplacian with
mixed boundary conditions.
\newblock {\em J. Differential Equations}, 427:689--718, 2025.

\bibitem{Ale56}
A.~D. Aleksandrov.
\newblock Uniqueness theorems for surfaces in the large. {I}.
\newblock {\em Vestnik Leningrad. Univ.}, 11(19):5--17, 1956.

\bibitem{AB04}
Rami Atar and Krzysztof Burdzy.
\newblock On {N}eumann eigenfunctions in lip domains.
\newblock {\em J. Amer. Math. Soc.}, 17(2):243--265, 2004.

\bibitem{BCN97a}
Henri Berestycki, Luis Caffarelli, and Louis Nirenberg.
\newblock Further qualitative properties for elliptic equations in unbounded
domains.
\newblock {\em Ann. Scuola Norm. Sup. Pisa Cl. Sci. (4)}, 25(1-2):69--94, 1997.
\newblock Dedicated to Ennio De Giorgi.

\bibitem{BCN97b}
Henri Berestycki, Luis Caffarelli, and Louis Nirenberg.
\newblock Monotonicity for elliptic equations in unbounded {L}ipschitz domains.
\newblock {\em Comm. Pure Appl. Math.}, 50(11):1089--1111, 1997.

\bibitem{BN91}
Henri Berestycki and Louis Nirenberg.
\newblock On the method of moving planes and the sliding method.
\newblock {\em Bol. Soc. Brasil. Mat. (N.S.)}, 22(1):1--37, 1991.

\bibitem{BNV94}
Henri Berestycki, Louis Nirenberg, and Srinivasa~R.S. Varadhan.
\newblock The principal eigenvalue and maximum principle for second-order
elliptic operators in general domains.
\newblock {\em Comm. Pure Appl. Math.}, 47(1):47--92, 1994.

\bibitem{BP89}
Henri Berestycki and Filomena Pacella.
\newblock Symmetry properties for positive solutions of elliptic equations with
mixed boundary conditions.
\newblock {\em J. Funct. Anal.}, 87(1):177--211, 1989.

\bibitem{BR15}
Henri Berestycki and Luca Rossi.
\newblock Generalizations and properties of the principal eigenvalue of
elliptic operators in unbounded domains.
\newblock {\em Comm. Pure Appl. Math.}, 68(6):1014--1065, 2015.

\bibitem{CGS89}
Luis~A. Caffarelli, Basilis Gidas, and Joel Spruck.
\newblock Asymptotic symmetry and local behavior of semilinear elliptic
equations with critical {S}obolev growth.
\newblock {\em Comm. Pure Appl. Math.}, 42(3):271--297, 1989.

\bibitem{CGY26}
Hongbin Chen, Changfeng Gui, and Ruofei Yao.
\newblock Uniqueness of the critical point of the second {N}eumann
eigenfunction in triangle.
\newblock {\em Invent. Math.}, 2026.
\newblock Online.

\bibitem{CLY21}
Hongbin Chen, Rui Li, and Ruofei Yao.
\newblock Symmetry of positive solutions of elliptic equations with mixed
boundary conditions in a sub-spherical sector.
\newblock {\em Nonlinearity}, 34(6):3858--3878, 2021.

\bibitem{CWY23}
Hongbin Chen, Ke~Wu, and Ruofei Yao.
\newblock Qualitative properties of nonnegative solutions of some semilinear
elliptic equations in cylindrical domains.
\newblock {\em Calc. Var. Partial Differential Equations}, 62(6):Paper No. 178,
23, 2023.

\bibitem{CWY26}
Hongbin Chen, Ke~Wu, and Ruofei Yao.
\newblock Monotone properties of the second even {N}eumann eigenfunction on
symmetric domains.
\newblock {\em Ann. Mat. Pura Appl. (4)}, 2025.
\newblock Online.

\bibitem{CY18}
Hongbin Chen and Ruofei Yao.
\newblock Symmetry and monotonicity of positive solution of elliptic equation
with mixed boundary condition in a spherical cone.
\newblock {\em J. Math. Anal. Appl.}, 461(1):641--656, 2018.

\bibitem{CL91}
Wenxiong Chen and Congming Li.
\newblock Classification of solutions of some nonlinear elliptic equations.
\newblock {\em Duke Math. J.}, 63(3):615--622, 1991.

\bibitem{CL18}
Wenxiong Chen and Congming Li.
\newblock Maximum principles for the fractional {$p$}-{L}aplacian and symmetry
of solutions.
\newblock {\em Adv. Math.}, 335:735--758, 2018.

\bibitem{CLL17}
Wenxiong Chen, Congming Li, and Yan Li.
\newblock A direct method of moving planes for the fractional {L}aplacian.
\newblock {\em Adv. Math.}, 308:404--437, 2017.

\bibitem{CW92}
Chie-Ping Chu and Hwai-Chiuan Wang.
\newblock Symmetry properties of positive solutions of elliptic equations in an
infinite sectorial cone.
\newblock {\em Proc. Roy. Soc. Edinburgh Sect. A}, 122(1-2):137--160, 1992.

\bibitem{DQ23}
Wei Dai and Guolin Qin.
\newblock Liouville-type theorems for fractional and higher-order
{H}\'{e}non-{H}ardy type equations via the method of scaling spheres.
\newblock {\em Int. Math. Res. Not. IMRN}, 2023(11):9001--9070, 2023.

\bibitem{DP19}
Lucio Damascelli and Filomena Pacella.
\newblock Morse index and symmetry for elliptic problems with nonlinear mixed
boundary conditions.
\newblock {\em Proc. Roy. Soc. Edinburgh Sect. A}, 149(2):305--324, 2019.

\bibitem{DP20}
Lucio Damascelli and Filomena Pacella.
\newblock Sectional symmetry of solutions of elliptic systems in cylindrical
domains.
\newblock {\em Discrete Contin. Dyn. Syst.}, 40(6):3305--3325, 2020.

\bibitem{DR04}
Juan D\'{a}vila and Julio~D. Rossi.
\newblock Self-similar solutions of the porous medium equation in a half-space
with a nonlinear boundary condition: existence and symmetry.
\newblock {\em J. Math. Anal. Appl.}, 296(2):634--649, 2004.

\bibitem{dCP89}
Matteo De~Cesare and Filomena Pacella.
\newblock Geometrical properties of positive solutions of semilinear elliptic
equations in some unbounded domains.
\newblock {\em Boll. Un. Mat. Ital. B (7)}, 3(3):691--703, 1989.

\bibitem{dPFW00}
Manuel del Pino, Patricio~L. Felmer, and Juncheng Wei.
\newblock Multi-peak solutions for some singular perturbation problems.
\newblock {\em Calc. Var. Partial Differential Equations}, 10(2):119--134,
2000.

\bibitem{DSV17}
Serena Dipierro, Nicola Soave, and Enrico Valdinoci.
\newblock On fractional elliptic equations in {L}ipschitz sets and epigraphs:
regularity, monotonicity and rigidity results.
\newblock {\em Math. Ann.}, 369(3-4):1283--1326, 2017.

\bibitem{EFMS22}
Francesco Esposito, Alberto Farina, Luigi Montoro, and Berardino Sciunzi.
\newblock On the {G}ibbons' conjecture for equations involving the
{$p$}-{L}aplacian.
\newblock {\em Math. Ann.}, 382(1-2):943--974, 2022.

\bibitem{FMR16}
Alberto Farina, Andrea Malchiodi, and Matteo Rizzi.
\newblock Symmetry properties of some solutions to some semilinear elliptic
equations.
\newblock {\em Ann. Sc. Norm. Super. Pisa Cl. Sci. (5)}, 16(4):1209--1234,
2016.

\bibitem{FV13}
Alberto Farina and Enrico Valdinoci.
\newblock On partially and globally overdetermined problems of elliptic type.
\newblock {\em Amer. J. Math.}, 135(6):1699--1726, 2013.

\bibitem{GNN79}
Basilis Gidas, Wei-Ming Ni, and Louis Nirenberg.
\newblock Symmetry and related properties via the maximum principle.
\newblock {\em Comm. Math. Phys.}, 68(3):209--243, 1979.

\bibitem{GT83}
David Gilbarg and Neil~S. Trudinger.
\newblock {\em Elliptic Partial Differential Equations of Second Order}, volume
224 of {\em Grundlehren der mathematischen Wissenschaften [Fundamental
Principles of Mathematical Sciences]}.
\newblock Springer-Verlag, Berlin, second edition, 1983.

\bibitem{GG23}
Francesca Gladiali and Antonio Greco.
\newblock Symmetry and monotonicity results for solutions of semilinear {PDE}s
in sector-like domains.
\newblock {\em Ann. Mat. Pura Appl. (4)}, 202(1):431--461, 2023.

\bibitem{GG98}
Changfeng Gui and Nassif Ghoussoub.
\newblock Multi-peak solutions for a semilinear {N}eumann problem involving the
critical {S}obolev exponent.
\newblock {\em Math. Z.}, 229(3):443--474, 1998.

\bibitem{GM18}
Changfeng Gui and Amir Moradifam.
\newblock The sphere covering inequality and its applications.
\newblock {\em Invent. Math.}, 214(3):1169--1204, 2018.

\bibitem{GWW00}
Changfeng Gui, Juncheng Wei, and Matthias Winter.
\newblock Multiple boundary peak solutions for some singularly perturbed
{N}eumann problems.
\newblock {\em Ann. Inst. H. Poincar\'e{} C Anal. Non Lin\'eaire},
17(1):47--82, 2000.

\bibitem{HL11}
Qing Han and Fanghua Lin.
\newblock {\em Elliptic Partial Differential Equations}, volume~1 of {\em
Courant Lecture Notes in Mathematics}.
\newblock Courant Institute of Mathematical Sciences, New York; American
Mathematical Society, Providence, RI, second edition, 2011.

\bibitem{HW53}
Philip Hartman and Aurel Wintner.
\newblock On the local behavior of solutions of non-parabolic partial
differential equations.
\newblock {\em Amer. J. Math.}, 75:449--476, 1953.

\bibitem{Hat24}
Lawford Hatcher.
\newblock First mixed {L}aplace eigenfunctions with no hot spots.
\newblock {\em Proc. Amer. Math. Soc.}, 152(12):5191--5205, 2024.

\bibitem{Hat25a}
Lawford Hatcher.
\newblock The hot spots conjecture for some non-convex polygons.
\newblock {\em SIAM J. Math. Anal.}, 57(4):4588--4608, 2025.

\bibitem{Hat25d}
Lawford Hatcher.
\newblock Hot spots in domains of constant curvature.
\newblock {\em arXiv:2508.13353}, 2025.

\bibitem{Hat25b}
Lawford Hatcher.
\newblock A hot spots theorem for the mixed eigenvalue problem with small
{D}irichlet region.
\newblock {\em J. Spectr. Theory}, 15(3):1367--1382, 2025.

\bibitem{HHHO99}
Bernard Helffer, Maria Hoffmann-Ostenhof, Thomas Hoffmann-Ostenhof, and Mark~P.
Owen.
\newblock Nodal sets for groundstates of {S}chr\"{o}dinger operators with zero
magnetic field in non-simply connected domains.
\newblock {\em Comm. Math. Phys.}, 202(3):629--649, 1999.

\bibitem{HHT09}
Bernard Helffer, Thomas Hoffmann-Ostenhof, and Susanna Terracini.
\newblock Nodal domains and spectral minimal partitions.
\newblock {\em Ann. Inst. H. Poincar\'{e} C Anal. Non Lin\'{e}aire},
26(1):101--138, 2009.

\bibitem{JN00}
David Jerison and Nikolai Nadirashvili.
\newblock The ``hot spots'' conjecture for domains with two axes of symmetry.
\newblock {\em J. Amer. Math. Soc.}, 13(4):741--772, 2000.

\bibitem{LrY24}
Rui Li and Ruofei Yao.
\newblock Monotonicity of positive solutions to semilinear elliptic equations
with mixed boundary conditions in triangles.
\newblock {\em arXiv:2401.17912}, 2024.

\bibitem{LZ95}
Yanyan Li and Meijun Zhu.
\newblock Uniqueness theorems through the method of moving spheres.
\newblock {\em Duke Math. J.}, 80(2):383--417, 1995.

\bibitem{Lin01}
Chang-Shou Lin.
\newblock Locating the peaks of solutions via the maximum principle. {I}. {T}he
{N}eumann problem.
\newblock {\em Comm. Pure Appl. Math.}, 54(9):1065--1095, 2001.

\bibitem{MY26}
Jingyi Mai and Ruofei Yao.
\newblock Symmetry and monotonicity of positive solutions to elliptic equations
with mixed boundary conditions in a kite.
\newblock {\em Discrete Contin. Dyn. Syst.}, 46:305--330, 2026.

\bibitem{Ni11}
Wei-Ming Ni.
\newblock {\em The Mathematics of Diffusion}.
\newblock SIAM, 2011.

\bibitem{NT91}
Wei-Ming Ni and Izumi Takagi.
\newblock On the shape of least-energy solutions to a semilinear {N}eumann
problem.
\newblock {\em Comm. Pure Appl. Math.}, 44(7):819--851, 1991.

\bibitem{Nor21}
Samuel Nordmann.
\newblock Maximum principle and principal eigenvalue in unbounded domains under
general boundary conditions.
\newblock {\em arXiv:2102.07558}, 2021.

\bibitem{PW67}
Murray~H. Protter and Hans~F. Weinberger.
\newblock {\em Maximum Principles in Differential Equations}.
\newblock Prentice-Hall, Inc., Englewood Cliffs, NJ, 1967.

\bibitem{Ser71}
James Serrin.
\newblock A symmetry problem in potential theory.
\newblock {\em Arch. Rational Mech. Anal.}, 43:304--318, 1971.

\bibitem{SGC97}
Xiaoding Shi, Yonggeng Gu, and Jing Chen.
\newblock Symmetry and monotonicity of positive solutions of systems of
semilinear elliptic equations.
\newblock {\em Acta Math. Sci. (Chinese)}, 17(1):1--9, 1997.

\bibitem{Siu15}
Bart{\l}omiej Siudeja.
\newblock Hot spots conjecture for a class of acute triangles.
\newblock {\em Math. Z.}, 280(3-4):783--806, 2015.

\bibitem{Ter95}
Susanna Terracini.
\newblock Symmetry properties of positive solutions to some elliptic equations
with nonlinear boundary conditions.
\newblock {\em Differential Integral Equations}, 8(8):1911--1922, 1995.

\bibitem{WH07}
Weimin Wang and Li~Hong.
\newblock Positive solutions of some semilinear elliptic equations in
{$\mathbb{R}^{n}_{+}$} with {N}eumann boundary conditions.
\newblock {\em Nonlinear Anal.}, 66(4):936--949, 2007.

\bibitem{WC20}
Leyun Wu and Wenxiong Chen.
\newblock The sliding methods for the fractional {$p$}-{L}aplacian.
\newblock {\em Adv. Math.}, 361:106933, 26, 2020.

\bibitem{YCG21}
Ruofei Yao, Hongbin Chen, and Changfeng Gui.
\newblock Symmetry of positive solutions of elliptic equations with mixed
boundary conditions in a super-spherical sector.
\newblock {\em Calc. Var. Partial Differential Equations}, 60(4):Paper No. 130,
25, 2021.

\bibitem{YCL18}
Ruofei Yao, Hongbin Chen, and Yi~Li.
\newblock Symmetry and monotonicity of positive solutions of elliptic equations
with mixed boundary conditions in a super-spherical cone.
\newblock {\em Calc. Var. Partial Differential Equations}, 57(6):Paper No. 154,
28, 2018.

\bibitem{Zhu01}
Meijun Zhu.
\newblock Symmetry properties for positive solutions to some elliptic equations
in sector domains with large amplitude.
\newblock {\em J. Math. Anal. Appl.}, 261(2):733--740, 2001.

\end{thebibliography}

\end{document}